\documentclass[a4paper, 10pt]{scrartcl}
\usepackage[twoside, bindingoffset=1cm, top=3cm, bottom=3.4cm, width=14cm]{geometry}

\makeatletter
\DeclareOldFontCommand{\rm}{\normalfont\rmfamily}{\mathrm}
\DeclareOldFontCommand{\bf}{\normalfont\bfseries}{\mathbf}
\DeclareOldFontCommand{\it}{\normalfont\itshape}{\mathit}
\makeatother

\usepackage[T1]{fontenc}
\usepackage{ucs}
\usepackage[utf8x]{inputenc}

\usepackage[reqno,intlimits]{mathtools}

\usepackage{amssymb, accents, amsthm, mathrsfs, dsfont, enumitem, cases}
\usepackage{enumitem}
\newlist{conditions}{enumerate}{1}
\setlist[conditions]{label=(\alph*),ref=(\alph*)}
\makeatletter
\newcommand{\mylabel}[2]{#2\def\@currentlabel{#2}\label{#1}}
\makeatother
\usepackage{stackrel} 

\usepackage{color}

\usepackage[bookmarks, colorlinks, citecolor=blue]{hyperref}

\usepackage{graphicx, float}
\usepackage{ragged2e}
\usepackage[rightcaption]{sidecap}
\usepackage{wrapfig}
\usepackage{subcaption}
\usepackage[labelfont=bf, format=plain, font=footnotesize]{caption}

\usepackage{pdfpages, scrpage2}
\automark{chapter}

\newtheorem{thm}{Theorem}[section]
\newtheorem{cor}[thm]{Corollary}
\newtheorem{lem}[thm]{Lemma}
\newtheorem{prop}[thm]{Proposition}
\newtheorem{ass}[thm]{Assumption}
\newtheorem{defn}[thm]{Definition}

\newtheorem{remark}[thm]{Remark}

\expandafter\let\expandafter\oldproof\csname\string\proof\endcsname
\let\oldendproof\endproof
\renewenvironment{proof}[1][\proofname]{%
  \oldproof[{\it #1}]%
}{\oldendproof}

\renewcommand{\d}{\mathrm{d}}

\newcommand{\del}{\partial}
\newcommand{\supp}{\text{supp}}

\newcommand{\N}{\mathbb{N}}
\newcommand{\R}{\mathbb{R}}
\newcommand{\Z}{\mathbb{Z}}
\newcommand{\E}{\mathbb{E}}
\renewcommand{\vec}{\boldsymbol}
\newcommand{\eps}{\epsilon}
\renewcommand{\Pr}{\mathbb{P}}

\newcommand{\Pas}{$\Pr$-a.s.\ }
\newcommand{\Prp}{\Pr}
\newcommand{\Pasp}{$\Prp$-a.s.\ }
\newcommand{\edges}{\mathfrak{E}}
\newcommand{\Open}{\mathfrak{E}_{\rm O}}
\newcommand{\C}{\mathscr{C}}
\newcommand{\Cinf}{\C_{\infty}}
\newcommand{\D}{\mathscr{D}}

\let\originalleft\left
\let\originalright\right
\renewcommand{\left}{\mathopen{}\mathclose\bgroup\originalleft}
\renewcommand{\right}{\aftergroup\egroup\originalright}

\title{\large Localization of the principal Dirichlet eigenvector in the heavy-tailed random conductance model}
  \date{}
  \author{\small Franziska Flegel\\\small Weierstrass Institute Berlin}

\begin{document}
\DeclareGraphicsExtensions{.pdf}
\numberwithin{equation}{section}
\numberwithin{figure}{section}

\rohead{Principal eigenvector localization in the RCM}
\rehead{F.\ Flegel}
\lohead{F.\ Flegel}
\lehead{Principal eigenvector localization in the RCM}
\setheadsepline{.5pt}

\maketitle

\vspace{-1cm}
\begin{abstract}
  We study the asymptotic behavior of the principal eigenvector and eigenvalue of the random conductance Laplacian in a large domain of $\Z^d$ ($d\geq 2$) with zero Dirichlet condition.
  We assume that the conductances $w$ are positive i.i.d.\ random variables, which fulfill certain regularity assumptions near zero.
  If $\gamma=\sup \{ q\geq 0\colon \E [w^{-q}]<\infty \}<1/4$, then we show that for almost every environment the principal Dirichlet eigenvector asymptotically concentrates in a single site and the corresponding eigenvalue scales subdiffusively.
  The threshold $\gamma_{\rm c} = 1/4$ is sharp.
  Indeed, other recent results imply that for $\gamma>1/4$ the top of the Dirichlet spectrum homogenizes.
  Our proofs are based on a spatial extreme value analysis of the local speed measure, Borel-Cantelli arguments, the Rayleigh-Ritz formula, results from percolation theory, and path arguments.
\end{abstract}

\tableofcontents

\newpage
\section{Introduction}
In dimensions greater than one, the spectrum of the i.i.d.\ random conductance Laplacian displays a sharp transition between complete localization and complete homogenization.
This is the result of the present paper in combination with recent papers from Flegel, Heida, and Slowik \cite{FHS2017} and Neukamm, Sch{\"a}ffner, and Schl{\"o}merkemper \cite{Neukamm2016}. While the other two papers
cover spectral homogenization, we investigate the localization phase. A simple moment condition distinguishes between the two phases.

More precisely, we investigate the spectrum of the Laplacian $\mathcal{L}_{\boldsymbol{w}}$ associated with the random conductance model on the euclidean lattice $\Z^d$.
The Laplacian acts on real-valued functions $f\in \ell^2(\Z^d)$ as
\begin{align}
  (\mathcal{L}_{\boldsymbol{w}} f)(x) =\sum_{y\colon |x-y|_1 =1} w_{xy}(f(y) - f(x)) \qquad (x\in\Z^d)\, .
  \label{equ:DefGenerator}
\end{align}
We assume that the conductances $w$ are positive, independent and identically distributed (i.i.d.) random variables.
We describe the almost-sure behavior of the principal eigenvector
with zero Dirichlet conditions outside a growing centered ball $B$.
It turns out that its behavior strongly depends on the lower tails of the conductances.
To be more precise, let us denote the expectation with respect to the conductances by $\E$ and define $\gamma = \sup \{ q\geq 0\colon \E [w^{-q}] <\infty \}$. Then we show that, under some further regularity assumptions,
\begin{align}
  \gamma < 1/4 \Rightarrow \begin{cases}
                                         \text{a.s.\ complete localization of principal Dirichlet eigenvector and}\\
                                         \text{a.s.\ subdiffusive scaling of principal Dirichlet eigenvalue.}
                                        \end{cases}
   \label{equ:EquivLoc}
\end{align}
On the other hand, as a special case the results in \cite{FHS2017} and \cite[Corollary 3.4, Proposition 3.18, Lemma 3.9]{Neukamm2016} imply that
\begin{align}
  \gamma > 1/4 \Rightarrow \begin{cases}
                                         \text{a.s.\ complete homogenization of first Dirichlet eigenvectors and}\\
                                         \text{a.s.\ convergence of diffusively rescaled first Dirichlet eigenvalues.}
                                        \end{cases}
  \label{equ:EquivHom}
\end{align}
We comment on this in Section \ref{subsec:literature}.
Together, \eqref{equ:EquivLoc} and \eqref{equ:EquivHom} imply that in the i.i.d.\ random conductance model there is a dichotomy between a completely homogenized and a completely localized phase.

Moreover, it is remarkable that the critical exponent $\gamma_{\rm c} = 1/4$ coincides with the critical exponent for the validity of a local central limit theorem (LCLT) of the corresponding random walk (see e.g.\ \cite[Theorem 1.9, Remark 1.10(1)]{Boukhadra2015}):
\begin{align*}
  \gamma > 1/4 \Rightarrow \text{LCLT holds} \qquad\text{and}\qquad
  \gamma < 1/4 \Rightarrow \text{LCLT does not hold.}
\end{align*}
The validity of a local CLT is a very strong kind of heat-kernel homogenization.
But in contrast to the principal Dirichlet eigenvector, the heat kernel does not display such a completely different behavior for $\gamma<1/4$. Although the heat kernel decays anomalously for $\gamma$ small enough \cite{Fontes2006, Berger2008, Boukhadra2010, Biskup2012}, a quenched functional CLT (QFCLT) still holds under minimal assumptions on the i.i.d.\ environment \cite{Andres2012}.
This is indeed \emph{not} a contradiction since the QFCLT associates with macroscopic properties of the random walk, whereas the anomalous heat-kernel bounds as well as the local CLT and the principal Dirichlet eigenvector are all sensitive to microscopic trapping structures.

Note that we do not generalize our results to higher order eigenvectors here since in Lemma \ref{lem:uniquemax} we rely on the Perron-Frobenius property of the principal Dirichlet eigenvector.
In the recent paper \cite{Flegel2018} we overcome this difficulty by using the Bauer-Fike theorem.

This paper is organized as follows:
In Section \ref{subsec:model} we define the model and our main objects.
We present our main results in Section \ref{sec:results}.
In Section \ref{subsec:literature} we compare our results with former results.
Section \ref{subsec:heatkernel} contains some comments on how a subdiffusive upper bound for the principal Dirichlet eigenvalue contradicts diffusive heat-kernel upper bounds.
We survey the proofs concerning the eigenvalues in Sections \ref{subsec:proofUpper} and \ref{subsec:proofseigval} where we rely on technical results from the subsequent sections. Section \ref{sec:BC} contains Borel-Cantelli arguments, which extend results from \cite{Cox1981}, \cite{Kesten2003} and \cite{Boukhadra2015}. In Section \ref{sec:perc} we adapt some standard results on percolation theory from \cite{Berger2008}, \cite{Mathieu2004} and \cite{Boukhadra2015} to our needs. Section \ref{sec:path} contains a path argument similar to the one in \cite{Boukhadra2015}, which was used to obtain a lower bound for the spectral gap of the random walk that is killed with a rate $\lambda$ on a certain percolation cluster.
Finally, we prove the localization of the principal Dirichlet eigenvector in Section \ref{sec:ProofThmLoc}.

\subsection{Model and main objects}
\label{subsec:model}
We consider the lattice with vertex set $\mathbb{Z}^d$ ($d\geq 2$) and edge set $\edges_d = \{ \{x,y\}: x,y \in \mathbb{Z}^d, |x-y|_1 = 1\}$.
If two sites $x,y\in \mathbb{Z}^d$ are neighbors according to $\edges_d$, we also write $x\sim y$.
To each edge $e\in \edges_d$ we assign a positive random variable $w_e$.
In analogy to a $d$-dimensional resistor network, we call the random weights $w_e$ \emph{conductances}.
We take $(\Omega,\mathcal{F})=\bigl( (0,\infty)^{\edges_d}, \mathcal{B} \left( (0,\infty) \right)^{\otimes\edges_d} \bigr)$ as the underlying measurable space and assume that an environment $\boldsymbol{w} = (w_e)_{ e\in\edges_d}\in \Omega$ is a family of i.i.d.\ positive random variables with law $\Pr$. We denote the expectation with respect to $\Pr$ by $\E$.

If $e$ is the edge between the sites $x,y\in \mathbb{Z}^d$, we will also write $w_{xy}$ or $w_{x,y}$ instead of $w_e$.
Note that by definition of the edge set $\edges_d$, the edges are undirected, whence $w_{xy} = w_{yx}$.
If we want to refer to an arbitrary copy of the conductances in general, we simply write $w$, i.e., for a set $A\in\mathcal{B} \left( (0,\infty) \right)$, the expression $\Pr [w\in A]$ equals $\Pr [w_{e}\in A]$ for an arbitrary edge $e$.

We call
\begin{align}
  F\colon [0,\infty) \to [0,1]\colon u\mapsto \Pr [w\leq u]
\end{align}
the distribution function of the conductances.

Given a realization $(w_{xy})_{\{x,y\}\in\edges_d}$ of the environment,
we consider the Markov chain on $\Z^d$ with transitions rates given by the conductances $w_{xy}$.
Its generator $\mathcal{L}_{\boldsymbol{w}}$ is defined as in \eqref{equ:DefGenerator}.
This Markov chain is known as the variable-speed random walk among random conductances. 
During the last decades both physicists and mathematicians have analyzed the random conductance model extensively and many questions regarding central limit theorems and heat-kernel behavior have been answered (for reviews see \cite{Bouchaud1990} and \cite{Biskup2011review}, respectively).
The model became popular for the description of materials where the transition rates between different states are independent of the states' energy levels.
This is the case e.g.\ for the spectral transport of optical excitations among impurity ions \cite{Lyo1979}, or charge transport in the one-dimensional ionic conductor hollandite \cite{Bernasconi1979}.

Our goal is to study the behavior of the principal eigenvalue $\lambda_1^{(n)}$ and eigenvector $\psi_1^{(n)}$ of the
sign-inverted generator $-\mathcal{L}_{\boldsymbol{w}}$ in the ball
\begin{align}
  B_n := \left\{ x\in\mathbb{Z}^d\colon |x|_\infty \leq n\right\} = [-n,n]^d \cap \Z^d
\end{align}
with zero Dirichlet conditions at the boundary.
For a real-valued function $f\in \ell^2 (\Z^d)$ let us define the Dirichlet energy $\mathcal{E}^{\boldsymbol{w}} (f)$ with respect to the operator $-\mathcal{L}_{\boldsymbol{w}}$ by
\begin{align}
  \label{equ:DefDirichletEnergy}
  \mathcal{E}^{\boldsymbol{w}} (f) = \langle f, -\mathcal{L}_{\boldsymbol{w}} f \rangle\, ,
\end{align}
where $\langle \cdot, \cdot \rangle$ denotes the standard scalar product,
i.e., for functions $f_1, f_2\colon \Z^d \to \R$ we let $\langle f_1, f_2 \rangle = \sum_{x\in\Z^d} f_1 (x) f_2 (x)$.
Then, according to the Courant-Fischer theorem, the principal Dirichlet eigenvalue is given by the variational formula
\begin{align}
	\lambda_1^{(n)} = \inf \left\{ \mathcal{E}^{\boldsymbol{w}} (f)\colon f\in \ell^2 (\Z^d),\, \supp\, f \subseteq B_n,\, \| f \|_{2} = 1 \right\}\, ,
		\label{equ:Variational}
\end{align}
where we let $\supp\, f$ denote the support of the function $f$.
The function $f$ that minimizes the RHS of \eqref{equ:Variational} is the principal Dirichlet eigenfunction $\psi_1^{(n)}$, whence $\lambda_1^{(n)} = \mathcal{E}^{\boldsymbol{w}} \bigl(\psi_1^{(n)}\bigr)$.
\begin{remark}[Perron-Frobenius]
 \label{rem:PF}
  For a given box $B_n$ the operator $\mathcal{L}_{\boldsymbol{w}}$ together with the zero Dirichlet boundary conditions can be written as a $|B_n|\times |B_n|$-matrix with non-negative entries everywhere except on the diagonal.
  Since the matrix is finite-dimensional, we can add a multiple of the identity to obtain a non-negative primitive matrix without changing the matrix' spectrum.
  By the Perron-Frobenius theorem (see e.g.\ \cite[Ch.\ 1]{SenetaNN}) it follows that its principal eigenvalue is simple and we can assume without loss of generality that its principal eigenvector is positive. 
\end{remark}

In this paper we are especially interested in the behavior of the principal Dirichlet eigenvalue and eigenfunction for dimensions $d\geq 2$ and for conductances with a very heavy tail near zero. More precisely, we consider those cases where the conductances are distributed such that a local central limit theorem is not valid (cf.\ \cite[Remark 1.10]{Boukhadra2015}).
Under different circumstances, the principal Dirichlet eigenvalue and eigenfunctions were studied before:
Boivin and Depauw \cite{Boivin2003} proved that the top of the spectrum homogenizes for uniformly elliptic conductances.
Recent results from Neukamm, Sch{\"a}ffner, and Schl{\"o}merkemper \cite{Neukamm2016} and Flegel, Heida, and Slowik \cite{FHS2017} imply that the uniform ellipticity condition can be weakened to suitable moment conditions.
The one-dimensional case was thoroughly covered by Faggionato \cite{Faggionato2012}.
In Section \ref{subsec:literature} we comment on this background and how our results relate to this previous work.

\subsection{Main results}
\label{sec:results}
First we give asymptotic lower and upper bounds for the principal Dirichlet eigenvalue $\lambda_1^{(n)}$.
How can we determine whether a function $g\colon (0,\infty)\to (0,\infty)$ that decreases monotonically to zero, is such an asymptotic lower or upper bound?
To settle this, we will see that it is crucial to determine whether the box $B_n$ contains a site such that all the $2d$ incident conductances are less than or equal to $g(n)$. We call such a site a $g(n)$-trap.
A function $\Lambda_g\colon (0,\infty)\to (0,\infty)$ that carries the information about how many $g(n)$-traps we can expect in the box $B_n$, is defined by
\begin{align}
   \Lambda_g (n) = n^d \Pr \left[ w\leq g(n) \right]^{2d}\, .
   \label{equ:DefLambda}
\end{align}
Note that the factor $n^d$ scales like the number of sites in the box $B_n$ and the factor $\Pr \left[ w\leq g(n) \right]^{2d}$ relates to the probability that for a given site all the $2d$ incident edges carry a conductance less than or equal to $g(n)$.
We will see in Lemma \ref{lem:BC:ii} that if $\Lambda_g$ diverges fast enough, then \Pas for $n$ large enough the box $B_n$ contains at least one $g(n)$-trap.
On the other hand, if $\Lambda_g$ decreases fast enough to zero, then \Pas for $n$ large enough $B_n$ does not contain a $g(n)$-trap, see Lemma \ref{lem:BC:i}.

In our results we often require recurring conditions on the function $g\colon (0,\infty) \to (0,\infty)$.
Let us recall that a function $g$ varies regularly at infinity (zero) with index $\rho\in\R$ if $g(u)=u^\rho L(u)$ where $L$ is a slowly varying function at infinity (zero), i.e., for all $C>0$ we have
\begin{align*}
  \frac{L(Cu)}{L(u)}\to 1\qquad\text{as $u\to\infty$ (as $u\to 0$).}
\end{align*}
see e.g.\ \cite[Chapter 1]{BinghamRV}.

\begin{ass}\label{ass:g}
  Let $g\colon (0,\infty) \to (0,\infty)$.
  \begin{conditions}
   \item\label{ass:RVless} The function $g$ varies regularly at infinity with index less than $-2$.
   \item\label{ass:mon} The function $u\mapsto u^2g(u)$ is monotone and has a finite limit as $u$ tends to infinity.
   \item[\mylabel{ass:mondec}{(b')}] The function $u\mapsto u^2g(u)$ converges monotonically to zero as $u$ tends to infinity.
  \end{conditions}
\end{ass}

Our first theorem gives a sufficient and a necessary condition for when the function $g$ is an asymptotic upper bound for the principal Dirichlet eigenvalue $\lambda_1^{(n)}$.
Note that, given one of the Assumptions \ref{ass:g} \ref{ass:RVless} or \ref{ass:mondec} is true, then the sufficient and necessary conditions coincide up to the case where $\Lambda_g$ scales exactly like $\log\log n$.
We summarize all the conditions of the following two theorems in a graphical overview (see Figure \ref{fig:Tails}).

\begin{thm}[Upper bound]\label{thm:EigvalsUpperBound} Let $g:(0,\infty) \rightarrow (0,\infty)$ be a function that converges monotonically to zero and let $\Lambda_g$ be as in \eqref{equ:DefLambda}. Then the following statements are true:
 \begin{enumerate}[label={(\roman*)},ref={\thethm~(\roman*)}]
   \item \label{thm:EigvalsUpperBoundSufficient} If there exists $\eps>0$ such that for all $n$ large enough
   \begin{align}
   	  \frac{\Lambda_g (n)}{\log\log n} \geq 2+\eps\, ,
   	  \label{equ:UpperCond}
   \end{align}
   then \Pas for $n$ large enough $\lambda_1^{(n)} \leq 2d g\left(n \right)$.
  \item \label{thm:EigvalsUpperBoundNecessary}
   On the other hand, if
   \begin{align}
   	  \lim_{u\rightarrow\infty} \frac{\Lambda_g(u)}{\log\log u} = 0\, ,
   	  \label{equ:NotUpperCond}
   \end{align}
   and one of the Assumptions \ref{ass:g} \ref{ass:RVless} or \ref{ass:mondec} is true,
   then \Pas $\limsup_{n\rightarrow\infty} \frac{\lambda_1^{(n)}}{g(n)} = \infty$.
\end{enumerate}
\end{thm}
We prove part (i) of this theorem in Section \ref{subsec:proofUpper} and part (ii) in Section \ref{subsec:proofseigval}.
Note that in (ii) the Assumptions \ref{ass:g} \ref{ass:RVless} and \ref{ass:mondec} correspond to the fact that we can only deduce that the limit superior diverges if we assume that $g$ is in $o(n^{-2})$. This is because in the diffusive regime $\lambda_1^{(n)}$ scales like $n^{-2}$.

In the case where the distribution function $F(a)$ varies regularly at zero with index $\gamma>0$, Theorem \ref{thm:EigvalsUpperBoundSufficient} implies the following corollary.
Since its proof is immediate, we omit it.

\begin{figure}
  \begin{center}
  \includegraphics[width=\linewidth]{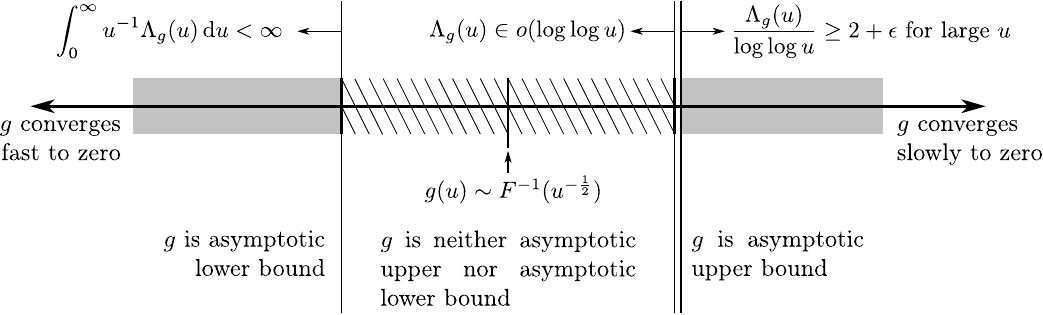}
  \caption{Visualization of our results from Theorems \ref{thm:EigvalsUpperBound} and \ref{thm:EigvalsLowerBound} for a fixed distribution function $F$ that is continuous and strictly monotone near zero.
  The figure shows the space of functions $g:(0,\infty)\to (0,\infty)$ that decrease to zero. The space is depicted such that if $f\in o(g)$, then $f$ appears left of $g$. For simplicity we assume that $g$ fulfills one of the Assumptions \ref{ass:g} \ref{ass:RVless} or \ref{ass:mondec}.
If $F^{-1} (u^{-1/2})\in o(g(u))$, then $\Lambda_g (u)$ diverges. If $g$ even decays slowly enough such that condition \eqref{equ:UpperCond} is fulfilled, then \Pas for $n$ large enough $\lambda_1^{(n)}\leq 2dg(n)$. On the other hand, if $g(u)\in o(F^{-1} (u^{-1/2}))$, then $\Lambda_g (u)$ converges to zero. If $g$ even decays fast enough such that \eqref{equ:HomCond} is fulfilled, then there exists $c>0$ such that \Pas for $n$ large enough $\lambda_1^{(n)}\geq cg(n)$.
The figure also shows that around $g(u) \sim F^{-1} (u^{-1/2})$ there is an interval where $g$ is definitely neither an a.s.\ asymptotic upper nor an a.s.\ asymptotic lower bound, see e.g.\ Corollary \ref{cor:liminfsup}.}
  \label{fig:Tails}
  \end{center}
\end{figure}

\begin{cor}
  \label{cor:UpperBoundPolylaw}
  Let $\delta>0$.
  If $F$ varies regularly at zero with index $\gamma> 0$, then \Pas for $n$ large enough the function $g(n)= n^{-\frac{1}{2\gamma}+\delta}$ is an asymptotic upper bound for $\lambda_1^{(n)}$.
  If even $F(a)=a^\gamma$ for $a\in[0,1]$, then the upper bound can be improved to $g(n)=n^{-\frac{1}{2\gamma}}\left((2+\eps)\log\log n\right)^{\frac{1}{2d\gamma}}$.
\end{cor}

Note that if $F$ varies regularly at zero with index $\gamma> 0$, then $\gamma = \sup \{ q\geq 0\colon \E [ w^{-q} ] <\infty \}$, as defined in the introduction.
Further note that if $\gamma\in [0,1/4)$, then there exists $\eta>0$ such that the expectation $\mathbb{E} \left[  w^{-1/4+\eta}\right]$ diverges, cf. the conditions of \cite[Theorem 1.13]{Andres2016}.

The second theorem gives conditions for when the function $g$ is an asymptotic lower bound of the principal Dirichlet eigenvalue $\lambda_1^{(n)}$.
Note that this theorem implies that, given one of the Assumptions \ref{ass:g} \ref{ass:RVless} or \ref{ass:mon} is true and $\Lambda_g$ is bounded, then the condition in \eqref{equ:HomCond} is sharp.
We further comment on these conditions in Section \ref{subsec:proofseigval}. As with the conditions of Theorem \ref{thm:EigvalsUpperBound}, we summarize them in the graphical overview Figure \ref{fig:Tails}.

\begin{thm}[Lower Bound]\label{thm:EigvalsLowerBound}
Let $g:(0,\infty) \rightarrow (0,\infty)$ be a decreasing function that fulfills one of the Assumptions \ref{ass:g} \ref{ass:RVless} or \ref{ass:mon}.
Let $\Lambda_g$ be as in \eqref{equ:DefLambda}.
Then the following statements are true: If
\begin{align}
  			\int_0^\infty u^{-1} \Lambda_g (u) \,\d u<\infty
  			\label{equ:HomCond}
\end{align}
and $\Lambda_g$ is bounded from above, then there exists a constant $c>0$ such that $\Pr$-a.s.\ for $n$ large enough $\lambda_1^{(n)} \geq c g(n)$.
If, on the other hand, Condition \eqref{equ:HomCond} does not hold, then \Pas $\liminf_{n\rightarrow\infty} \frac{\lambda_1^{(n)}}{g(n)} = 0$.
\end{thm}
We prove the first part of this theorem in Section \ref{subsec:proofseigval}.
The second part, i.e., where Condition \eqref{equ:HomCond} does not hold, is covered in Section \ref{subsec:proofUpper}.

Similarly as for Theorem \ref{thm:EigvalsUpperBound}, we obtain the following corollary.
As before, its proof is immediate and therefore we omit it.
\begin{cor}
  \label{cor:LowerBoundPolylaw}
  Let $\delta>0$. If $F$ varies regularly at zero with index $\gamma\in (0,1/4]$, then \Pas for $n$ large enough the function $g(n)= n^{-\frac{1}{2\gamma}-\delta}$ is an asymptotic lower bound for $\lambda_1^{(n)}$.
  If even $F(a)=a^\gamma$ for $a\in[0,1]$, then the lower bound can be further improved.
  For example $g(n)=n^{-\frac{1}{2\gamma}}\left(\log n\right)^{-\frac{1}{2d\gamma}-\delta}$ is an asymptotic lower bound in this case.
  
  Furthermore, if $F$ varies regularly at zero with index $\gamma>1/4$, then there exists $c>0$ such that $cn^{-2}$ is an asymptotic lower bound for $\lambda_1^{(n)}$. Then $\mathcal{E}^{\boldsymbol{w}} (f) \geq cn^2 \| f \|_2$ for all $f\in\ell^2 (B_n)$, which is a Poincar{\'e} inequality for functions with bounded support.
\end{cor}

  Note that for i.i.d.\ conductances with finite expectation of $w^{-1/4}$ the Poincar{\'e} inequality for functions with bounded support is also a consequence of \cite[Proposition 2.4]{Andres2016} (with $q=d/2$, $\eta$ a step function and $\nu_\omega$ replaced a $\tilde{\nu}_\omega$ which for each neighbor sums over the optimal detour from the $2d$ independent paths in Figure 2 of \cite{Andres2016}).

When we assume that $F$ is bijective near zero and set $g(u) = F^{-1} (u^{-1/2})$, then Theorems \ref{thm:EigvalsUpperBoundNecessary} and \ref{thm:EigvalsLowerBound} directly imply the following corollary.
Its proof is immediate once we have observed that $\Lambda_g (u)$ is constant in $u$.

\begin{cor}\label{cor:liminfsup}
	Assume that there exists $v>0$ such that $F\colon [0,v) \to F\left( [0,v) \right)$ is bijective and that the function $u\mapsto u^2 F^{-1} \bigl( u^{-\frac{1}{2}}\bigr)$ converges monotonically to zero.
	Then
   \begin{align}
   		\liminf_{n\rightarrow\infty} \frac{\lambda_1^{(n)}}{F^{-1} \left( n^{-\frac{1}{2}}\right)} = 0\quad\text{and}\quad
   		\limsup_{n\rightarrow\infty} \frac{\lambda_1^{(n)}}{F^{-1} \left( n^{-\frac{1}{2}}\right)} = \infty 	\qquad \Pr\text{-a.s.}
   \end{align}
\end{cor}
We comment on this behavior in Remark \ref{rem:liminfsup} in Section \ref{subsec:proofUpper}.

Note that in the special case where there exists $\gamma>0$ such that the law $\Pr$ of the conductances fulfills $\Pr\left[ w\leq a \right] = a^\gamma$ for $a\in[0,1]$, Corollary \ref{cor:liminfsup} implies that
\begin{align}
   \liminf_{n\rightarrow\infty}\, n^{\frac{1}{2\gamma}}\lambda_1^{(n)} = 0\quad\text{and}\quad
   \limsup_{n\rightarrow\infty}\, n^{\frac{1}{2\gamma}} \lambda_1^{(n)} = \infty	\qquad \Pr\text{-a.s.}
   \label{equ:liminfsupPoly}
\end{align}

Finally, we let $\psi_1^{(n)}$ be the principal Dirichlet eigenvector associated with the principal Dirichlet eigenvalue $\lambda_1^{(n)}$ and normalize it such that $\| \psi_1^{(n)} \|_2 = 1$.
By virtue of Remark \ref{rem:PF} we can assume without loss of generality that $\psi_1^{(n)}$ is nonnegative.
We show that if we define the local speed measure $\pi\colon\Z^d \to [0,\infty)$ by
\begin{align}
  \pi_x = \sum_{y\colon y\sim x} w_{xy}\, , \quad x\in \Z^d\, ,
  \label{equ:DefPi}
\end{align}
then \Pas as $n$ tends to infinity, the principal Dirichlet eigenvector $\psi_1^{(n)}$ localizes in the sequence of sites $(z_n)_{n\in\N}$ that minimize $\pi$ over $B_n$.
Since we assume that $F$ is continuous, this sequence is \Pas uniquely defined.
Further, since $F$ is continuous, for each $a\in[0,1)$ there exists $s\geq 0$ such that $F(s) = a$.
For what follows, we thus define the function $g$ as
\begin{align}
  \label{def:ThmLocFuncG}
  g:[0,\infty)\to [0,\infty)\colon u\mapsto \sup\,\left\{ s\geq 0 \colon F(s) = u^{-1/2}\right\}\, .
\end{align}

\begin{samepage}
\begin{thm}[Localization of the principal Dirichlet eigenvector]\label{thm:Loc}
  Let $F$ be continuous and vary regularly at zero with index $\gamma\in [0,1/4)$.
  Assume that there exists $a^\ast>0$ such that $F(ab) \geq bF(a)$ for all $a\leq a^\ast$ and all $0\leq b\leq 1$.
  In the case where $\gamma=0$, assume additionally that there exists $\eps_1\in (0,1)$ such that the product $n^{2+\eps_1} g \left( n \right)$ converges monotonically to zero as $n$ grows to infinity where $g$ is given by \eqref{def:ThmLocFuncG}.
  For $n\in\N$ let $z_n$ be the site that minimizes $\pi$ over $B_n$.
  Then \Pas the mass of the principal Dirichlet eigenvector $\psi_1^{(n)}$ with zero Dirichlet conditions outside the box $B_n$ increasingly concentrates in the site $z_n$, i.e.,
  \begin{align*}
    \psi_1^{(n)} (z_n) \to 1\quad\Pr\text{-a.s. as }n\to\infty\, .
  \end{align*}
  More precisely, \Pas for $n$ large enough
  \begin{align}
  \psi_1^{(n)} (z_n)^2 \geq  1 - n^{-\eps_1/4}\, ,
  \label{equ:Psi1zn}
  \end{align}
  where for $\gamma>0$ the value of $\eps_1\in (0,1)$ is chosen such that $1/(2\gamma) > 2+\eps_1$.
\end{thm}
We prove this theorem in Section \ref{subsec:ProofLoc}.
\end{samepage}
\begin{remark}[Dimension one]
  Note that we cannot expect that a result like Theorem \ref{thm:Loc} holds in dimension one.
  This is because in dimension one, the probabilistic cost to generate a hardly reachable area is independent of the area's diameter.
\end{remark}
As a consequence of Theorem \ref{thm:Loc} we have the following two corollaries, which we prove in Section \ref{subsec:ProofCorLoc}.
\begin{cor}
  \label{cor:Loc}
  Assume that $F$ fulfills the conditions of Theorem \ref{thm:Loc}.
  Then the principal Dirichlet eigenvalue $\lambda_1^{(n)}$ \Pas behaves like $\min_{x\in B_n} \pi_x$ for large $n$, i.e.,
  \begin{align}
    \mathbb{P}  \left[ \lim_{n\to\infty} \frac{\lambda_1^{(n)}}{\min_{x\in B_n} \pi_x} = 1 \right] = 1\, .
    \label{equ:PrincipalEigValMinPi}
  \end{align}
\end{cor}

Let $F_\pi$ be the distribution function of the random variable $\pi$, i.e., the distribution function of the sum of $2d$ independent copies of the conductance $w$.
Note that since $F$ is continuous, $F_\pi$ is continuous as well.
Similar to \cite[p.\ 7]{Faggionato2012}, we define
\begin{align}
\label{equ:Defh}
 h\colon (0,\infty) \to (0,\infty)\colon u \mapsto \inf\,\left\{ s\colon \frac{1}{F_\pi (1/s)} = u \right\}\, .
\end{align}
If $F$ varies regularly at zero with index $\gamma\geq 0$, then by virtue of Lemma \ref{lem:DistrPi}, it follows that $F_\pi$ varies regularly at zero with index $2d\gamma$.
It thus follows by virtue of \cite[Proposition 0.8(v)]{Resnick1987} that $h$ varies regularly at infinity with index $1/(2d\gamma)$.
Therefore there exists a function $L^\ast$ that varies slowly at infinity such that
\begin{align}
\label{def:Last}
 h(|B_n|) = n^{\frac{1}{2\gamma}} L^\ast (n)\, .
\end{align}
 
\begin{cor}
\label{cor:ConvInLaw}
   Assume that $F$ fulfills the conditions of Theorem \ref{thm:Loc} with $\gamma>0$ and let $L^\ast$ be as in \eqref{def:Last}.
   Then as $n$ tends to infinity, the product $L^\ast (n) n^{\frac{1}{2\gamma}} \lambda_1^{(n)}$ converges in distribution to a non-degenerate random variable.
   More precisely,
   \begin{align}
      \lim_{n\to \infty}\Pr \left[ L^\ast (n) n^{\frac{1}{2\gamma}} \lambda_1^{(n)} \leq \zeta\right] = 1 - \exp\left( -\zeta^{2d\gamma} \right)\qquad\text{for all } \zeta\in [0,\infty)\, .
   \end{align}
\end{cor}
We prove this corollary in Section \ref{subsec:ProofCorLoc}.

\begin{remark}[Constant speed]
  If the conductances are bounded from above,
  we conjecture that, qualitatively, the above results should also hold for the constant-speed random conductance model, i.e., where the Laplacian is given by
  \begin{align*}
    (\mathcal{L}_{\boldsymbol{w}} f)(x) = \pi_x^{-1} \sum_{y\colon |x-y|_1 =1} w_{xy}(f(y) - f(x)) \qquad (x\in\Z^d, f\in\ell^2 (\Z^d))\, .
  \end{align*}
  In this case, the critical exponent $\gamma_{\rm c}^\pi$ should be $\frac{1}{8}\frac{d}{d-1/2}$ (cf.\ \cite[Theorem 1.8 (1)]{Boukhadra2015}). Further, the typical trapping structures are not single sites but pairs of sites (cf.\ \cite[Figure 1]{Andres2016}). In a similar way as we adapt the proof techniques of \cite{Boukhadra2015} for the variable-speed case, this should be possible for the constant-speed model. However, the proofs become much more technical.
\end{remark}

\subsection{Comparison with former results}
\label{subsec:literature}
Our investigation on the spectral behavior of the random conductance generator supplements the results of former research.

For the lower bound on the principal Dirichlet eigenvalue $\lambda_1^{(n)}$, we adapt a path argument from \cite[Lemma 5.1]{Boukhadra2015}, see Section \ref{sec:path}.
In that paper as well as in \cite[Lemma 3.4]{Boukhadra2010}, the authors find a lower bound for the spectral gap $\lambda_1^{(n)}$ of the random walk that is additionally killed with a rate $\lambda=\lambda(n)$ on a certain percolation cluster $\C^\xi$.
Because of the massive killing, the spectral gap decays quite slowly with $n$.
Here, on the other hand, we deal with a scenario where the spectral gap decays very rapidly to zero.

When the spectral gap decays as $n^{-2}$ without introducing additional killing inside the box $B_n$, we are in the regime of spectral homogenization.
Boivin and Depauw \cite{Boivin2003} proved spectral homogenization for stationary and ergodic conductances that fulfill the uniform ellipticity condition, i.e., where there exist positive and finite constants $a,b$ that uniformly bound the conductances from above and below. 
As a special case they show that for uniformly elliptic i.i.d.\ conductances there exists a constant $c>0$ such that if $\lambda_1^{(n)}<\lambda_2^{(n)}\leq\ldots\leq\lambda_k^{(n)}$ are the first $k$ Dirichlet eigenvalues of $-\mathcal{L}_{\boldsymbol{w}}$, then for almost every realization of the conductance landscape
\begin{align*}
  \lim_{n\to\infty} n^2 \lambda_k^{(n)} = c\lambda_k\, ,
\end{align*}
where $\lambda_k$ is the $k$th eigenvalue of the operator $-\Delta$ in $(-1,1)^d$ with zero Dirichlet conditions.
Additionally, the principal Dirichlet eigenfunction of $-\mathcal{L}_{\boldsymbol{w}}$ converges, properly rescaled, to the principal Dirichlet eigenfunction of the operator $-\Delta$ in $(-1,1)^d$ with zero Dirichlet conditions \cite[Theorem 1, Corollary 1]{Boivin2003}.

In the special case of dimension one, the uniform ellipticity condition was already weakened by Faggionato \cite{Faggionato2012}:
She showed that for $d=1$ a finite inverse moment of $w$ is sufficient for spectral homogenization \cite[Proposition 2.6]{Faggionato2012}.
Further, if the inverse conductances $w^{-1}$ are i.i.d.\ and in the domain of attraction of an $\alpha$-stable law with $0<\alpha<1$, then Faggionato showed that the vector of the first $k$ Dirichlet eigenvalues rescaled by $n^{1+1/\alpha}$ times a slowly varying function converges in distribution to the vector of the first $k$ Dirichlet eigenvalues of a random generalized differential operator \cite[Theorem 2.5]{Faggionato2012}.

Recent results from two teams of authors imply that the uniform ellipticity condition can also be weakened in higher dimensions:
Neukamm, Sch\"affner, and Schl\"omerkemper \cite[Corollary 3.4, Proposition 3.18]{Neukamm2016} proved amongst other results that for $\gamma>1/4$ the Dirichlet energy of $-\mathcal{L}_{\boldsymbol{w}}$ (as in \eqref{equ:DefDirichletEnergy}) $\Gamma$-converges to a deterministic, homogeneous integral. This together with their compactness result \cite[Lemma 3.9]{Neukamm2016} and \cite[Theorem 13.5]{DalMaso1993} implies that Conditions I--IV of \cite[Chapter 11]{JKO1994} are fulfilled and spectral convergence follows.
On the other hand, Flegel, Heida, and Slowik \cite{FHS2017} use the method of stochastic two-scale convergence by Zhikov and Pyatniskii \cite{Zhikov2006} to show that the Poisson equation homogenizes. Their approach is similar to the one of Faggionato \cite{Faggionato2008} who already employed two-scale convergence in order to show homogenization for a Laplacian with shifted spectrum and bounded conductances.
From the homogenization of the Poisson equation, the spectral homogenization follows again by \cite[Chapter 11]{JKO1994}.
Furthermore, Flegel, Heida, and Slowik identify the corresponding limit operator.

The basis for both \cite{Neukamm2016} and \cite{FHS2017} are Poincar{\'e} and Sobolev inequalities that were already used by Andres, Deuschel, and Slowik \cite{Andres2016} to prove a quenched local CLT under suitable moment conditions.

\subsection{Survey on proofs for upper bounds}
\label{subsec:proofUpper}
Let us consider the variational formula \eqref{equ:Variational}. The equation implies that for any real-valued test function $f\in\ell^2 (\Z^d)$ with $\supp\, f\subseteq B_n$ and $\| f \|_2 = 1$ we can estimate
\begin{align*}
  \lambda_1^{(n)} \leq \langle f, -\mathcal{L}_{\boldsymbol{w}} f \rangle = \frac{1}{2}\sum_{x\in\Z^d}\sum_{y\colon y\sim x} w_{xy} (f(x) - f(y))^2\, .
\end{align*}
Suppose that $z_n$ is a random site that minimizes $\pi$ (see \eqref{equ:DefPi}) in $B_n$.
Now we choose the function $f$ such that its whole mass is concentrated in the site $z_n\in B_n$, i.e., $f=\delta_{z_n}$. When we insert this into the variational formula \eqref{equ:Variational}, then we obtain that 
\begin{align}
  \lambda_1^{(n)} \leq \min_{x\in B_n} \pi_x \leq 2d \min_{x\in B_n}\, \max_{y\colon x\sim y}w_{xy}\, .
  \label{equ:TrivialUpperBound}
\end{align}
It remains to find conditions under which the above RHS can be bounded from above by a decreasing function $g(n)$. As we have already mentioned before, a quantity which carries this information, is the function $\Lambda_g$ defined in \eqref{equ:DefLambda}, as we see in the two following proofs.

\begin{proof}[Proof of Theorem \ref{thm:EigvalsUpperBoundSufficient}]
  Condition \eqref{equ:UpperCond} together with Lemma \ref{lem:BC:ii} implies that \Pas for $n$ large enough there exists a site $z_n \in B_n$ such that $\max_{y\colon y\sim z_n} w_{z_n y} \leq g(n)$. Choose the test function $f_n = \delta_{z_n}$ and insert it into the variational formula \eqref{equ:Variational}. The claim follows.
\end{proof}

\begin{proof}[Proof of Theorem \ref{thm:EigvalsLowerBound} if Condition \eqref{equ:HomCond} fails]
  Now we have $c\int_0^\infty u^{-1} \Lambda_g (u)\, \d u = \infty$ for any $c>0$.
Let $\mathfrak{N} = \left\{ e\in\edges_d \colon 0\in e \right\}$ be the set of edges incident to the origin and note that $|\mathfrak{N}|=2d$.
A substitution of variables and Lemma \ref{lem:BC:i} imply that for any $c>0$ the following event occurs \Pas infinitely often as $n\rightarrow\infty$: There exists a site $z_n\in B_{n+1}$ such that all edges in $\mathfrak{N}\circ\tau_{z_n}$ have conductance smaller than or equal to $g(cn)$.
Here, $\tau_{z}$ ($z\in\Z^d$) denotes the spatial shift operator.

Every time this event occurs, we choose the test function $f_n = \delta_{z_n}$ (as in the proof of Theorem \ref{thm:EigvalsUpperBoundSufficient}), insert it into the variational formula \eqref{equ:Variational} and immediately obtain that \Pas
\begin{align*}
	\liminf_{n\rightarrow\infty} \frac{\lambda_1^{(n)}}{g(cn)} \leq 2d \quad\text{for any $c>0$}\, .
\end{align*}
We now show that this implies the claim.
Let $c>1$ and recall that we have assumed that one of the Assumptions \ref{ass:g} \ref{ass:RVless} or \ref{ass:mon} is true. 
In any case it follows that eventually $2 g(n) \geq c^2 g(cn)$.
It follows that \Pas $\liminf_{n\rightarrow\infty} \frac{\lambda_1^{(n)}}{g(n)} \leq 4dc^{-2}$.
This holds for any $c>1$, implying that \Pas $\liminf_{n\rightarrow\infty} \lambda_1^{(n)}/g(n) = 0$.
\end{proof}

\begin{remark}
  \label{rem:liminfsup}
  Now we can intuitively understand the result of Corollary \ref{cor:liminfsup}:
  For the choice $g(u) = F^{-1} (u^{-1/2})$, the function $\Lambda_g$ is constant one. Therefore for every $c>0$ \Pas there exists an infinite subsequence $n_k$ where the box $B_{n_k}$ contains a $(cg(n_k))$-trap. However, as we will see in Section \ref{subsec:proofseigval}, \Pas there also exists an infinite subsequence $n'_k$ where the box $B_{n'_k}$ does not contain a sufficiently good trap. It follows that the asymptotics of $\lambda_1^{(n)}$ fluctuate around the asymptotics of $F^{-1} (u^{-1/2})$.
\end{remark}

\subsection{Relation to heat-kernel upper bounds}
\label{subsec:heatkernel}
In this section we explain why the subdiffusive scaling of the principal Dirichlet eigenvalue contradicts the validity of a local central limit theorem.
We say that $\lambda_1^{(n)}$ scales subdiffusively, if $\lambda_1^{(n)} \in o(n^{-2})$.
In contrast, if the principal Dirichlet eigenvalue $\lambda_1^{(n)}$ scales like $n^{-2}$, then we call this diffusive scaling.

The local CLT was established in 2015 by the two teams of authors Andres, Deuschel, Slowik \cite{Andres2016} and Boukhadra, Kumagai, Mathieu \cite{Boukhadra2015} who require that for i.i.d.\ conductances  there exists $\eps>0$ such that $\E [w]$ and $\E [w^{-1/4-\eps}]$ are finite.
Let us briefly comment on this.

First we define the random walk among random conductances. This is the Markov chain $X_t$ that is generated by the operator $\mathcal{L}_{\boldsymbol{w}}$. Its behavior is as follows:
When the walker is at a site $x\in \mathbb{Z}^d$ it waits for an exponential time with expectation $\pi_x^{-1}$ (see \eqref{equ:DefPi}) and then jumps to one of the neighboring sites. This is why we call $\pi_x$ the local speed of the random walk at the site $x$. To which neighbor the random walker jumps is random with probabilities proportional to the corresponding conductances: If $z$ is a specific neighbor of $x$, the random walker jumps to $z$ with probability $w_{xz}/\pi_x$.
We call $P_x^{\boldsymbol{w}}$ the probability w.r.t.\ to the random walk where the superscript $\vec{w}$ refers to a fixed environment (quenched probability) and the subscript $x$ refers to the starting point of the random walker: $P_x^{\boldsymbol{w}} \left[ X_0 = x \right] = 1$. The corresponding expectation is called $E_x^{\boldsymbol{w}}$.

Let $\tau_{A}$ be the escape time from a set $A\subset \Z^d$, i.e., $\tau_A = \inf\left\{ t\geq 0\colon X_t \notin A \right\}$.
There exists a natural relation between the principal Dirichlet eigenvalue of the operator $-\mathcal{L}_{\boldsymbol w}$ and the expected escape time $E_x^{\boldsymbol{w}} \left[ \tau_{B_n}\right]$ from the box $B_n$, see \cite[Section 8.4.1]{BovierHollander}:
\begin{align*}
  \lambda_1^{(n)} \geq \left( \max_{z\in B_n} E_z^{\boldsymbol{w}} \bigl[\tau_{B_n}\bigr] \right)^{-1}\, .
\end{align*}
Thus, an upper bound for the principal Dirichlet eigenvalue $\lambda_1^{(n)}$ implies a lower bound on the maximal expected escape time from the box $B_n$.

Now Lemma 2.1(i) of \cite{Boukhadra2015} implies that if the heat kernel
\begin{align*}
  p_t (x,y) := P_x^{\boldsymbol{w}}\left[ X_t = y \right]\qquad (x,y\in\Z^d\,,\; t\geq 0)\, , 
\end{align*}
has a diffusive on-diagonal upper bound, i.e. there exists $c\in(0,\infty)$ and a random $n_0\in\N$ such that
\begin{align*}
  p_{n^2} (x,y) \leq c n^{-d} \qquad\forall x,y\in B_n\, ,~n\geq n_0\, ,
\end{align*}
then $\max_{z\in B_n} E_z^{\boldsymbol{w}} \bigl[\tau_{B_n}\bigr]\sim n^{-2}$. Diffusive heat-kernel upper bounds are a necessary condition for the validity of a local CLT.

But if we assume that the principal Dirichlet eigenvalue scales subdiffusively, i.e., $\lambda_1^{(n)} \in o(n^{-2})$, then $\max_{z\in B_n} E_z^{\boldsymbol{w}} \bigl[\tau_{B_n}\bigr]$ grows faster than $n^{-2}$ and therefore a subdiffusively scaling principal Dirichlet eigenvalue contradicts the validity of a local CLT.

We can explain the exploding escape times by showing that a large box contains some sites where the expected time to even leave the initial position is anomalously long.
Although this effect is related to the one responsible for the anomalous heat-kernel decay observed in \cite{Berger2008}, it is still a different one.
In \cite{Berger2008}, the dominating effect is that a random walk finds a trap elsewhere and then returns to its initial position. This behavior, however, has a more complex dependence on the Laplacian's eigenvalues.

\subsection{Survey on proofs for lower bounds}
\label{subsec:proofseigval}

For the lower bound of the principal Dirichlet eigenvalue we have to put in significantly more work than for the upper bound.
The key idea, however, is linked to the considerations for the shape theorem of first-passage percolation, see e.g.\ Cox and Durrett \cite{Cox1981} for the sample case $d=2$. The philosophy is that we have to show that each site in the box $B_n$ is sufficiently well reachable by conductances that are significantly greater than the lower bound candidate $g$.
Note that this is similar to the idea of Lemma 4.6 in \cite{Boukhadra2015} where the authors proved this for a polynomial tail of the conductances with parameter $\gamma$ and the candidate $g(n) = n^{-\alpha}$ with $\alpha>1/(2\gamma)$.
It turns out that a crucial element of the proof is to give a condition that implies that \Pas for $n$ large enough all sites in the box $B_n$ have at least one edge with conductance greater than $g(n)$, similar to \cite[p.\ 585]{Cox1981} and \cite[Theorem (1.7)]{Kesten1984}.
In general, if $g\colon \R_+ \to \R_+$ is monotonically decreasing and bijective, $w_1\ldots, w_{2d}$ are $2d$ independent copies of the conductance $w$, and
\begin{align}
   \mathbb{E}\left[ g^{-1}\left( \max \left\{ w_1,\ldots, w_{2d} \right\} \right)^d \right] = d\int_0^\infty u^{-1} \Lambda_g (u)\d u <\infty\, ,
   \label{equ:HomCondHeur}
\end{align}
then \Pas for $n$ large enough, all sites in the box $B_n$ have at least one incident edge with conductance greater than $g(n)$.
This together with a path argument, which we adapt from \cite{Boukhadra2015} gives the \Pas lower bound for the principal Dirichlet eigenvalue $\lambda_1^{(n)}$ (given that $g(n)$ is not asymptotically larger than $n^{-2}$).
On the other hand, if Condition \eqref{equ:HomCondHeur} is violated, then the same arguments as in Cox and Durrett \cite[p.\ 585]{Cox1981} yield that \Pas as $n$ tends to infinity, the box $B_n$ contains a $g(n)$-trap infinitely often.
We have already dealt with this case at the end of Section \ref{subsec:proofUpper}.

In what follows we give a survey on the proofs of Theorem \ref{thm:EigvalsUpperBoundNecessary} as well as Theorem \ref{thm:EigvalsLowerBound} if Condition \eqref{equ:HomCond} (or equivalently \eqref{equ:HomCondHeur}) holds. The arguments described above are made rigorous in several auxiliary lemmas, which we present in the subsequent sections.

\begin{proof}[Proof of Theorem \ref{thm:EigvalsLowerBound} if Condition \eqref{equ:HomCond} holds.]
By virtue of Lemma \ref{lem:BC:i} with $\mathfrak{A}=\{ e\in\edges_d\colon 0\in e \}$ it follows that \Pas there exists $n_1^\ast\in\mathbb{N}$ such that for all $n\geq n_1^\ast$ all sites $z\in B_n$ have an incident edge with conductance greater than $g(n)$.
Since we assumed that $\Lambda_g$ is bounded from above, it follows by virtue of Corollary \ref{cor:BC:i3} (with $m=2d$ and $\kappa=d$) that there exists $\eps>0$ such that \Pas there exists $n_2^\ast\in\mathbb{N}$ such that for all $n\geq n_2^\ast$ and for all $z\in B_{n+3d}$ the box $B_{3d} (z)$ contains at most $3d-1$ edges with conductance less than or equal to $g(n^{1-\epsilon})$.
Now we choose $\xi$ small enough such that \Pas there exists $n^\ast_3$ such that for all $n\geq n^\ast_3$ Assumptions \ref{item:SpectralGapD:Cluster}, \ref{item:SpectralGapD:SpectralGapC} and \ref{item:SpectralGapD:InjMap} of Proposition \ref{prop:SpectralGapD} are fulfilled. 
This is possible by virtue of \eqref{equ:PropIICBox} and Lemmas \ref{lem:InjectiveMap} and \ref{lem:SpectralGapPerc}.
Then the claim follows by virtue of Proposition \ref{prop:Path} with $n_k = k + \max(n_1^\ast, n^\ast_2, n^\ast_3)$.
\end{proof}

\begin{proof}[Proof of Theorem \ref{thm:EigvalsUpperBoundNecessary}]
  Let $\overline c > 1$.
  In any case of \ref{ass:g} \ref{ass:RVless} or \ref{ass:mondec}, we observe that for $n$ large enough $\overline{c} g(n) \leq g \left( \overline{c}^{-1/2} n\right)$.
  It follows that the quotient $\Lambda_{\overline{c} g(u)}/\log\log u $ is bounded as $u$ tends to infinity.
  Thus we know the following by Corollary \ref{cor:BC:i4}:
  There exists $\eps>0$ such that \Pas for $n$ large enough, there are at most $2d$ edges in any subbox $B_{3d}(z) \subset B_{n+3d}$ with conductance smaller than or equal to $\overline{c} g(n^{1-\epsilon})$.
  This implies Assumption \ref{item:SpectralGapD:Box} of Proposition \ref{prop:SpectralGapD} with $\overline{c}g$ instead of $g$.
  Now we choose $\xi$ small enough such that \Pas for $n$ large enough Assumptions \ref{item:SpectralGapD:Cluster}, \ref{item:SpectralGapD:SpectralGapC} and \ref{item:SpectralGapD:InjMap} of Proposition \ref{prop:SpectralGapD} are fulfilled.
  This is possible by virtue of \eqref{equ:PropIICBox} and Lemmas \ref{lem:InjectiveMap} and \ref{lem:SpectralGapPerc}.

  Further, Condition \eqref{equ:NotUpperCond} together with Lemma \ref{lem:BC:iii} implies that \Pas as the box size $n$ grows to infinity, there exists a random subsequence $n' = n'(\omega)$ along which each site $z\in B_{n'}$ has at least one incident edge $e$ such that $w_e > \overline{c} g(n')$.
  It follows that we can apply Proposition \ref{prop:Path} with $\overline{c} g$ instead of $g$ and obtain that there exists $C>0$ (independent of $\overline{c}$ since we have assumed \ref{ass:g} \ref{ass:RVless} or \ref{ass:mondec}) such that along the random subsequence $n_k'$ and for $k$ large enough
\begin{align*}
	\mathcal{E}^{\vec{w}} (f) \geq  C \overline{c} g(n_k') \| f \|_2^2\qquad\text{for any $f\colon\, \Z^d\to \R$ with $\supp\, f\subseteq B_{n_k'}$.}
\end{align*}
Since this holds for any $\overline c >1$, this implies the claim.
\end{proof}

\section{Borel-Cantelli arguments}
\label{sec:BC}
In this section we always assume that the dimension $d\geq 2$ and
that the conductances are i.i.d.\ with law $\Pr$.
We further let $g\colon (0,\infty)\to(0,\infty)$ be a function that decreases monotonically to zero.
Moreover, we use the following abbreviations: For $\alpha>0$ and an edge set $\mathfrak{A}\subseteq \edges_d$ we define the event
\begin{align}
	J_{\alpha} (\mathfrak{A}) = \left\{ \exists e\in \mathfrak{A}\colon w_e > \alpha \right\}\, .
\end{align}
For a set $A\subset\mathbb{Z}^d$ we define $\edges (A)$ to be the set of edges that connect a site in $A$ with a neighbor in positive axes direction (i.e., right, above, in front, etc.), i.e.,
   \begin{align*}
      \edges (A) = \left\{ \{x,y\} \in\edges_d\colon x\in A \text{ and } \exists j\in\{1,\ldots, d\} \text{ such that }y = x + \boldsymbol{e}_j\right\}\, ,
   \end{align*}
where $\{\boldsymbol{e}_j\}$ is the canonical basis of $\mathbb{Z}^d$.
For $\mathfrak{A}\subseteq \edges_d$ we write $\mathfrak{A}\circ \tau_z$ for the translation of $\mathfrak{A}$ by $z\in\mathbb{Z}^d$, i.e., for $x,y,z\in \Z^d$ with $\{ x,y \}\in\E_d$ we define $\tau_z \{ x,y \} = \{ x+z, y+z \}$.

Further, for a sequence of events $(E_n)_{n\in \N}$ we recall the definitions $\liminf_{n\to\infty} E_n = \bigcup_{n=1}^\infty \left( \bigcap_{k=n}^\infty E_n \right)$ and 
$\limsup_{n\to\infty} E_n = \bigcap_{n=1}^\infty \left( \bigcup_{k=n}^\infty E_n \right)$.

\begin{lem}\label{lem:BC:i}
If $b\in \mathbb{N}$ and $\mathfrak{A}\subseteq\edges(B_b)$ is an edge set with $|\mathfrak{A}|=m$,
then
\begin{subnumcases}{\Pr\left[ \liminf_{n\rightarrow\infty} \bigcap_{z\in B_{n+b}} J_{g(n)} (\mathfrak{A}\circ\tau_z)\right] =}
		1\, , & if $\int_0^\infty u^{d-1} \Pr \left[ w\leq g(u) \right]^m \d u <\infty\, ,$\label{equ:BC:ia}\\
		0\, , & otherwise.
	\label{equ:BC:ib}
\end{subnumcases}
I.e., if and only if the integral $\int_0^\infty u^{d-1} \Pr \left[ w\leq g(u) \right]^m \d u$ is finite, then \Pas for $n$ large enough for all sites $z\in B_{n+b}$ the edge set $\mathfrak{A}\circ\tau_z$ contains a conductance greater than $g(n)$.
Otherwise the complement of this event occurs for infinitely many $n$.
\end{lem}

\begin{remark}
   The result of Lemma \ref{lem:BC:i} as well as the proof are generalizations of the considerations of Cox and Durrett \cite{Cox1981} and Kesten \cite[p.\ 108]{Kesten2003} (there, $m=2d$ and $g(n) = n^{-1}$).
   For the sake of completeness, we included the proofs here.
\end{remark}

\begin{proof}[Proof of Lemma \ref{lem:BC:i}] 
For \eqref{equ:BC:ia}:
We first show that
\begin{align}
   0 = 1 - \Pr\left[ \liminf_{|z|_\infty\rightarrow\infty} J_{g(|z|_\infty-b)} (\mathfrak{A}\circ\tau_z)\right]
   = \Pr\left[ \limsup_{|z|_\infty\rightarrow\infty} \left( J_{g(|z|_\infty-b)} (\mathfrak{A}\circ\tau_z) \right)^{\rm c}\right]\, .
   \label{equ:ProofBC:i1}
\end{align}
We achieve this by applying the first Borel-Cantelli lemma, i.e.,
we have to estimate
\begin{align*}
   \sum_{z\in\mathbb{Z}^d\backslash B_b} \Pr \left[ \left( J_{g(|z|_\infty-b)} (\mathfrak{A}\circ\tau_z) \right)^{\rm c} \right]
   &= \sum_{k = b+1}^\infty \sum_{|z|_\infty = k} \Pr \left[ w\leq g(|z|_\infty-b) \right]^m\\
   &\leq 2d \sum_{k=b+1}^\infty (2k+1)^{d-1} \Pr\left[ w \leq g(k-b) \right]^m\, .
\end{align*}
Since $g(\cdot)$ is monotonically decreasing, $\Pr\left[ w \leq g(\cdot) \right]$ is monotonically decreasing as well.
Further, there exists an index $k_b$ such that $k-b \geq 2^{-1} (k+1)$ for all $k\geq k_b$.
It follows that there exists $C<\infty$ such that
\begin{align*}
\sum_{z\in\mathbb{Z}^d\backslash B_b} \Pr \left[ \left( J_{g(|z|_\infty-b)} (\mathfrak{A}\circ\tau_z) \right)^{\rm c} \right]
   &\leq C + 4^d d \sum_{k=b+1}^\infty (2^{-1}(k+1))^{d-1} \Pr\left[ w \leq g\left( 2^{-1}(k+1) \right) \right]^m\, .
\end{align*}
This implies that there exists $c<\infty$ such that the LHS is bounded from above by
$C + c \int_0^\infty u^{d-1} \Pr\left[ w\leq g(u) \right]^m \d u$,
which is finite by assumption.
The claim \eqref{equ:ProofBC:i1} follows from the first Borel-Cantelli lemma.

To arrive at the claim of the lemma, we observe that \eqref{equ:ProofBC:i1} implies that \Pas there exists $n^\ast\in\mathbb{N}$
such that for all $|z|_\infty\geq n^\ast$ the set $\mathfrak{A}\circ\tau_z$ contains at least one conductance greater than $g(|z|_\infty-b)$.
If $n>n^\ast$ and $z\in B_{n+b}\backslash B_{n^\ast}$, i.e., $|z|_\infty\in (n^\ast, n+b]$, this means that $\mathfrak{A}\circ\tau_z$ contains at least one conductance greater than $g(n)$
(recall that $g$ is monotonically decreasing).
Since $n^\ast$ is finite and $g$ decreases monotonically to zero,
it also follows that there exists a finite $n'\geq n^\ast$ such that for all edges $e\in\edges (B_{n^\ast+1})$ we have $g(n')< w_e$.
Thus, $\bigcap_{z\in B_{n+b}} J_{g(n)} (\mathfrak{A}\circ\tau_z)$ is true \Pas for $n$ large enough.

For \eqref{equ:BC:ib}:
Let $\int_0^\infty u^{d-1}\Pr\left[ w\leq g(u) \right]^m \d u = \infty$.
We want to show that this implies
\begin{align}
	\Pr\left[ \liminf_{n\rightarrow\infty} \bigcap_{z\in B_{n+b}} J_{g(n)} (\mathfrak{A}\circ\tau_z)\right] = 0\, .
	\label{lem:ProofBC:i2}
\end{align}
Let us define the set $A_b = (2b+1)\mathbb{Z}^d$.
It suffices to prove the claim \eqref{lem:ProofBC:i2} for the intersection over $z\in B_{n+b} \cap A_b$,
which in turn follows by the second Borel-Cantelli lemma if
\begin{align*}
	\sum_{z\in A_b} \Pr \left[ \left( J_{g(|z|+b)} (\mathfrak{A}\circ\tau_z) \right)^{\rm c} \right] = \infty\, ,
\end{align*}
since the events $\left\{ J_{g(|z|+b)} (\mathfrak{A}\circ\tau_z)\right\}_{z\in A_b}$ are independent.
To prove that the above sum diverges, we observe that there exists a constant $C>0$ such that
\begin{align*}
   \sum_{z\in A_b} \Pr \left[ \left( J_{g(|z|+b)} (\mathfrak{A}\circ\tau_z) \right)^{\rm c} \right]
   &\geq 2d (2b+1)^d \sum_{k=1}^\infty (2k-1)^{d-1} \Pr\left[ w \leq g((2b+1)k+b) \right]^m\\
   &\geq C \int_0^\infty u^{d-1}\Pr\left[ w\leq g(u) \right]^m \d u\, .
\end{align*}
By the assumption that $\int_0^\infty u^{d-1}\Pr\left[ w\leq g(u) \right]^m \d u = \infty$, the sum diverges.
\end{proof}

\begin{cor}[of Lemma \ref{lem:BC:i}]
\label{cor:BC:i2}
	Let $b\in\N$ and $m\leq |\edges (B_b)|$. Then the following equivalence holds:
	\begin{align}
			\int_0^\infty u^{d-1} \Pr \left[ w\leq g(u) \right]^m \d u <\infty
			~\Leftrightarrow~
			\Pr\left[ \liminf_{n\rightarrow\infty} \bigcap_{\mathfrak{A}\subseteq \edges(B_b),\atop |\mathfrak{A}|\geq m} \bigcap_{z\in B_{n+b}} J_{g(n)} (\mathfrak{A}\circ\tau_z)\right] = 1\, .
			\label{equ:Hom:m}
	\end{align}
\end{cor}
\begin{proof}[Proof of Corollary \ref{cor:BC:i2}]
	For ``$\Leftarrow$'', we apply Lemma \ref{lem:BC:i} for an arbitrary $\mathfrak{A} \subseteq \edges \left( B_b \right)$ with $|\mathfrak{A}|= m$.
	For ``$\Rightarrow$'', note that
	since $B_b$ is finite, the intersection over the edge sets $\mathfrak{A}$ on the RHS of \eqref{equ:Hom:m} runs over finitely many events.
	By virtue of Lemma \ref{lem:BC:i} the claim holds for each of these events and therefore also for the finite intersection.
\end{proof}

\begin{cor}[of Corollary \ref{cor:BC:i2}]
\label{cor:BC:i3}
	Let $b,m,\kappa\in\N$ with $m < |\edges (B_b)|$.
	If $u^d \Pr \left[ w\leq g(u) \right]^m$ is bounded from above, then
\begin{align}
	\Pr\left[ \liminf_{n\rightarrow\infty} \bigcap_{\mathfrak{A}\subseteq \edges(B_b),\atop |\mathfrak{A}|\geq m+\kappa} \bigcap_{z\in B_{n+b}} J_{g(n^{1-\epsilon})} (\mathfrak{A}\circ\tau_z)\right] = 1\qquad\text{for all } \eps\in [0,\kappa(m+\kappa)^{-1})\, .
	\label{equ:Hom:m+1}
\end{align}
\end{cor}

\begin{proof}
We show that the integral $\int_0^\infty v^{d-1} \Pr \left[ w \leq g\left( v^{1-\eps}\right) \right]^{m+\kappa}\d v$ is finite and then we apply Corollary \ref{cor:BC:i2}.

The change of variable $v^{1-\eps}=u$ yields
\begin{align*}
  \int_0^\infty v^{d-1} \Pr \left[ w \leq g\left( v^{1-\eps}\right) \right]^{m+\kappa}\d v
  &= (1-\eps)^{-1}\int_0^\infty u^{d(1-\eps)^{-1}-1} \Pr \left[ w \leq g\left( u\right) \right]^{m+\kappa} \,\d u\, .
\end{align*}
Now we consider that
\begin{align*}
	u^{d(1-\eps)^{-1}-1} \Pr \left[ w \leq g(u) \right]^{m+\kappa}= u^{d\left((1-\eps)^{-1} -1 -\frac{\kappa}{m}\right)-1}\left( u^d \Pr \left[ w \leq g(u) \right]^m \right)^{1+\frac{\kappa}{m}}\, .
\end{align*}
Since both $u^d \Pr \left[ w\leq g(u) \right]^m$ and $\Pr \left[ w\leq g(u) \right]^m$ are bounded from above, we obtain that
\begin{align*}
  \int_0^\infty u^{d(1-\eps)^{-1}-1} \Pr \left[ w \leq g\left( u\right) \right]^{m+\kappa} \,\d u \leq \int_0^{u_1} u^{d(1-\eps)^{-1} -1} \,\d u  + C \int_{u_1}^\infty u^{d\left((1-\eps)^{-1} -1 -\frac{\kappa}{m}\right)-1}\, \d u < \infty
\end{align*}
for any $u_1\in (0,\infty)$ and a suitable $C<\infty$.
Since $\eps\in[0,1)$ and $d\geq 2$, the first integral on the RHS is finite. Further, since $\eps<\kappa(m+\kappa)^{-1}$, the second integral on the RHS is finite as well.
\end{proof}

For the next three results, we define $\Lambda_g$ as in \eqref{equ:DefLambda}.

\begin{cor}[of Corollary \ref{cor:BC:i2}]
\label{cor:BC:i4}
Let $b\geq 2$ and assume that $\Lambda_g (u)/\log\log u$ is bounded from above for $u$ large enough.
Then 
\begin{align}
	\Pr\left[ \liminf_{n\rightarrow\infty} \bigcap_{\mathfrak{A}\subseteq \edges (B_b), \atop |\mathfrak{A}| \geq 2d+1} \bigcap_{z\in B_{n+b}} J_{g\left( n^{1-\epsilon}\right)} \left(\mathfrak{A}\circ\tau_z \right)\right] = 1\qquad\text{for all }\eps\in \left(0,(2d+1)^{-1}\right)\, .
	\label{equ:BC:i4}
\end{align}
\end{cor}

\begin{proof}
We show that the integral in \eqref{equ:Hom:m} is finite for $m=2d+1$ and $g(u^{1-\epsilon})$ instead of $g(u)$.
The assumption on the function $\Lambda_g$ implies that there exists $C<\infty$ such that for $u$ large enough
\begin{align*}
	\frac{u^{\left( d+\frac{1}{2}\right) (1-\epsilon)} \Pr\left[ w \leq g(u^{1-\epsilon}) \right]^{2d+1}}{\left( \log\log u^{1-\epsilon}\right)^{1 + \frac{1}{2d}}} < C\quad \text{for all $\epsilon\in (0,1)$.}
\end{align*}
It follows that for $u$ large enough
\begin{align*}
	u^d \Pr\left[ w \leq g(u^{1-\epsilon}) \right]^{2d+1} \leq C u^{-\frac{1}{2}+\epsilon\left(d+ \frac{1}{2}\right)} \left( \log\log u^{1-\epsilon}\right)^{1 + \frac{1}{2d}}\, .
\end{align*}
For all $\epsilon< (2d+1)^{-1}$, this implies that the integral
$\int_0^\infty u^{d-1} \Pr\left[ w \leq g(u^{1-\epsilon}) \right]^{2d+1} \,\d u$ is finite.
The claim follows by virtue of Corollary \ref{cor:BC:i2}.
\end{proof}

For the next lemma we need the following definition:
For $i,k\in\N$ with $i\leq k$ we define $A_{i,k}$ as the set which ``has residue class $i$ modulo $k$'', i.e.,
  \begin{align}
	A_{i,k} = \left\{ z=(z_1,\ldots,z_d)\in \mathbb{Z}^d\colon z_1+\ldots+z_d \equiv i \text{ mod } k \right\}\, .
	\label{equ:DefAik}
  \end{align}
  Note that for fixed $k$, the sets $A_{i,k}$ are disjoint and eventually the sets $B_n \cap A_{i,k}$ have cardinality greater than
  $(2n)^d/k$.
For $k=2$ we especially define the even lattice as the set
\begin{align}
	A_{\rm e} = \left\{ z\in \mathbb{Z}^d\colon |z|_1 \equiv 0 \text{ mod } 2 \right\}\, .
	\label{equ:DefA0V}
\end{align}
Accordingly, the odd lattice is $A_{\rm o} = \Z^d\backslash A_{\rm e}$.

\begin{lem}\label{lem:BC:ii}
Let $k\in \N$ with $k\geq 2$.
Further, let $\mathfrak{N} = \left\{ e\in\edges_d \colon 0\in e \right\}$.
Then the following implication is true: If there exists $\eps>0$ and $n^\ast\in\N$ such that
\begin{align*}
	\frac{\Lambda_g (n)}{\log\log n} \geq k+\eps
	\text{  for all $n\geq n^\ast$, then } \Pr\left[ \limsup_{n\rightarrow\infty}  \left( \bigcap_{z\in B_{n}\cap A} J_{g(n)} (\mathfrak{N}\circ\tau_z) \right) \right] = 0\, ,
\end{align*}
for any $A\in\{ A_{1,k},\ldots, A_{k,k} \}$, i.e., \Pas for $n$ large enough there exists a site $z_n\in B_n\cap A$ that is completely surrounded by edges with conductance less than or equal to $g(n)$.
It follows that \Pas for $n$ large enough there exist $k$ distinct sites $z_{(1,n)}, \ldots, z_{(k,n)}\in B_n$ that all fulfill $\pi_{z_{(i,n)}}\leq 2dg(n)$ ($i\leq k$).
\end{lem}

\begin{proof}[Proof of Lemma \ref{lem:BC:ii}]
We first prove the claim for the subsequence $n_j = 2^j$ with $j\in\N$ and with $g(2n)$ instead of $g(n)$.
Then we show how to infer the claim along the whole sequence $n\in\N$.

For the first part, let $w_1,\ldots, w_{2d}$ be $2d$ independent copies of $w$.
Since $k\geq 2$, it follows that for any $\alpha>0$ and any fixed $i$ the events $\left\{ J_{\alpha} (\mathfrak{N}\circ\tau_z) \right\}_{z\in A_{i,k}}$ are independent and thus we can estimate
\begin{align*}
    \Pr&\left[ \bigcap_{z\in B_{n_j}\cap A_{i,k}} J_{g(2n_j)} (\mathfrak{N}\circ\tau_z) \right]
   = \Pr \left[ \max \{ w_1,\ldots, w_{2d} \} > g(2n_j) \right]^{\left| B_{n_j}\cap A_{i,k} \right|}\\
   &\qquad\qquad\leq \left( 1 - \Pr \left[ w \leq g(2n_j) \right]^{2d} \right)^{(2n_j)^d/k} \leq \exp\left( - \frac{1}{k} \left( 2n_j\right)^d \Pr \left[ w \leq g(2n_j) \right]^{2d} \right)\, .
\end{align*}
The assumption on $\Lambda_g$ implies that the RHS is summable along the sequence $n_j = 2^j$.
Thus, it follows directly by the Borel-Cantelli lemma that the statement of this lemma holds along the subsequence $n_j$ and with $g(2n_j)$ instead of $g(n_j)$.

To infer the claim of the lemma along the entire sequence, we define
\begin{align*}
	M_n := \inf_{x\in B_n\cap A} \sup_{e\in \mathfrak{N}\circ\tau_x} w_{e}\, ,
\end{align*}
where $A\in\{ A_{1,k},\ldots, A_{k,k} \}$.

Note that $M_n$ is monotonically decreasing in $n$.
By the first part of the proof we know that
\begin{align}
	\Pr \left[ \liminf_{j\rightarrow\infty }\frac{M_{n_j}}{g\left( 2n_j\right)} \leq 1 \right] = 1\, .
	\label{equ:LiminfSubseq}
\end{align}
For $n\in\mathbb{N}$ we now choose $j_n$ such that
\begin{align*}
	2^{j_n} \leq n \leq 2^{j_n + 1}\, .
\end{align*}
Since $g$ and $M_{(\,\cdot\,)}$ are both monotonically decreasing, this implies that
\begin{align*}
	M_{2^{j_n}} \geq M_n \quad\text{and}\quad g(n) \geq g \left( 2^{j_n+1}\right)\, .
\end{align*}
Thus, the claim follows by \eqref{equ:LiminfSubseq}.
\end{proof}

\begin{lem}\label{lem:BC:iii}
	Let $\mathfrak{N}$ be as in Lemma \ref{lem:BC:ii}.
	If the function $u\mapsto u g(u)$ decreases monotonically to zero or $g$ varies regularly at infinity with index less than $-1$ and in any case
		\begin{align}
			\lim_{u\rightarrow \infty} \frac{\Lambda_g (u)}{\log\log u} = 0\, , \quad\text{then}\quad
			\Pr\left[ \limsup_{n\rightarrow\infty}  \left( \bigcap_{z\in B_{n}} J_{cg(n)} (\mathfrak{N}\circ\tau_z) \right) \right] = 1 \quad\forall c>0\, .
			\label{equ:BC:iiilocal}
		\end{align}
\end{lem}

\begin{proof}
For $A\subset\mathbb{Z}^d$, a fixed $c>0$, and a fixed function $g$ let us abbreviate
\begin{align*}
	H^n_A	&= \bigcap_{z\in A} J_{cg(n)} (\mathfrak{N}\circ\tau_z)\, .
\end{align*}
Let us briefly outline the idea of the proof:
It is sufficient to show that the claim is true along the subsequence $n_j = j^j$.
First we show that
\begin{align}
	\sum_{j=1}^\infty\Pr \left[ H^{n_j}_{B_{n_j}}\right] = \infty
	\label{equ:SumHInfty}
\end{align}
which, since $H^{n_j}_{B_{n_j}}\subset H^{n_j}_{B_{n_j}\backslash B_{n_{j-1}+1}}$,
implies that $\sum_{j=1}^\infty\Pr \left[ H^{n_j}_{B_{n_j}\backslash B_{n_{j-1}+1}}\right] = \infty$.
Note that since for $i,j\in\N$ with $i\neq j$ the intersection
\begin{align*}
  \left( \bigcup_{z\in B_{n_j}\backslash B_{n_{j-1}+1}} \mathfrak{N}\circ\tau_z \right) \cap \left( \bigcup_{z\in B_{n_i}\backslash B_{n_{i-1}+1}} \mathfrak{N}\circ\tau_z \right) = \emptyset\, ,
\end{align*}
the events $\left\{ H^{n_j}_{B_{n_j}\backslash B_{n_{j-1}+1}}\right\}_{j\geq 2}$ are independent.
Thus, given \eqref{equ:SumHInfty}, we can infer by the second Borel-Cantelli lemma that
\begin{align}
	\Pr \left[ \limsup_{j\rightarrow \infty} H^{n_j}_{B_{n_j}\backslash B_{n_{j-1}+1}}\right] = 1\, .
	\label{equ:PBC:limsup}
\end{align}
Then we show that 
\begin{align}
	\Pr \left[ \liminf_{j\rightarrow \infty} H_{B_{n_{j-1}+1}}^{n_{j}}\right] = 1\, .
	\label{equ:PBC:liminf}
\end{align}
Since by definition
\begin{align*}
 H_{B_{n_j}}^{n_{j}} &= H^{n_j}_{B_{n_j}\backslash B_{n_{j-1}+1}}\cap H_{B_{n_{j-1}+1}}^{n_{j}}\, ,
\end{align*}
\eqref{equ:PBC:liminf} together with \eqref{equ:PBC:limsup} implies the claim of the lemma.

Let us start with the proof of \eqref{equ:SumHInfty}.
We note that for $A_{\rm e}$ and $A_{\rm o}$ as defined in \eqref{equ:DefA0V} the FKG-inequality implies that
\begin{align*}
	\Pr \left[ H^{n_j}_{B_{n_j}} \right] = \Pr \left[ H^{n_j}_{A_{\rm e} \cap B_{n_j}} \cap H^{n_j}_{A_{\rm o} \cap B_{n_j}}\right] \geq \Pr \left[ H^{n_j}_{A_{\rm e} \cap B_{n_j}} \right]^2\, .
\end{align*}
Then we recall that $A_{\rm e}$ was constructed such that $H^{n_j}_{A_{\rm e} \cap B_{n_j}}$ is the intersection of less than $(2n+1)^d$ i.i.d.\ subevents
$\left\{ J_{cg(n_j)} (\mathfrak{N}\circ\tau_z)\right\}_{z\in A_{\rm e} \cap B_{n_j}}$,
each with probability
\begin{align*}
	\Pr \left[ J_{cg(n_j)} (\mathfrak{N}) \right] = 1 - \Pr \left[ w\leq cg(n_j) \right]^{2d}\, .
\end{align*}
Thus for $j$ large enough, there exists $C<\infty$ such that
\begin{align}
	\Pr \left[ H^{n_j}_{B_{n_j}} \right]
		&\geq \left( 1 - \Pr \left[ w\leq cg(n_j) \right]^{2d} \right)^{2(2n_j+1)^d}\nonumber\\
		&= \left( \left( 1 - \Pr \left[ w\leq cg(n_j) \right]^{2d} \right)^{\Pr \left[ w\leq cg(n_j) \right]^{-2d}} \right)^{2(2n_j+1)^d \Pr \left[ w\leq cg(n_j) \right]^{2d}}\nonumber\\
		&\geq \exp\left( -C n_j^d \Pr \left[ w\leq cg(n_j) \right]^{2d}\right) = \exp\left( -C \Lambda_{cg} \left( n_j  \right) \right)\, .
		\label{equ:BoundPH}
\end{align}
Now we explain why the assumptions on $g$ and $\Lambda_g$ imply that the RHS of \eqref{equ:BoundPH} is not summable for any $c>0$.
If $c\leq 1$, then $\Lambda_{cg} \left( n\right) \leq \Lambda_{g} \left( n\right)$ for all $n\in\N$.
It follows that for any $\varepsilon>0$ there exists $j^\ast\in\N$ such that for all $j>j^\ast$ we have
\begin{align*}
  \Lambda_{cg} \left( n_j\right) < \varepsilon \left(\log j + \log\log j\right)< 2\varepsilon \log j\, .
\end{align*}
When we choose $\varepsilon < (2C)^{-1}$, then we see that the RHS of \eqref{equ:BoundPH} is not summable.
Let us now assume that $c>1$.
If $u\mapsto ug(u)$ decreases monotonically to zero, then $cg(n) \leq g(n/c)$ for all $n$.
If $g$ varies regularly at infinity with index less than $-1$, then for any $\tilde c>c$ and for $n$ large enough $cg(n) \leq g(n/\tilde c)$.
This implies that for $n$ large enough $\Lambda_{cg} \left( n\right) \leq {\tilde c}^d \Lambda_{g} \left( n/\tilde c\right)$.
Thus, by similar arguments as for the case $c\leq 1$, we obtain that the RHS of \eqref{equ:BoundPH} is not summable.
This concludes the argument for \eqref{equ:SumHInfty}.

Let us proceed with the proof of \eqref{equ:PBC:liminf}.
Note that for any $\epsilon>0$ we have $n_{j+1} \geq h(n_j)$ with $h(u) = u(\log u)^{1-\epsilon}$. This is because for $j$ large enough and any $\epsilon>0$ we have
\begin{align*}
  n_{j+1} = (j+1) \left(1+\frac{1}{j}\right)^j j^j \geq (j+1) n_j \geq n_j \left(\log n_j \right)^{1-\epsilon}.
\end{align*}
Thus, \eqref{equ:PBC:liminf} is a consequence of
\begin{align*}
	\Pr\left[ \liminf_{n\rightarrow\infty} \bigcap_{z\in B_{n+1}} J_{cg(h(n))} (\mathfrak{A}\circ\tau_z)\right] = 1 \quad \forall c>0\, .
\end{align*}
By virtue of Lemma \ref{lem:BC:i} we can thus verify \eqref{equ:PBC:liminf} by showing that for all $c>0$ the integral
\begin{align*}
	\int_0^\infty u^{d-1} \Pr \left[ w\leq cg(h(u)) \right]^{2d} \d u <\infty\, .
\end{align*}
To see that the integral is indeed finite, we consider the following:
There exists a constant $C<\infty$ such that
\begin{align*}
	\int_0^\infty u^{d-1} \Pr \left[ w\leq cg(h(u)) \right]^{2d} \d u
		&\leq C + \int_2^\infty u^{-1} \left( \frac{u}{h(u)}\right)^d \left( h(u)^d\, \Pr \left[ w\leq cg(h(u)) \right]^{2d} \right) \,\d u\\
		&= C + \int_2^\infty u^{-1} \left( \log u \right)^{-d (1-\epsilon)} \Lambda_{cg} (h(u)) \,\d u\, .
\end{align*}
Again, we distinguish the cases $c\leq 1$ and $c>1$.
If $c\leq 1$, then $cg\leq g$.
If $c>1$, then we observe as before that in any case we have $cg(u)\leq g(u/c)$ for $u$ large enough.
Therefore the condition on $\Lambda_g$ implies that also $\Lambda_{cg} (u)/\log\log u \to 0$ as $u$ tends to infinity.
Since $h$ diverges, it follows that there exists $u^\ast<\infty$ such that $\Lambda_{cg} (h(u))\leq \log\log h(u)$ for all $u\geq u^\ast$.
Further, since on the interval $[2,u^\ast]$ the function $\Lambda_{cg} (h(\cdot))$ is bounded, the claim follows since $\int_{u^\ast}^\infty u^{-1} \left( \log u \right)^{-d (1-\epsilon)} \log\log h(u) \,\d u$ is finite.
\end{proof}

\section{Percolation results}
\label{sec:perc}
In this section we adapt three standard percolation results that we need for the path arguments of the next section in order to establish the lower bound for the principal Dirichlet eigenvalue.

We consider the standard Bernoulli bond percolation on the graph $(\Z^d, \edges_d)$, i.e., we assume that the conductances are independent random variables with common law $\mathbb{P}$ such that an individual conductance is $1$ with probability $p$ and $0$ otherwise.
For an introduction to percolation we refer the reader to \cite{Grimmett1999}.
As in the previous section, we call $\vec{w} = (w_e)_{e\in\edges_d} \in \{ 0,1 \}^{\edges_d}$ an environment and we denote the law of the environment by $\Prp$.
If the conductance $w_e$ of an edge $e$ is equal to 1, then we call $e$ an open edge. Otherwise we call the edge $e$ closed. Given a realization $\vec{w}$ of the environment, we denote the set of open edges by $\Open\subset\edges_d$.

Consider the random graph $(\Z^d, \Open)$.
Following the terminology of Grimmet \cite {Grimmett1999}, we call the connected components of this graph \emph{open clusters} and, for $x\in\Z^d$, we write
$\C(x)$ for the open cluster that contains the site $x$.
Note that $\C (x)\subset (\Z^d, \Open)$ is a graph.
We define the clusters in this way in order to make sense of Dirichlet forms defined as in \eqref{equ:DefEperc} below.
However, when we write $|\C (x)|$, we refer to the number of sites in $\C (x)$.
Furthermore, when $\C$ is a cluster and $y$ is a site in the vertex set of $\C$, then we use the shorthand notation $y\in\C$. Similarly, if $e$ is in the edge set of $\C$, then we write $e\in\C$.

We say that a path $l=\left( x_0, \ldots, x_m \right)$ is open if and only if $\{ x_{i-1}, x_i \}\in\Open$ for all $i\in\{ 1,\ldots, m \}$.

Let $p_c(d)$ be the critical probability such that \Pas there exists an infinite open cluster $\Cinf$. This cluster is \Pas unique.
We assume that $p_c(d)< p < 1$.
Note that $\Cinf$ contains all sites $x$ that are connected to infinity through an open path \emph{as well as} all open edges that are incident to a site in $\Cinf$. We further define $\mathscr{H}$ as the complement of $\Cinf$ in $\Z^d$, i.e., we regard $\mathscr{H}$ as a set of sites.

The main object of this section is to collect results from the literature and adapt the details such that they exactly fit our needs.

\begin{lem}[{\cite[Lemma 4.2]{Boukhadra2015}}]\label{lem:PropIICBox}
  Let $\eta\in (0,1)$. Then for $p$ sufficiently close to one, there exist constants $C <\infty$ and $c>0$ such that
  \begin{align}
    \Prp \left[ |B_n \cap \Cinf| \leq \eta |B_n| \right] \leq C {\rm e}^{-c n}\quad \text{ for all $n\geq 1$.}
    \label{equ:PropIICBox}
  \end{align}
\end{lem}
The second lemma is an implication of Lemma \ref{lem:PropIICBox} above.

\begin{lem}\label{lem:InjectiveMap}
   Let $d\geq 2$ and choose $p$ such that Lemma \ref{lem:PropIICBox} holds with $\eta=\frac{1}{2}$.
   Then \Pasp for $n$ large enough there exists an injective map $\varphi_1 \colon \mathscr{H} \cap B_n \to \Cinf$ such that for any site $x\in \mathscr{H} \cap B_n$ the distance $|x-\varphi_1(x)|_1 \leq 2d(\log n)^{(d+1)}$.
\end{lem}
\begin{proof}[Proof of Lemma \ref{lem:InjectiveMap}]
  The proof of this lemma follows the lines of the first paragraph of the proof of \cite[Lemma 4.7]{Boukhadra2015} but we included the proof here for completeness.
  For $z\in\Z^d$ and $m\geq 0$, we denote $B_m (z) = \left\{ x\in\Z^d\colon |x-z|_\infty\leq m \right\}$.
  Choose the percolation parameter $p$ such that \eqref{equ:PropIICBox} is fulfilled with $\eta=\frac{1}{2}$.
  Let $m=\lfloor (\log n)^{d+1} \rfloor$ and consider the disjoint partition $\mathcal{P}_m := \left\{ B_m ((2m+1)z) \right\}_{z\in\Z^d}$ of $\Z^d$.
  Then Lemma \ref{lem:PropIICBox} implies that there exist $c,C\in(0,\infty)$ such that
  \begin{align*}
    \Prp \left[ \bigcup_{B\in\mathcal{P}_m,\atop B\cap B_n \neq \emptyset} \left\{ |B \cap \Cinf| \leq \tfrac{1}{2} |B|\right\} \right]
    &\leq \sum_{B\in\mathcal{P}_m,\atop B\cap B_n \neq \emptyset} \Prp \left[ |B \cap \Cinf| \leq \tfrac{1}{2} |B|\right]\\
    &\leq C (2n+1)^d \exp\left( -c \left(\log n\right)^{d+1} \right)\, ,
  \end{align*}
  which is summable.
  By the Borel-Cantelli lemma it follows that \Pasp for $n$ large enough
  we have $|B \cap \Cinf| > |B|/2$ in any $B\in\mathcal{P}_m$ with $B\cap B_n \neq \emptyset$.
  
  Now we construct $\varphi_1$ as follows: For $x\in\mathscr{H}\cap B_n$ choose $B\in\mathcal{P}_m$ (unique) such that $x\in B$.
  Choose $\varphi_1 (x) \in B\cap \Cinf$ in an injective way - this is possible since $|\mathscr{H}\cap B| < |\Cinf\cap B|$. The $\ell_1$-distance between $x$ and $\varphi_1 (x)$ is thus smaller than or equal to $2d(\log n)^{(d+1)}$.
\end{proof}

For $f:\mathbb{Z}^d \to \mathbb{R}$ with $\| f \|_2^2 < \infty$ we define the Dirichlet-form $\mathcal{E}_{\Cinf} (f)$:
\begin{align}
   \mathcal{E}_{\Cinf} (f) = \sum_{\{ x,y \}\in\Cinf} (f(x) - f(y))^2\, ,
   \label{equ:DefEperc}
\end{align}
as well as the norm $\| f \|_{\ell^2 (\Cinf)} = \sum_{x\in\Cinf} f^2 (x)$.

In the following lemma we give a lower bound for the principal Dirichlet eigenvalue on $B_n\cap \Cinf$.
The lemma is similar to Theorem 1.3 from \cite{Mathieu2004} with the difference that $B_n\cap \Cinf$ is in general not connected and we do not include the condition that $0\in \Cinf$.

\begin{lem}\label{lem:SpectralGapPerc}
  Let $d\geq 2$ and choose $p$ such that Lemma \ref{lem:PropIICBox} holds with $\eta>\frac{1}{2}$.
  Then there exists a (deterministic) constant $c>0$ such that \Pasp for $n$ large enough and all real-valued functions $f\in\ell^2 (\Z^d)$ with $\supp\, f\subseteq B_n$ we have
  \begin{align}
  	\left\| f\right\|_{\ell^2(\Cinf)}^2 \leq cn^2 \mathcal{E}_{\Cinf} (f)\, .
  	\label{equ:SpectralGapPerc}
  \end{align}
\end{lem}
The proof of this lemma is rather standard given the relative isoperimetric inequality from Theorem \ref{thm:Isoper} below (see e.g. \cite[p.\ 83]{SaloffCoste1997}) but since the details are slightly different, we include the proof for the convenience of the reader.
Let $A\subseteq \Cinf$ be a set of sites. We define the relative edge boundary of $A$ with respect to $\Cinf$ as the edge set
\begin{align*}
  \del_{\rm E} \left( A | \Cinf \right)
  = \left\{ \{ x,y \}\in\Cinf\colon x\in A \text{ and } y\in \Cinf \backslash A\right\}\, .
\end{align*}
Further, as in \cite{Berger2008}, given a percolation environment $\vec{w}$, we call the set $A\subseteq\Z^d$ $\vec{w}$-connected if every two sites in $A$ can be connected by a finite path that uses only open edges and runs only through sites in $A$.
Then we have the following theorem.

\begin{thm}[\cite{Berger2008}, Theorem A.1]
\label{thm:Isoper}
  For all $d \geq 2$ and $p > p_c (d)$, there are positive and finite constants $c_1 = c_1 (d, p)$ and $c_2 = c_2 (d, p)$
  and a \Pas finite random variable $n_0 = n_0 (\vec{w})$ such that for each $n \geq n_0$ and each $\vec{w}$-connected $A$ satisfying
  $A \subset \Cinf\cap B_n$ and $|A|\geq (c_1 \log n)^{d/(d-1)}$ we have
  \begin{align}
  \label{equ:Iso}
    \del_{\rm E} \left( A | \Cinf \right) \geq c_2 |A|^{(d-1)/d}\, .
  \end{align}
\end{thm}

\begin{remark}
  \label{rem:Isoper}
  Let $A\subset \Cinf\cap B_n$ and $n_0, c_1, c_2$ be as in Theorem \ref{thm:Isoper}.
  If $A$ is $\vec{w}$-connected and $|A|\geq (c_1\log n)^{d/(d-1)}$, then the relative isoperimetric inequality \eqref{equ:Iso} yields
  \begin{align*}
    \frac{\left| \del_{\rm E} \left( A|\Cinf\right) \right|}{|A|} \geq \frac{c_2}{|A|^{1/d}} \geq \frac{c_2}{3n}\, ,
  \end{align*}
  where we have used that $A\subseteq B_n$ and thus $|A|^{1/d} \leq (2n+1)^d$.
  If, on the other hand, $|A|< (c_1\log n)^{d/(d-1)}$, then eventually
  \begin{align*}
    \frac{\left| \del_{\rm E} \left( A|\Cinf\right) \right|}{|A|} \geq \frac{1}{|A|} \geq \frac{1}{(c_1\log n)^{d/(d-1)}} \geq \frac{1}{n}\, .
  \end{align*}
  It follows that there exists $c>0$ such that for $n$ large enough and all $\vec{w}$-connected $A\subset \Cinf\cap B_n$ we have
  \begin{align}
    \frac{\left| \del_{\rm E} \left( A|\Cinf\right) \right|}{|A|} \geq \frac{c}{n}\, .
    \label{equ:Iso2}
  \end{align}
  If $A$ is not $\vec{w}$-connected, then similar to the arguments in \cite[Section 3.1]{Mathieu2004}, we write $A=\bigcup_i A_i$ where the $A_i$ are the $\vec{w}$-connected components of the set $A$.
  Thus,
  \begin{align*}
    \frac{\left| \del_{\rm E} \left( A|\Cinf\right) \right|}{|A|}
    \;=\; \frac{1}{|A|}\sum_i \frac{\left| \del_{\rm E} \left( A_i|\Cinf\right) \right|}{\left|A_i\right|}\cdot \left|A_i\right|
    \;\geq\; \frac{c}{n|A|}\sum_i \left|A_i\right|
    \;=\; \frac{c}{n}\, .
  \end{align*}
  It follows that \eqref{equ:Iso2} holds for all sets $A\subset \Cinf\cap B_n$.
\end{remark}

\begin{proof}[Proof of Lemma \ref{lem:SpectralGapPerc}]
  Let $n_0$ be as in Theorem \ref{thm:Isoper} and let $n\geq n_0$.
  Further let $f\colon\Z^d \to \R$ such that $\supp\, f \subseteq B_n$.
  We apply the mean value inequality and H\"older's inequality to obtain
  \begin{align}
    \sqrt{4d}\, \| f \|_{\ell^2 (\Cinf)} \sqrt{\mathcal{E}_{\Cinf} (f)}
      &\geq \sqrt{\sum_{\{ x,y \}\in\Cinf} \left( f(x) + f(y) \right)^2} \sqrt{\sum_{\{ x,y \}\in\Cinf} \left( f(x) - f(y) \right)^2}\nonumber\\
      &\geq \sum_{\{ x,y \}\in\Cinf} \left| f^2(x) - f^2(y) \right|\, .
      \label{equ:Hoelder}
  \end{align}
  Now we use a standard approach which is known as the co-area formula (see e.g. \cite[p. 83]{SaloffCoste1997}):
  \begin{align*}
    \sum_{\{ x,y \}\in\Cinf} \left| f^2(x) - f^2(y) \right|
    &= \sum_{x \in\Cinf} \sum_{~y\colon\{ x,y \}\in\Cinf, \atop f(x)\geq f(y)} \int_0^\infty \mathds{1}_{\left\{ f^2(x)>t\geq f^2(y) \right\}}\, \d t\, .
  \end{align*}
  If for $t\geq 0$ we define the set of sites $A_t = \left\{ x\in\Cinf \colon f^2(x)>t \right\}$, then we see that
  \begin{align*}
    \sum_{x \in\Cinf} \sum_{~y\colon\{ x,y \}\in\Cinf, \atop f(x)\geq f(y)} \mathds{1}_{\left\{ f(x)^2>t\geq f^2(y) \right\}}
    = \left| \del_{\rm E} \left( A_t|\Cinf\right) \right|\, .
  \end{align*}
  By virtue of Theorem \ref{thm:Isoper} and Remark \ref{rem:Isoper} it follows that there exists $c>0$ such that eventually
  \begin{align*}
    \sum_{\{ x,y \}\in\Cinf} \left| f^2(x) - f^2(y) \right|
    &\geq \frac{c}{n}\int_0^\infty |A_t|\, \d t = \frac{c}{n}\sum_{x\in\Cinf} f^2 (x)\, .
  \end{align*}
  Together with \eqref{equ:Hoelder} this implies that
  \begin{align*}
     \sqrt{\mathcal{E}_{\Cinf} (f)}
    &\geq \frac{c}{\sqrt{4d}\cdot n} \| f \|_{\ell^2(\Cinf)}\, .
  \end{align*}
\end{proof}

\section{Path argument}
\label{sec:path}

In this section we give the two Propositions \ref{prop:SpectralGapD} and \ref{prop:Path}, which transfer the knowledge we obtained by the Borel-Cantelli arguments in Section \ref{sec:BC} to lower bounds on Dirichlet energies.
In order to achieve this, Lemma \ref{lem:PathVSRW} generalizes and modifies the path argument in \cite[Lemma 4.7]{Boukhadra2015}.
Before we start, we give a definition which is crucial for the remaining part of the paper.
\begin{defn}\label{def:EwG}
	Let $\mathscr{G} = (V,\edges)$ be an undirected graph and $\boldsymbol{w} = (w_e)_{e\in\edges}$. For $f\colon V \to \mathbb{R}$, we define the Dirichlet energy on $\mathscr{G}$ as
	\begin{align}
		\mathcal{E}^{\boldsymbol{w}}_\mathscr{G} (f) = \frac{1}{2}\sum_{x\in V} \sum_{y\in V,\atop \left\{ x,y\right\}\in\edges} w_{xy} (f(x) - f(y))^2\, .
	\end{align}
\end{defn}
\begin{remark}
\label{rem:EnvPerc}
 For $\xi>0$ let us define $a_e = \mathds{1}_{\{ w_e\geq \xi\} }$ ($e\in\edges_d$). Let us call an edge $e$ open if and only if $a_e=1$ and let $\C$ be an open cluster in the environment $\vec{a} = (a_e)_{e\in\edges_d}$. Then, with reference to \eqref{equ:DefEperc}, we obtain that $\xi \mathcal{E}_\C (f) \leq \mathcal{E}^{\boldsymbol{w}}_\C (f)$ for all real-valued functions $f\in\ell^2 (\Z^d)$.
\end{remark}

Since we apply a similar argument for two slightly different situations (i.e., once for the proofs of Theorems \ref{thm:EigvalsUpperBound} and \ref{thm:EigvalsLowerBound}, see Proposition \ref{prop:Path}, and once for the proof of Theorem \ref{thm:Loc}, see Proposition \ref{prop:SpectralGapD}), we kept the conditions of the following lemma as general as necessary.
\begin{lem}
\label{lem:PathVSRW}
	Let $\mathscr{G} = (V,\edges)$ be a subgraph of $(\Z^d, \mathfrak{E}_d)$ and let $\C = (V_\C, \edges_\C)$ be a subgraph of $\mathscr{G}$.
	Assume that $\nu, L\in (0,\infty)$ and $B\subseteq V$ are such that the following conditions are fulfilled:
	\begin{enumerate}[label={(\roman*)}]
		\item\label{item:PoincareCond} There exists a constant $\mu>0$ such that for all $f\colon V \to \mathbb{R}$ with $\supp\, f \subseteq B$ the following inequality holds:
			\begin{align}
			  \mathcal{E}^{\boldsymbol{w}}_\C (f) \geq \mu \| f \|_{\ell^2(\C)}^2.
			  \label{equ:PoincareCond}
			\end{align}
		\item There exists an injective map $\varphi: B \backslash V_\C \to V_\C$ such that the following holds:
		From any $x\in B\backslash V_\C$ there exists a (self-avoiding) directed path $l(x,\varphi(x))$ to $\varphi(x)$ in $\mathscr{G}$ such that
		\begin{enumerate}
			\item\label{item:allbinl} all $e\in l(x,\varphi(x))$ fulfill $w_e>\nu$,
			\item\label{item:upperboundL} $|l(x,\varphi(x))| \leq L$.
		\end{enumerate}
	\end{enumerate}	
	Then for all $f\colon V \to \mathbb{R}$ with $\supp\, f \subseteq B$ the following holds:
					\begin{align}
						\mathcal{E}^{\boldsymbol{w}}_\mathscr{G} (f) \geq \left( (2 L)^{d+1}\nu^{-1} + 3\mu^{-1}\right) ^{-1} \| f \|_{\ell^2 (\mathscr{G})}^2\, .
					\end{align}	
\end{lem}

\begin{proof}[Proof of Lemma \ref{lem:PathVSRW}]
	We generalize the proof of \cite[Lemma 5.1]{Boukhadra2015}, which uses arguments from \cite[Lemma 3.4]{Boukhadra2010}.
	Let $f\colon V \to \mathbb{R}$ with $\supp\, f \subseteq B$. For the following calculation we abbreviate
	$f(y) - f(z) = \d f ((y,z))$ where $(y,z)$ is the (directed) edge from site $y$ to its neighbor $z$. For $x\in B\backslash V_\C$ we write $f(x)$ as a telescopic sum
	\begin{align*}
		f(x) = \sum_{b\in l(x,\varphi(x))} \d f(b) + f(\varphi(x))\, .
	\end{align*}
	We apply the Cauchy-Schwarz inequality and expand the terms on the RHS by the conductances:
	\begin{align*}
		f^2(x)	&\leq \frac{2|l(x,\varphi(x))|}{\nu}\sum_{b\in l(x,\varphi(x))} w_b\, (\d f(b))^2 + 2f^2(\varphi(x))\, .
	\end{align*}
	Now we sum over all $x\in B\backslash V_\C$ and use the upper bound for $|l(x,\varphi(x))|$ according to Condition  \ref{item:upperboundL}:
	\begin{align}
		\sum_{x\in B\backslash V_\C} f^2(x) &\leq \frac{2L}{\nu}\sum_{x\in B\backslash V_\C} \sum_{b\in l(x,\varphi(x))} w_b\, (\d f(b))^2 + 2\sum_{x\in B\backslash V_\C}f^2(\varphi(x))\, .
			\label{equ:SumOverAllx}
	\end{align}
	Let us look at the last term on the RHS: By definition $\varphi$ is injective and its image is in $V_\C$. This means that
	\begin{align*}
	   \sum_{x\in B\backslash V_\C}f^2(\varphi(x)) \leq \sum_{x\in V_\C} f^2 (x)\, .
	\end{align*}
	Since the path $l(x,\varphi(x))$ has a length of at most $L$, any path that uses a given edge $b$ must have started in an $\ell_1$-ball of radius $L$ around $b=:\{ b_1, b_2 \}$ with $b_1,b_2\in\Z^d$.
	Thus, if the path $l(x,\varphi(x))$ runs through the edge $b$, then
	\begin{align*}
	  x \in \left\{ z\in\Z^d \colon \| z-b_1 \|_1 \leq L-1 \right\} \cup \left\{ z\in\Z^d \colon \| z-b_2 \|_1 \leq L-1 \right\}\, .
	\end{align*}
	Since in dimension $d\geq 2$ and for $L\geq 2$, the cardinality of either one of the above $\ell_1$-balls is bounded from above\footnote{In dimension $d=2$, the cardinality of an $\ell^1$-ball with radius $R$ is $1 + 2R(R+1) < 2(R+1)^2$.
	If $V_d^{(1)} (R)$ is the cardinality of an $\ell^1$-ball with radius $R$ in dimension $d$, then one convinces oneself that
	$V_d^{(1)} (R) < 2(R+1) V_{d-1}^{(1)} (R)$.} by $2^{d-1} L^d$, it follows that the cardinality of the whole set on the above RHS is bounded from above by $(2L)^d$.
	Thus, the sum over $b\in l(x,\varphi(x))$ on the RHS in \eqref{equ:SumOverAllx} uses each edge not more than $(2L)^d$ times, whence
		\begin{align*}
		\sum_{x\in B\backslash V_\C} \sum_{b\in l(x,\varphi(x))} w_b\, (\d f(b))^2 \leq (2L)^d\mathcal{E}^{\vec{w}}_\mathscr{G} (f)\, .
	\end{align*}
	Completing the sum to all sites $x\in\mathscr{G}$ and using the comparability between $\mathcal{E}^{\boldsymbol{w}}_\mathscr{G} (f)$ and $\mathcal{E}^{\boldsymbol{w}}_\C (f)$, we obtain by virtue of Condition \ref{item:PoincareCond}:
	\begin{align*}
		\sum_{x\in V} f^2(x) &\leq \frac{(2 L)^{d+1}}{\nu}\mathcal{E}^{\boldsymbol{w}}_\mathscr{G} (f)+ 3\sum_{x\in V_\C}f^2(x)
		\leq \left( \frac{(2 L)^{d+1}}{\nu} + \frac{3}{\mu}\right) \mathcal{E}^{\boldsymbol{w}}_\mathscr{G} (f)\, .
	\end{align*}
\end{proof}

\subsection{Asymptotics of the principal Dirichlet eigenvalue}
\label{sec:ProofThmEigVals}
From the path argument in Lemma \ref{lem:PathVSRW} we can use our observations from Section \ref{sec:BC} to obtain lower bounds of the Dirichlet forms. We use similar arguments as in \cite[Lemma 5.1]{Boukhadra2015}. 

We fix $\xi>0$ such that
\begin{align}
	\Pr\left[ w > \xi \right]>p_c(d)\, .
	\label{equ:CondXi}
\end{align}
Moreover, we fix an environment $\boldsymbol{w}$ and define a new environment $\boldsymbol{a}$ by setting 
\begin{align}
	a_e = \mathds{1}_{\{ w_e > \xi\} }\qquad (e\in\edges_d)\, ,
	\label{equ:DefAE}
\end{align}
as in Remark \ref{rem:EnvPerc}.
We denote the unique infinite cluster of the environment $\boldsymbol{a}$ by $\C^\xi$ and we use the same shorthand notations as explained at the beginning of Section \ref{sec:perc}. Further we define
$\C^\xi_n$ as the restriction of $\C^\xi$ to the box $B_n$ and similarly the holes $\mathscr{H}_n^\xi$.

Additionally, we define a second percolation environment $\tilde{\boldsymbol{w}}_{g(n)}$ for $g:(0,\infty)\to (0,\infty)$ by setting
\begin{align}
	\tilde{w}_{g(n)} (e) = w_e \mathds{1}_{\{ w_e > g(n)\} } \qquad (e\in \edges_d)\, .
\end{align}
Thus, edges with conductance less than or equal to $g(n)$ are considered to be closed and all others keep their original conductance.
With this terminology we can now define the following clusters.
\begin{defn}\label{def:InDn}
  For a fixed function $g$ and a fixed $\eps>0$, let $\mathscr{D}_{n}$ be the unique infinite open cluster of $\tilde{\boldsymbol{w}}_{g(n^{1-\eps})}$.
  Regarding this cluster, we use the same shorthand notations as introduced at the beginning of Section \ref{sec:perc}.
  Furthermore, let $\mathscr{I}_n = B_n\backslash \D_n$ be the set of holes in $B_n$.
\end{defn}
\begin{defn}\label{def:sparse}
  We call a set $\mathscr{I}\subset\Z^d$ {\bf sparse} if the set $\mathscr{I}$ does not contain any neighboring sites.
  Further, a set $\mathscr{I}\subset\Z^d$ is {\bf $\mathbf{b}$-sparse} if any box $B_b(z)\subset \Z^d$ with $z\in \Z^d$ contains at most one site of the set $\mathscr{I}$.
\end{defn}
\begin{remark}
  \label{rem:b1b2sparse}
  Let $b_1<b_2$ be natural numbers.
  If a set $\mathscr{I}\subset\Z^d$ is $b_2$-sparse, it is also $b_1$-sparse and sparse.
\end{remark}

\begin{lem}\label{lem:InSparse}
  Let $b\in\N$ with $b\geq 2d$ and $g\colon (0,\infty) \mapsto (0,\infty)$ be a decreasing function.
  For a fixed environment $\vec{w}$ assume that for $n$ large enough for all $z\in B_{n+b}$ the edge set $\mathfrak{E}\left(B_b (z)\right)$ contains at most $3d-1$ edges
  with conductance less than or equal to $g(n^{1-\eps})$.
  Further, let $\D_n$ be as in Definition \ref{def:InDn} the unique infinite cluster.
  Then, for $n$ large enough the set $\mathscr{I}_n = B_n\backslash \D_n$ is $b$-sparse.
\end{lem}
\begin{proof}
To show that for $n$ large enough the set $\mathscr{I}_n$ is $b$-sparse, we first show that for $n$ large enough the set $\mathscr{I}_n$ is sparse.
\begin{figure}
  \begin{subfigure}{4cm}\begin{center}\includegraphics[width=3.5cm]{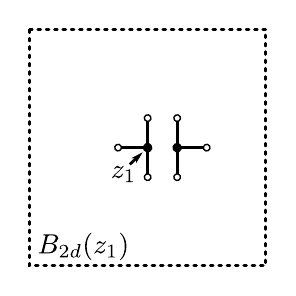}\caption{Case 1}\label{subfig:APairInIn1}\end{center}\end{subfigure}
  \begin{subfigure}{4cm}\begin{center}\includegraphics[width=3.5cm]{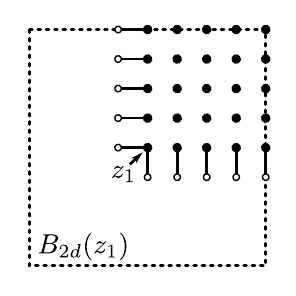}\caption{Case 2}\label{subfig:APairInIn2}\end{center}\end{subfigure}
  \begin{subfigure}{4cm}\begin{center}\includegraphics[width=3.5cm]{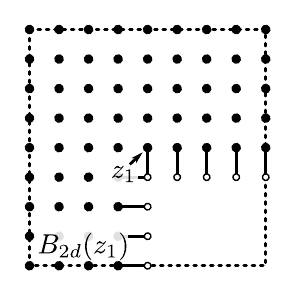}\caption{Case 2}\label{subfig:APairInIn2b}\end{center}\end{subfigure}
  \caption{Boundary edges needed to separate the set $\tilde{\mathscr{I}}_n (z_1)$ from the infinite cluster $\D_n$.
  The full circles represent sites of $\tilde{\mathscr{I}}_n$ while open circles represent sites of $\D_n$.
  In Figure \ref{subfig:APairInIn1} the component $\tilde{\mathscr{I}}_n (z_1)$ is a subset of $B_{2d}(z_1)$,
  in Figures \ref{subfig:APairInIn2} and \ref{subfig:APairInIn2b} it is not.}
  \label{fig:APairInIn}
\end{figure}
We define $\tilde{\mathscr{I}}_n = \Z^d\backslash \D_n$.
Let us assume that for infinitely many $n$ there exists a pair of neighbors $z_1, z_2$ in the set $\mathscr{I}_n = \tilde{\mathscr{I}}_n \cap B_n$.
Since by assumption $\D_n$ is the unique infinite cluster, it follows that for $n$ large enough $\D_n \cap B_n \neq \emptyset$.
Thus, we can assume without loss of generality that $z_1$ has a neighbor $x\in\D_n$.
If $z_1$ does not have a neighbor in $\D_n$, then we consider a self-avoiding path $l$ inside $B_n$ from $z_1$ to a site $x\in \D_n \cap B_n$.
Let $x'$ be the first site on the path $l$ that is in $\D_n$ and let $z_1'$ be the preceeding site to $x'$ on the path $l$.
Since $z_1$ does not have a neighbor in $\D_n$, the site $z_1'$ is different from $z_1$ and thus $z_1'$ has a further predecessor $z_2'\in\mathscr{I}_n$ on the path $l$.
It follows that the neighbors $z_1', z_2'$ are in $\mathscr{I}_n \cap B_n$ and further $z_1'$ has a neighbor $x'\in\D_n \cap B_n$. 

In the context of this proof, for $z\in\tilde{\mathscr{I}}_n$ we define $\tilde{\mathscr{I}}_n (z)\subset\tilde{\mathscr{I}}_n$ as the connected component that contains $z$, i.e., $y\in\tilde{\mathscr{I}}_n (z)$ if there exists a path $l\subset\E_d$ between the sites $z$ and $y$ that runs only through sites in $\tilde{\mathscr{I}}_n$.

Let $z_1$ be as above. We distinguish two cases now:
\begin{enumerate}
\setlength{\itemsep}{0pt}
 \item $\tilde{\mathscr{I}}_n (z_1)$ is a subset of $B_{2d} (z_1)$.
 \item $\tilde{\mathscr{I}}_n (z_1)$ is \emph{not} a subset of $B_{2d} (z_1)$.
\end{enumerate}
In the first case, the edge boundary that separates the cluster $\D_n$ and the component $\tilde{\mathscr{I}}_n (z_1)$, consists of at least $4d-2$ edges that are in the edge set $\mathfrak{E}\left(B_{2d} (z_1)\right)$.
For a sketch see Figure \ref{subfig:APairInIn1}.
This is a contradiction to the first claim of this lemma.
Since we have assumed that $\D_n$ is the infinite cluster, in the second case the edge boundary that separates the cluster $\D_n$ and the component $\tilde{\mathscr{I}}_n (z_1)$, has to link two (not necessarily different) faces of the cube $B_{2d} (z_1)$ and at the same time enclose the site $z_1$ (e.g.\ as in Figure \ref{subfig:APairInIn2}) in the very middle of the cube $B_{2d} (z_1)$ or enclose one of its neighbors (e.g.\ as in Figure \ref{subfig:APairInIn2b}).
It follows that the set $\mathfrak{E}\left(B_{2d} (z_1)\right)$ consists of more than $4d$ edges with conductance less than or equal to $g(n^{1-\eps_2})$.
This is again a contradiction and it follows that $\mathscr{I}_n$ is sparse.

Now we further show that \Pas for $n$ large enough the set $\mathscr{I}_n$ is $b$-sparse.
Let us assume that for infinitely many $n$ there exists $z\in B_{n+b}$ such that $B_b(z)$ contains two sites of $\mathscr{I}_n$.
Since we already know that \Pas for $n$ large enough the set $\mathscr{I}_n$ is sparse and each site has $2d$ incident edges, it follows that for infinitely many $n$ there exists $z\in B_{n+b}$ such that the edge set $\mathfrak{E}\left( B_{b+1}(z)\right)$ contains $4d$ edges with conductance less than or equal to $g(n^{1-\eps_2})$.
This is a contradiction to the first claim of this lemma.
\end{proof}

\begin{prop}
\label{prop:SpectralGapD}
	Fix an environment $\vec{w}\in\Omega$.
	Let $g$ be a positive function decreasing to zero and let $\epsilon, \xi, c_1 \in (0,\infty)$.
	Let $(n_k)_{k\in\mathbb{N}}$ be any (possibly empty) subsequence, along which the following assumptions are true:
	\begin{enumerate}[label={(\roman*)}]
	 \item\label{item:SpectralGapD:Box} For all $z\in B_{n_k+3d}$ the edge set $\mathfrak{E}\left(B_{3d} (z)\right)$ contains at most $3d-1$ edges with conductance less than or equal to $g(n^{1-\epsilon})$.
	 \item\label{item:SpectralGapD:Cluster} $\C^\xi$ is the unique infinite open cluster of the environment $\vec{a}$ defined through \eqref{equ:DefAE}.
	 \item\label{item:SpectralGapD:SpectralGapC} All real-valued functions $f\in\ell^2 (\Z^d)$ with $\supp\, f \subseteq B_{n}$ fulfill
	 \begin{align}
  		\| f \|_{\ell^2 (\C^\xi)}^2 \leq c_1 n_k^2 \mathcal{E}_{\C^\xi} (f) \leq \xi^{-1} c_1 n_k^2 \mathcal{E}^{\vec{w}}_{\C^\xi} (f)\, .
	 \end{align}
	 \item\label{item:SpectralGapD:InjMap} There exists an injective map $\varphi_1\colon \mathscr{H}^\xi\cap B_n \to \C^\xi$ such that for any $x\in \mathscr{H}^\xi\cap B_n$ there exists a directed path $l_1(x,\varphi_1(x))$ in $(\Z^d, \edges_d)$ from $x$ to $\varphi_1 (x)$ of length $|l_1(x,\varphi_1 (x))|\leq 2d(\log n)^{(d+1)}$.
	\end{enumerate}
	Then for $k$ large enough, $\D_{n_k}$ is the unique infinite open cluster of the environment $\tilde{\boldsymbol{w}}_{g\left(n^{1-\eps}\right)}$ (see Definition \ref{def:InDn}) and
	\begin{align*}
		\mathcal{E}_{\D_{n_k}}^{\vec{w}} \bigl( f \bigr) \geq  \left( 2^{d+1}(\log n_k)^{4d^2} g\bigl(n_k^{1-\eps}\bigr)^{-1} + 3c_1 \xi^{-1} n_k^2 \right)^{-1} \| f \|^2_{\ell^2(\D_{n_k})}\, ,
	\end{align*}
	for all real-valued functions $f\in\ell^2 (\mathbb{Z}^d)$ with $\supp\, f\subseteq B_{n_k}$ and with $\mathcal{E}_{\mathscr{D}_{n_k}}^{\vec{w}}$ as in Definition \ref{def:EwG}.
\end{prop}

\begin{prop}
\label{prop:Path}
	Let the assumptions of Proposition \ref{prop:SpectralGapD} be true for a subsequence $(n_k)_{k\in\mathbb{N}}$ as well as one of Assumptions \ref{ass:g} \ref{ass:RVless} or \ref{ass:mon}.
	Further, assume that along the same subsequence for all $z\in B_{n_k}$ there exists an incident edge with conductance greater than $g(n_k)$.
	Then there exists $c>0$ such that for $k$ large enough
	\begin{align*}
		\mathcal{E}^{\vec{w}} \bigl( f \bigr) \geq cg(n_k) \| f \|_2^2
	\end{align*}
	for all real-valued functions $f\in\ell^2(\mathbb{Z}^d)$ with $\supp\, f\subseteq B_{n_k}$.
	If one of the Assumptions \ref{ass:g} \ref{ass:RVless} or \ref{ass:mondec} is fulfilled, then the constant $c$ can be chosen independently of $g$.
\end{prop}
We prove these propositions in the next section.

\subsection{Proofs of Propositions \ref{prop:SpectralGapD} and \ref{prop:Path}}

\begin{proof}[Proof of Proposition \ref{prop:SpectralGapD}]
In this proof we shortly write $n$ for a member of the subsequence $(n_k)_{k\in\mathbb{N}}$.
\begin{figure}
  \begin{subfigure}{.32\linewidth}\begin{center}\includegraphics[width=4.3cm]{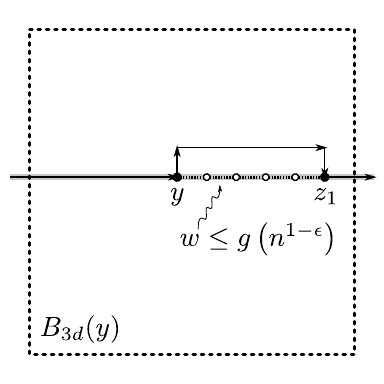}\caption{}\label{subfig:ConstrPath1}\end{center}\end{subfigure}
  \begin{subfigure}{.32\linewidth}\begin{center}\includegraphics[width=4.3cm]{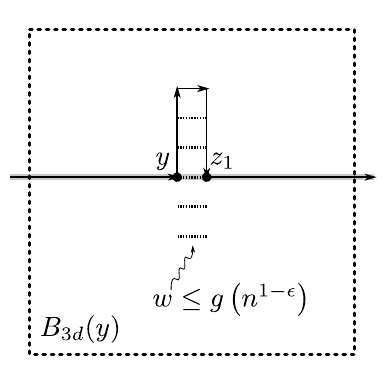}\caption{}\label{subfig:ConstrPath2}\end{center}\end{subfigure}
  \begin{subfigure}{.32\linewidth}\begin{center}\includegraphics[width=4.3cm]{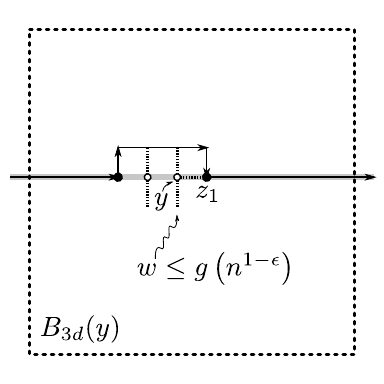}\caption{}\label{subfig:ConstrPath3}\end{center}\end{subfigure}
  \caption{Construction of the path $l(x,\varphi(x))$ (solid black line with arrows) from $l_1(x,\varphi(x))$ (thick gray line).
  Inside the box $B_{3d}(y)$ there are at most $3d-1$ bad conductances (dotted lines).
  The path $l(x,\varphi(x))$ follows $l_1(x,\varphi(x))$ until it hits an edge with a bad conductance at site $y$.
  Let $e=\left\{ z_1, z_2 \right\}$ be the first good conductance on $l_1(x,\varphi(x))$ after $y$, where the site $z_1$ comes before $z_2$ on the path  $l_1(x,\varphi(x))$.
  Then between the sites $y$ and $z_1$ the path $l(x,\varphi(x))$ takes the shortest detour from $y$ to $z_1$ without using any edge with a bad conductance.
  If, for this purpose, the path has to take a loop, as depicted in Example \ref{subfig:ConstrPath3}, then we delete the loop.}
  \label{fig:ConstrPath}
\end{figure}
The fact that for $n$ large enough there exists a unique infinite open cluster of the environment $\tilde{\boldsymbol{w}}_{g\left(n^{1-\eps}\right)}$, follows from Assumption \ref{item:SpectralGapD:Cluster} when we choose $n$ such that $g\left(n^{1-\eps}\right)\leq \xi$, i.e., when $\C^\xi \subset \mathscr{D}_n$.

For the actual claim we apply Lemma \ref{lem:PathVSRW} with $\mathscr{G}$ given by the cluster $\D_n$ and $\C$ given by $\C^\xi$.
Further, let $\nu_n = g\left(n^{1-\eps}\right)$.
Further, if we choose $\mu_n = \frac{\xi}{c_1 n^2}$ in the place of $\mu$ in \eqref{equ:PoincareCond}, then Condition \ref{item:PoincareCond} of Lemma \ref{lem:PathVSRW} is fulfilled.

We are now going to construct the map $\varphi\colon B_n\cap\D_n \cap \mathscr{H}^\xi \to \C^\xi$ and the path $l(x,\varphi(x))$.
For the next paragraph we say that a conductance is ``bad'' if it is smaller than or equal to $g(n^{1-\epsilon})$.
Let $\varphi = \varphi_1 |_{\mathscr{H}^\xi\cap B_n\cap \D_n}$ (see Assumption \ref{item:SpectralGapD:InjMap}).
By Assumption \ref{item:SpectralGapD:Box}, each subbox $B_{3d}(z)$ with $z\in B_{n+3d}$ contains at most $3d-1$ bad conductances.
We thus construct the path $l(x,\varphi(x))$ by the following algorithm:
The path $l(x,\varphi(x))$ follows $l_1(x,\varphi(x))$ until it hits an edge with a bad conductance.
Let $y$ be the last site that the path $l(x,\varphi(x))$ reached before hitting this bad conductance.
Let $e=\left\{ z_1, z_2 \right\}$ be the first good conductance on $l_1(x,\varphi(x))$ after $y$, where the site $z_1$ comes before $z_2$ on the path  $l_1(x,\varphi(x))$.
Then between the sites $y$ and $z_1$ the path $l(x,\varphi(x))$ takes the shortest detour from $y$ to $z_1$ without using any edge with a bad conductance, see Figure \ref{fig:ConstrPath} for a sketch.
This is always possible, even if $y=x$ or $z_1=\varphi(x)$, since $x,\varphi(x)\in\D_n$.
If, for the purpose of the detour, the path has to take a loop as depicted in Example \ref{subfig:ConstrPath3}, then we simply delete the entire loop.
Since the box $B_{3d}(y)$ contains only $3d-1$ bad conductances, the detour is always contained in the edge set $\mathfrak{E}\left(B_{3d}(y)\right)$.
Thus the length of each detour is bounded by a constant $C\leq|\mathfrak{E}\left(B_{3d}(0)\right)|$.
After the detour, $l(x,\varphi(x))$ continues again on $l_1(x,\varphi(x))$ until it hits the next bad conductance and so on.
For all $x\in B_n\cap\D_n \cap \mathscr{H}^\xi$ and for $n$ large enough it follows that $|l(x,\varphi(x))| \leq 2dC(\log n)^{(d+1)} < (\log n)^{2d}$.

We can now apply Lemma \ref{lem:PathVSRW} to obtain the claim.
\end{proof}
\begin{proof}[Proof of Proposition \ref{prop:Path}]
Again, we shortly write $n$ for a member of the subsequence $(n_k)_{k\in\mathbb{N}}$ and, again, we apply Lemma \ref{lem:PathVSRW}.
Let $\mathscr{G} = \left( \mathbb{Z}^d, \edges_d \right)$ and $\nu_n = g(n)$.
Further, let $\mathscr{D}_n$ be as in Proposition \ref{prop:SpectralGapD} and let $\C$ be given by $\mathscr{D}_n$.
Then Condition \ref{item:PoincareCond} of Lemma \ref{lem:PathVSRW} is fulfilled with
\begin{align*}
  \mu_n = \left( 2^{d+1}(\log n)^{4d^2} g\bigl(n^{1-\eps}\bigr)^{-1} + 3c_1 n^2 \xi^{-1}\right)^{-1}\, .
\end{align*}
By virtue of Assumption \ref{item:SpectralGapD:Box} of \ref{prop:SpectralGapD} we can apply Lemma \ref{lem:InSparse} and thus each site $x\in\mathscr{I}_n = B_n \backslash \mathscr{D}_n$ has only neighbors in $\D_n$.
By assumption there exists a neighbor $\varphi(x)$ of $x$ such that the conductance $w_{x,\varphi(x)} > g(n)$.
Further, since $\mathscr{I}_n$ is $3d$-sparse, any neighbor $y$ of $x\in\mathscr{I}_n$ has no second neighbor in $\mathscr{I}_n$.
It follows that the map $\varphi\colon\, \mathscr{I}_n\to\mathscr{D}_n$ is injective and the path $l(x,\varphi(x))=(x,\varphi(x))$ fulfills the requirements of Lemma \ref{lem:PathVSRW}.

It follows that for $n$ large enough
\begin{align*}
	\mathcal{E}^{\vec{w}} \bigl( f\bigr)  \geq  \left( 2^{d+1}g(n)^{-1} + 2^{d+3}(\log n)^{4d^2} g\bigl(n^{1-\eps}\bigr)^{-1} + 9c_1 n^2 \xi^{-1}\right)^{-1} \| f \|^2_2
\end{align*}
for all $f\colon \Z^d \to \mathbb{R}$ with $\supp\, f \subseteq B_{n}$:

We have assumed that one of Assumptions \ref{ass:g} \ref{ass:RVless} or \ref{ass:mon} is fulfilled.
Let us first assume that Assumption \ref{ass:g} \ref{ass:mon} is true and that the limit of $u^2 g(u)$ is smaller than $c_2\in (0,\infty)$.
It follows that eventually
\begin{align*}
  9c_1 \xi^{-1} n^2 g(n) < 9c_1 c_2 \xi^{-1} \quad\text{and }2^{d+3}(\log n)^{4d^2} \frac{g(n)}{g\bigl(n^{1-\eps}\bigr)} < 2^{d+5}(\log n)^{4d^2} n^{-2\eps} <1\, ,
\end{align*}
and therefore for $n$ large enough
\begin{align*}
	\mathcal{E}^{\vec{w}} \bigl( f\bigr)  \geq  \frac{g(n)}{1+ 2^{d+1} + 9c_1 c_2 \xi^{-1}} \| f \|^2_2\qquad (\supp f\subseteq B_{n})\, .
\end{align*}
If we assume that Assumption \ref{ass:g} \ref{ass:mondec} is fulfilled, then eventually even $9c_1 \xi^{-1} n^2 g(n) < 1$ and thus the lower bound becomes independent of $c_1, c_2$, and $\xi$.

Let us now assume that Assumption \ref{ass:RVless} is true.
Then there exists $\rho<-2$ such that we can write $g(n) = n^\rho L(n)$ where $L$ varies slowly at infinity.
It follows that eventually
\begin{align*}
  9c_1 \xi^{-1} n^2 g(n) < 1\text{ and } 2^{d+3} (\log n)^{4d^2} \frac{g(n)}{g\bigl(n^{1-\eps}\bigr)} = n^{\rho \eps}\,\, \frac{2^{d+3}(\log n)^{4d^2} L(n)}{L(n^{1-\eps})} <1\, .
\end{align*}
It follows that in this case for $n$ large enough
\begin{align*}
	\mathcal{E}^{\vec{w}} \bigl( f\bigr)  \geq  \frac{g(n)}{2^{d+1} + 2}\, \| f \|^2_2\, .
\end{align*}
\end{proof}

\section{Localization of the principal eigenvector}
\label{sec:ProofThmLoc}
For the proof of Theorem \ref{thm:Loc} we need to analyse the extreme value statistics of the dependent field of random variables $(\pi_z)_{z\in B_n}$.
Heuristically speaking, since the smallest values of $(\pi_z)_{z\in B_n}$ are far apart (see e.g.\ Lemma \ref{lem:Loc:In}), they are asymptotically independent.
In order to make this argument rigorous (see e.g.\ Lemma \ref{lem:QuotOrder}), we first introduce a decomposition of the lattice $\Z^d$.
Then we continue with a number of auxiliary lemmas in Section \ref{subsec:auxlem} and the extreme value analysis in Section \ref{subsec:Extreme}.
Finally, we give the proof of Theorem \ref{thm:Loc} in Section \ref{subsec:ProofLoc}.
In Section \ref{subsec:ProofCorLoc}, we prove Corollaries \ref{cor:Loc} and \ref{cor:ConvInLaw} .

\subsection{Decomposition of the lattice}
\label{subsec:decomp}
To reduce the number of indices, we fix $k\in\N$ throughout this section, i.e., although most of the quantities discussed in this section depend on some $k\in\N$, it will not show as an index.
For the proof of Theorem \ref{thm:Loc} in Section \ref{subsec:ProofLoc} we only need the case $k=1$.
However, mainly for the purpose of a future paper about the localization of the first $k$ eigenvectors, we define and state everything for general $k\in\N$.

We define the cube $\mathscr{N}\subset \Z^d$ as 
\begin{align*}
  \mathscr{N} := \left\{ 1,\ldots,2(k+1)\right\}^d\, ,
\end{align*}
and the vertex set $\mathscr{V}$ as
\begin{align*}
  \mathscr{V} := \mathop{\dot{\bigcup}}_{x\in (2k+3)\Z^d} \mathscr{N}\circ \tau_x\, ,
\end{align*}
where $\tau_x$ denotes the spatial shift by $x\in\Z^d$.

The important two features of the set $\mathscr{V}$ are that first, for all $a,b\geq 0$ and all $x,y\in (2k+3)\Z^d$ with $x\neq y$, we see that
\begin{align}
  \Pr \left[\min_{z\in \mathscr{N}\circ \tau_x} \pi_z \leq a, \min_{z\in \mathscr{N}\circ \tau_y} \pi_z \leq b \right]
  = \Pr \left[\min_{z\in \mathscr{N}\circ \tau_x} \pi_z \leq a\right]\, \Pr \left[\min_{z\in \mathscr{N}\circ \tau_y} \pi_z \leq b\right]\, ,
\end{align}
and second, the following lemma.

\begin{lem}\label{lem:decompV}
For any vertex set $\mathscr{A}\subset\Z^d$ with cardinality $|\mathscr{A}|\leq 2(k+1)$ there exists $x\in B_{k+1}=\{ -k-1,-k,\ldots,k,k+1 \}^d$ such that $\mathscr{A}\subset \mathscr{V} \circ \tau_x$.
\end{lem}

\begin{proof}
First we note that
 \begin{align*}
  \mathscr{V} = \left\{ y= (y_1, \ldots y_d)\in\Z^d\colon \left( y_1 \not\equiv 0 \text{ mod }(2k+3)\right), \ldots, \left(y_d \not\equiv 0 \text{ mod }(2k+3)\right)\right\}\, .
\end{align*}
Let $\mathscr{A} = \left\{ v_1, \ldots, v_{2k+2} \right\}$ with $v_1, \ldots, v_{2k+2} \in \Z^d$ and let $v_{1,1}, \ldots, v_{2k+2,1}$ be the first components of the vectors $v_1, \ldots, v_{2k+2}$.
Then we choose the first component $x_1$ of the translation vector $x=(x_1,\ldots,x_d)\in B_{k+1}$ such that its residue class modulo $(2k+3)$ is not among the residue classes of $-v_{1,1}, \ldots, -v_{2k+2,1}$ modulo $(2k+3)$.
This is possible since $-v_{1,1}, \ldots, -v_{2k+2,1}$ assume at most $2k+2$ different residue classes modulo $2k+3$.
The other components of the translation vector $x$ are chosen likewise.
\end{proof}

Let us now define the random variable $\chi$ as
\begin{align}
  \chi := \min_{z\in \mathscr{N}} \pi_z
\end{align}
and, for $x\in\Z^d$, analogously $\chi_x$ as
\begin{align}
  \chi_x := \min_{z\in \mathscr{N}\circ \tau_x} \pi_z\, .
\end{align}

Let $F_\pi$ be the distribution function of the random variable $\pi$ as defined before \eqref{equ:Defh}.

\begin{lem}
  \label{lem:DistrChi}
  For any $a\geq 0$, the value of $F_\chi (a) := \Pr \left[ \chi \leq a \right] $ can be bounded by
  \begin{align*}
    (2k+2)^d F_\pi (a) - \binom{(2k+2)^d}{2}F(a)^{4d-1}\;\leq\; F_\chi (a) \;\leq\; (2k+2)^d F_\pi (a)\, .
  \end{align*}
\end{lem}
\begin{proof}
  First we note that
  \begin{align}
    \Pr \left[ \chi \leq a \right] \;=\; \Pr \left[ \min_{z\in \mathscr{N}} \pi_z \leq a \right] \;=\; \Pr \left[ \bigvee_{z\in\mathscr{N}} \left(\pi_z \leq a\right) \right]
    \label{equ:Proof:lem:DistrChi}
  \end{align}
  Since the set $\mathscr{N}$ contains $(2k+2)^d$ vertices, the above RHS is bounded from above by $(2k+2)^d \Pr \left[ \pi\leq a \right]$.
  For the lower bound, we simply expand the RHS of \eqref{equ:Proof:lem:DistrChi} by one term more, i.e.,
  \begin{align}
    \Pr \left[ \bigvee_{z\in\mathscr{N}} \left(\pi_z \leq a\right) \right]
    &\geq (2k+2)^d \Pr \left[ \pi\leq a \right] - \sum_{z_1, z_2\in\mathscr{N},\atop z_1\neq z_2} \Pr \left[ \pi_{z_1}\leq a, \pi_{z_2}\leq a\right]\, .
    \label{equ:Proof:lem:DistrChi:Bound}
  \end{align}
  The claim follows since there are $\binom{(2k+2)^d}{2}$ pairs $z_1, z_2\in\mathscr{N}$ with $z_1\neq z_2$ and in order to achieve that simultaneously $\pi_{z_1}\leq a$ and $\pi_{z_2}\leq a$, at least $4d-1$ independent conductances have to be less than or equal to $a$.
\end{proof}

\begin{lem}\label{lem:DistrChiCont}
  Let $F$ be continuous and let $F_\chi$ be as in Lemma \ref{lem:DistrChi}.
  Then the random variable $F_\chi (\chi)$ is uniformly distributed on $[0,1]$.
\end{lem}
\begin{proof}
  Since $\pi$ is the sum of $2d$ independent random variables with continuous distribution function, it has a continuous distribution function as well.
  It follows that $F_\chi\colon [0,\infty) \to [0,1]$ is also continuous and thus surjective.
  
  Let $a\in[0,1)$.
  Since $F_\chi$ is surjective, there exists $b$ such that $F_\chi (b) = a$.
  Since $F_\chi$ is also monotonically increasing, it follows that
  \begin{align*}
    \Pr \left[ F_\chi(\chi) \leq a \right] \;\leq\; \Pr \left[ \chi \leq \sup\, \left\{ b\colon F_\chi (b) = a \right\} \right]
    \;=\; F_\chi \left( \sup\, \left\{ b\colon F_\chi (b) = a \right\} \right)
    \;=\; a
  \end{align*}
  and
  \begin{align*}
    \Pr \left[ F_\chi(\chi) \leq a \right] \;\geq\; \Pr \left[ \chi \leq \inf\, \left\{ b\colon F_\chi (b) = a \right\} \right]
    \;=\; F_\chi \left( \inf\, \left\{ b\colon F_\chi (b) = a \right\} \right)
    \;=\; a\, .
  \end{align*}
\end{proof}

\subsection{Auxiliary lemmas}
\label{subsec:auxlem}
Throughout this section we assume that the distribution function $F$ is continuous.
We recall that $\psi_1^{(n)}$ is the principal Dirichlet eigenvector and that it is associated with the principal Dirichlet eigenvalue $\lambda_1^{(n)}$.
We normalize it such that $\|\psi_1^{(n)} \|_2 =1$.
By virtue of the Perron-Frobenius Theorem we can assume without loss of generality that $\psi_1^{(n)}$ is nonnegative everywhere, see Remark \ref{rem:PF}.

In Lemma \ref{lem:Loc:MassDn} we are going to see that $\psi_1^{(n)}$ concentrates on the cluster $\D_n$, which we defined in Definition \ref{def:InDn}.
Further, when the sites $z_{(1,n)}, z_{(2,n)}, \ldots, z_{(k,n)}$ are the locations of the smallest, the second-smallest up to the $k$th smallest value of $\pi_z$ for $z\in B_n$, then Lemma \ref{lem:uniquemax} implies that the smaller the quotient $\pi_{z_{(1,n)}}/\pi_{z_{(2,n)}}$, the more $\psi_1^{(n)}$ tends to concentrate in the site $z_{(1,n)}$.
Since $F$ is continuous, these minimizers $z_{(1,n)},\ldots, z_{(k,n)}$ are \Pas unique.

In order to bound the quotient $\pi_{z_{(1,n)}}/\pi_{z_{(2,n)}}$ from above in Section \ref{subsec:Extreme}, we collect some further structural properties of the environment in this section.

For what follows it is important to note that with $g$ as defined in \eqref{def:ThmLocFuncG}, we have
\begin{align*}
  \Lambda_g (n) = n^d \Pr \left[ w\leq \sup \left\{ s\colon F(s) = n^{-1/2} \right\} \right]^{2d} = 1\, .
\end{align*}
We thus have the following lemma.

\begin{lem}\label{lem:Loc:In}
  Let $g$ be as in \eqref{def:ThmLocFuncG} and $\eps_2 \in (0, 1/3)$.
  Let $b,k\in\N$.
  Then \Pas for $n$ large enough and for all $z\in B_{n+b}$ the edge set $\mathfrak{E}\left(B_b (z)\right)$ contains at most $3d-1$ edges with conductance less than or equal to $g(n^{1-\eps_2})$.
  Furthermore, if $\D_n$ is as in Definition \ref{def:InDn} with $\eps=\eps_2$, then \Pas for $n$ large enough
  the set $\mathscr{I}_n = B_n\backslash \D_n$ is $b$-sparse and $z_{(1,n)},\ldots, z_{(k,n)}\in\mathscr{I}_n$.
\end{lem}
\begin{proof}
Since $\Lambda_g$ is constant and therefore bounded, the first claim follows by virtue of Corollary \ref{cor:BC:i3} (with $m=2d$ and $\kappa=d$).

Since the function $n\mapsto g(n^{1-\eps_2})$ decreases to zero, \Pas for $n$ large enough the cluster $\D_n$ is the unique infinite cluster of the environment $\tilde{\vec{w}}_{g\left( n^{1-\eps_2} \right)}$.
We can thus apply Lemma \ref{lem:InSparse} and obtain that \Pas for $n$ large enough the set $\mathscr{I}_n$ is $b$-sparse.

For the last statement consider the following:
Since the quotient $\Lambda_{g\left((\cdot)^{1-\eps_2/2}\right)}(n)/\log\log n$ diverges for $n$ growing to infinity, Lemma \ref{lem:BC:ii} implies that \Pas for $n$ large enough $\pi_{z_{(1,n)}}\leq \ldots \leq \pi_{z_{(k,n)}} < 2d g(n^{1-\eps_2/2})$.
This implies that eventually $z_{(1,n)},\ldots, z_{(k,n)}\in\mathscr{I}_n$.
\end{proof}

\begin{lem}\label{lem:Loc:MassDn}
  Let the function $g$ be as in \eqref{def:ThmLocFuncG}.
  Assume that there exists $\eps_1\in (0,1)$ such that one of the two cases occurs: $g$ varies regularly at infinity with index $\rho<-(2+\eps_1)$ or the product $n^{2+\eps_1} g(n)$ converges monotonically to zero as $n$ grows to infinity. Further, let $\eps= \eps_2 = \frac{7\eps_1}{8(2+\eps_1)}$ and $\D_n$ be as in Definition \ref{def:InDn}. Then \Pas for $n$ large enough
  \begin{align}
     \bigl\| \psi_1^{(n)} \bigr\|^2_{\ell^2(\D_n)} \leq n^{-\eps_1/2}\, .
  \end{align}
\end{lem}
\begin{proof}
We aim to apply Proposition \ref{prop:SpectralGapD} to the set $\D_n$.
By virtue of Lemma \ref{lem:Loc:In} with $b=3d$, it follows that Assumption \ref{item:SpectralGapD:Box} of Proposition \ref{prop:SpectralGapD} is fulfilled \Pas for $n$ large enough.
Further we choose $\xi>0$ small enough such that \Pas for $n$ large enough Assumptions \ref{item:SpectralGapD:Cluster}, \ref{item:SpectralGapD:SpectralGapC} and \ref{item:SpectralGapD:InjMap} are fulfilled.
This is possible by virtue of the Lemmas \ref{lem:PropIICBox}, \ref{lem:InjectiveMap} and \ref{lem:SpectralGapPerc}.
It follows that there exists $c>0$ such that \Pas for $n$ large enough
\begin{align}
	\mathcal{E}_{\D_n}^{\vec{w}} \bigl( f \bigr) \geq  \left( 2^{d+1} (\log n)^{4d^2} g\bigl(n^{1-\eps_2}\bigr)^{-1} + c n^2 \right)^{-1} \| f \|^2_{\ell^2(\D_n)}\, ,
	\label{equ:Loc:MassDnOne}
\end{align}
for any function $f\colon\Z^d\to \R$ with $\supp\, f\subseteq B_n$.
In any case, the assumptions imply that the product $n^{2+\eps_1} g(n)$ converges to zero as $n$ grows to infinity.
Therefore $n^2 g\bigl(n^{1-\eps_2}\bigr)/(\log n)^{4d^2}$ converges to zero as well.
It follows that if $C = 2^{d+1} + 1$, then \eqref{equ:Loc:MassDnOne} implies that \Pas for $n$ large enough
\begin{align*}
  \mathcal{E}_{\D_n}^{\vec{w}} \bigl( f \bigr) \geq  \frac{1}{C}\, \frac{g\bigl(n^{1-\eps_2}\bigr)}{(\log n)^{4d^2} } \| f \|^2_{\ell^2(\D_n)}\, .
\end{align*}
On the other hand, we know that for any $\eps_3>0$ the term $\Lambda_{g\left( (\cdot)^{1-\eps_3} \right)} (n)/\log\log n$ diverges.
Let us specifically choose $\eps_3 = \eps_1 \left(8 (2+\eps_1)\right)^{-1}$.
Now we use Theorem \ref{thm:EigvalsUpperBoundSufficient} and the fact that the Dirichlet energy $\mathcal{E}^{\vec{w}}$ majorizes $\mathcal{E}_{\D_n}^{\vec{w}}$ to infer that \Pas for $n$ large enough
\begin{align}
  2d g(n^{1-\eps_3}) \geq \lambda_1^{(n)} = \mathcal{E}^{\vec{w}} \left( \psi_1^{(n)} \right)\geq \mathcal{E}_{\D_n}^{\vec{w}} \bigl( \psi_1^{(n)} \bigr) \geq  \frac{1}{C}\, \frac{g\bigl(n^{1-\eps_2}\bigr)}{(\log n)^{4d^2} } \bigl\| \psi_1^{(n)} \bigr\|^2_{\ell^2(\D_n)}\, .
  	\label{equ:ProofLoc1}
\end{align}
When we solve this inequality for $\bigl\| \psi_1^{(n)} \bigr\|_{\ell^2 (\D_n)}^2$, we obtain that
\begin{align*}
   \bigl\| \psi_1^{(n)} \bigr\|_{\ell^2 (\D_n)}^2 \leq c_1 C\,\frac{g(n^{1-\eps_3}) (\log n)^{4d^2}}{g(n^{1-\eps_2})}\, .
\end{align*}
To finish the proof, we use one of the additional assumptions about $g$:
If $g$ varies regularly at infinity with index $\rho<-(2+\eps_1)$, then we can write $g(n) = n^{\rho} L(n)$ where $L$ varies slowly at infinity. In this case we observe that eventually
\begin{align*}
   c_1 C\,\frac{g(n^{1-\eps_3}) (\log n)^{4d^2}}{g(n^{1-\eps_2})} = c_1 C n^{\frac{3\rho \eps_1}{4 (2+\eps_1)}} \frac{(\log n)^{4d^2} L(n^{1-\eps_3})}{L(n^{1-\eps_2})}\leq n^{-\eps_1/2}\, ,
\end{align*}
which implies the claim.
In the other case, i.e., if the product $n^{2+\eps_1} g(n)$ converges monotonically to zero as $n$ tends to infinity, we observe that eventually
\begin{align*}
   c_1 C\,\frac{g(n^{1-\eps_3}) (\log n)^{4d^2}}{g(n^{1-\eps_2})} \leq c_1 C n^{-(2+\eps_1)(\eps_2-\eps_3)} (\log n)^{4d^2}\leq n^{-\eps_1/2}\, ,
\end{align*}
which implies the claim as well.
\end{proof}
\begin{lem}\label{lem:uniquemax}
   Let $y,z\in B_n$ with $\pi_z < \pi_y$ and $y\nsim z$. Assume that $\psi_1^{(n)}$ is nonnegative. Further, define $m_y = 2 \max_{x\colon x\sim y} \psi_1^{(n)}(x)$. Then the mass $\psi_1^{(n)} (y)$ is bounded from above by
   \begin{align}
	\psi_1^{(n)}(y) \leq \frac{m_y}{1 - \frac{\pi_z}{\pi_y}}\, .
	\end{align}
\end{lem}
\begin{proof}[Proof of Lemma \ref{lem:uniquemax}]
We assume the contrary, i.e., we assume that
\begin{align}
  m_y \pi_y + \psi_1^{(n)}(y) \left( \pi_z - \pi_y \right) <0\, .
  \label{equ:uniquemaxass}
\end{align}
Then we define a new function $\phi:\mathbb{Z}^d \rightarrow \mathbb{R}_+$ by setting
\begin{align}
   \phi(x) &= \begin{cases}
   			\psi_1^{(n)}(x), &\text{for } x \notin \{ y,z \},\\
   			m_y, &\text{for } x = y,\\
   			\sqrt{\psi_1^{(n)} (y)^2 + \psi_1^{(n)} (z)^2 - m_y^2}, &\text{for } x = z.
   \end{cases}
\end{align}
Note that since \eqref{equ:uniquemaxass} implies that $\psi_1^{(n)}(y) > m_y$, it must be $\phi(z) > \psi_1^{(n)} (z)$.
Obviously, $\supp\, \phi \subseteq B_n$ and $\| \phi\|_2 = 1$.
Therefore, by the variational formula \eqref{equ:Variational} and Remark \ref{rem:PF}, the Dirichlet energy $\left\langle \phi, -\mathcal{L}_{\boldsymbol{w}} \phi\right\rangle$ is larger than the principal Dirichlet eigenvalue $\lambda_1^{(n)}$.

However, the Dirichlet energy of $\phi$ is given by
\begin{align}
   \left\langle \phi, -\mathcal{L}_{\boldsymbol{w}} \phi\right\rangle &= \lambda_1^{(n)} + \left[ \sum_{x:x\sim y} w_{xy} \left( \psi_1^{(n)}(x) - m_y \right)^2 - \sum_{x:x\sim y} w_{xy} \left( \psi_1^{(n)}(x) - \psi_1^{(n)}(y) \right)^2\right]\nonumber\\
   			&\qquad + \left[\sum_{x:x\sim z} w_{xz} \left( \psi_1^{(n)}(x) - \phi(z) \right)^2-  \sum_{x:x\sim z} w_{xz} \left( \psi_1^{(n)}(x) - \psi_1^{(n)}(z) \right)^2\right]\, .
   \label{equ:DirEnVarphi}
\end{align}
Evaluation of the first bracketed summand on the RHS gives:
\begin{align}
  \sum_{x:x\sim y} w_{xy} &\left( \psi_1^{(n)}(x) - m_y \right)^2 - \sum_{x:x\sim y} w_{xy} \left( \psi_1^{(n)}(x) - \psi_1^{(n)}(y) \right)^2\nonumber\\
  &= \sum_{x:x\sim y} w_{xy} \left(  \psi_1^{(n)}(y) -m_y \right) \left(  2\psi_1^{(n)}(x) - m_y - \psi_1^{(n)}(y) \right)\nonumber\\
  &\leq -\psi_1^{(n)}(y) \sum_{x:x\sim y} w_{xy}  \left(  \psi_1^{(n)}(y) -m_y \right)\, ,
  \label{equ:EvalFirstSummand}
\end{align}
where the last inequality follows by the definition of $m_y$ and since Assumption \eqref{equ:uniquemaxass} implies that $\psi_1^{(n)}(y) > m_y$.
Further, we evaluate the second bracketed summand on the RHS of \eqref{equ:DirEnVarphi} as
\begin{align*}
\sum_{x:x\sim z} w_{xz} \left( \psi_1^{(n)}(x) - \phi(z) \right)^2&-  \sum_{x:x\sim z} w_{xz} \left( \psi_1^{(n)}(x) - \psi_1^{(n)}(z) \right)^2\nonumber\\
  	&\hspace{-.7cm}= \sum_{x:x\sim z} w_{xz} \left( \phi(z) - \psi_1^{(n)}(z)\right) \left( \phi(z) + \psi_1^{(n)}(z) -2 \psi_1^{(n)}(x)  \right)\, .
\end{align*}
Since $\psi_1^{(n)}$ is nonnegative and since Assumption \eqref{equ:uniquemaxass} implies that $\phi(z) > \psi_1^{(n)} (z)$, we conclude that
\begin{align}
\sum_{x:x\sim z} w_{xz} \left( \psi_1^{(n)}(x) - \phi(z) \right)^2&-  \sum_{x:x\sim z} w_{xz} \left( \psi_1^{(n)}(x) - \psi_1^{(n)}(z) \right)^2\nonumber\\
  	&\hspace{-1cm}\leq \sum_{x:x\sim z} w_{xz} \left( \phi(z)^2 - \psi_1^{(n)}(z)^2\right) = \sum_{x:x\sim z} w_{xz} \left( \psi_1^{(n)}(y)^2 - m_y^2\right)\, ,
  	\label{equ:EvalSecondSummand}
\end{align}
where the last equality follows by the definition of $\phi(z)$.
When we insert \eqref{equ:EvalFirstSummand} and \eqref{equ:EvalSecondSummand} into \eqref{equ:DirEnVarphi}, then we obtain
\begin{align}
   \left\langle \phi, -\mathcal{L}_{\boldsymbol{w}} \phi\right\rangle &\leq \lambda_1^{(n)} - \psi_1^{(n)}(y)\sum_{x:x\sim y} w_{xy}  \left(  \psi_1^{(n)}(y) -m_y \right) + \sum_{x:x\sim z} w_{xz} \left( \psi_1^{(n)}(y)^2 - m_y^2\right)\nonumber\\
   		&= \lambda_1^{(n)} + \psi_1^{(n)} (y)^2 \left( \pi_z - \pi_y \right) + m_y \left( \psi_1^{(n)}(y) \pi_y  - m_y \pi_z  \right)\nonumber\\
   		&\leq \lambda_1^{(n)} + \psi_1^{(n)} (y) \left[ m_y \pi_y + \psi_1^{(n)}(y) \left( \pi_z - \pi_y \right) \right]\, .
\end{align}
Under Assumption \eqref{equ:uniquemaxass} and because $\psi_1^{(n)} (y)$ is nonnegative, it follows that the Dirichlet energy of $\phi$ is not larger than $\lambda_1^{(n)}$.
This is a contradiction to the considerations above.
\end{proof}

Let $F_\pi$ be as defined before \eqref{equ:Defh}.
Then we have the following two lemmas.

\begin{lem}
  \label{lem:InequDistrPi}
  If there exists $a^\ast>0$ such that $F(ab) \geq bF(a)$ for all $a\leq a^\ast$ and all $0\leq b\leq1$, then $F_\pi(ab) \geq b^{2d} F_\pi(a)$ for all $a\leq a^\ast$ and all $0\leq b\leq1$.
\end{lem}
\begin{proof}
  Let $a\leq a^\ast$, $0\leq b\leq1$ and let $w_1, w_2$ be two independent copies of $w$.
  Then
  \begin{align*}
    \Pr [ w_1 + w_2 \leq ab] &\;=\; \int_0^\infty \Pr [w_2 \leq ab - t] \d \Pr [w_1 \leq t]
      \;\geq\; b \int_0^\infty \Pr \left[w_2 \leq a - \frac{t}{b}\right] \d \Pr [w_1 \leq t]\, .
  \end{align*}
  where we have used that $\Pr [w_2 \leq ab - t]\geq b \Pr [w_2 \leq a - t/b]$ and $\d \Pr [w_1 \leq t] \geq 0$ for all $t\in[0,\infty)$.
  It follows that
    \begin{align*}
    \Pr [ w_1 + w_2 \leq ab] &\;\geq\; b \Pr [ w_2 \leq a - w_1/b] \;=\; b \Pr [ w_1 \leq ab - bw_2]\, .
  \end{align*}
  By the same reasoning as before we infer that $\Pr [ w_1 + w_2 \leq ab] \geq b^2 \Pr [ w_1 + w_2 \leq a]$.
  The claim follows by induction.
\end{proof}

\begin{lem}
  \label{lem:DistrPi}
  If there exists $\gamma\in [0,1/4)$ such that $F$ varies regularly at zero with index $\gamma$, then $F_\pi$ varies regularly at zero with index $2d\gamma$.
\end{lem}
\begin{proof}
  Let $\mathscr{L} [F]$ be the Laplace transform of $F$. Then the Laplace transform of $F_\pi$ fulfills
  \begin{align*}
    \mathscr{L} \left[F_\pi\right] = \left( \mathscr{L} [F]\right)^{2d}\, .
  \end{align*}
  By virtue of the Tauberian theorems, more precisely by virtue of Theorem 3 in \cite[XIII.5]{Feller2} (or, equivalently Theorem 1.7.1' of \cite{BinghamRV}), $\mathscr{L} [F]$ varies regularly at infinity with index $-\gamma$. It follows that $\mathscr{L} \left[F_\pi\right]$ varies regularly at infinity with index $-2d\gamma$. Hence, by another application of Theorem 3 in \cite[XIII.5]{Feller2} we obtain that $F_\pi$ varies regularly at zero with index $2d\gamma$.
\end{proof}

\begin{lem}
  \label{lem:IndepSpacingRelRank}
Let $\sigma_{1},\sigma_2,\ldots$ be a sequence of i.i.d.\ random variables with continuous distribution.
For $N\in\N$, let $\sigma_{1,N}\geq \sigma_{2,N} \geq \ldots \geq \sigma_{N,N}$ be the $N$th order statistics.
Let $a>0$ and $i,j,k,l,m\in\N$ with $N>j>i$ as well as $N>l>m>k$.
Then the events $\left\{ \sigma_{i,N}\ - \sigma_{j,N} \leq a\right\}$ and $\left\{ \sigma_{k,l} > \sigma_{k,m} \right\}$ are independent.
\end{lem}
\begin{proof}
  As in the proof of \cite[Proposition 4.3]{Resnick1987}, we observe that since the distribution of the $\sigma$'s is continuous, we can assume without loss of generality that there are no ties between the $\sigma$'s and we observe that each of the $N!$ orderings $\sigma_{q_1}<\ldots<\sigma_{q_N}$ is equally likely where $q_1,\ldots,q_N$ is a permutation of $1,\ldots,N$.
  Now we note that whether or not $\left\{ \sigma_{k,l} > \sigma_{k,m} \right\}$ is univocally given by the specific ordering $(q_1,\ldots,q_N)$.
  On the other hand, the difference between any two order statistics $\sigma_{i,N}$ and $\sigma_{j,N}$ is completely independent of the ordering $(q_1,\ldots,q_N)$.
\end{proof}

\subsection{Extreme value analysis}
\label{subsec:Extreme}
In what follows, we let $\pi_{1,B_n}\leq \pi_{2,B_n} \leq \ldots \leq \pi_{|B_n|, B_n}$ denote the order statistics of the set $\left\{ \pi_x \colon x\in B_n\right\}$.

\begin{lem}[Quotient of order statistics]
\label{lem:QuotOrder}
  Let the assumptions of Theorem \ref{thm:Loc} be true and let $\eps>0$ and $k\in\N$.
  Then \Pas for $n$ large enough
  \begin{align}
    1-\frac{\pi_{k,B_n}}{\pi_{k+1,B_n}} > n^{-\eps}\, .
  \end{align}
\end{lem}
\begin{proof}
Without loss of generality we may and will assume that $\eps<1$.
For $l\in\N$ let $z_{(l,n)}$ be the site in $B_n$ that fulfills $\pi_{z_{(l,n)}} = \pi_{l,B_n}$.
Note that since $F$ is continuous, the sites $z_{(l,n)}$ are \Pas unique.

Let us recall that we assumed that one of the two following cases occurs: $\gamma\in (0,1/4)$ or $\gamma=0$ and there exists $\eps_1\in (0,1)$ such that the product $n^{2+\eps_1} g(n)$ converges monotonically to zero as $n$ grows to infinity.

In the case where $\gamma>0$, it follows that $(1/F(1/s))^2$ varies regularly at infinity with index $2\gamma$.
Further, $(1/F(1/s))^2$ diverges as $s\to\infty$.
It follows by virtue of \cite[Prop. 0.8(v)]{Resnick1987} that $1/g(u) = \inf\,\left\{ s\geq 0\colon (1/F(1/s))^2 = u \right\}$ varies regularly at infinity with index $1/(2\gamma)$ and thus $g$ varies regularly at infinity with index $-1/(2\gamma)$.
Since in addition $\gamma<1/4$, there exists $\eps_1\in (0,1)$ such that $-1/(2\gamma) < -(2+\eps_1)$.

In both cases we define $\D_n$ as in Definition \ref{def:InDn} with $\eps= \eps_2=\frac{7\eps_1}{8(2+\eps_1)}$.
Let $\mathscr{I}_n = B_n\backslash \D_n$.
By virtue of Lemma \ref{lem:Loc:In} and Remark \ref{rem:b1b2sparse} we know that \Pas for $n$ large enough the set $\mathscr{I}_n$ is $(2k+3)$-sparse in the sense of Definition \ref{def:sparse} and further $z_{(1,n)},\ldots, z_{(k+1,n)}\in\mathscr{I}_n$.
It follows that \Pas for $n$ large enough there is no pair of neighbors among the the sites $z_{(1,n)}, \ldots, z_{(k+1,n)}$.
In the proof of Lemma \ref{lem:Loc:In}, we also find that \Pas for $n$ large enough $\pi_{z_{(1,n)}}\leq \ldots\leq \pi_{z_{(k+1,n)}} < 2d g(n^{1-\eps_2/2})$, which is eventually smaller than $g(n^{1-\eps_2})$.
Moreover, Lemma \ref{lem:Loc:In} yields that \Pas for $n$ large enough and for all $z\in B_{n+2k+3}$ the edge set $\mathfrak{E}(B_{2k+3})$ contains at most $3d-1$ edges with conductance less than or equal to $g(n^{1-\eps_2})$.
Further, recall the definitions of Section \ref{subsec:decomp} and let $\chi_{(1,n)}^{(x)} \leq \chi_{(2,n)}^{(x)} \leq \ldots$ be the order statistics of $\left\{ \chi_y \colon y\in ((2k+3)\Z^d + x), \mathscr{N}\circ \tau_y \subset B_{n+2k+1} \right\}$.
Since for any $j$ and any $x\in B_{k+1}$, the random variable $\chi_{(j,n)}^{(x)}$ \Pas decreases monotonically to zero as $n$ tends to infinity, it follows that \Pas for $n$ large enough $\chi_{(k+2,n)}^{(x)} \leq a^\ast$ with $a^\ast$ from the assumptions of the theorem.
In summary, the event
\begin{align}
  G_n &:= \left\{ z_{(1,n)}, \ldots, z_{(k+1,n)} \in\mathscr{I}_n\right\} \cap \left\{ \mathscr{I}_n \text{ $(2k+3)$-sparse}\right\} \cap \left\{ \max_{x\in B_{k+1}} \chi_{(k+2,n)}^{(x)} \leq a^\ast \right\}\nonumber\\
  &\qquad\cap \left\{ \pi_{z_{(1,n)}}< \ldots< \pi_{z_{(k+1,n)}} < g(n^{1-\eps_2}) \right\}\nonumber\\
  &\qquad\cap\left\{ \forall z\in B_{n+2k+3}\colon \left|\left\{ e\in\mathfrak{E}(B_{2k+3} (z))\colon w_e \leq g(n^{1-\eps_2})\right\}\right|\leq 3d-1 \right\}
  \label{equ:GnPas}
\end{align}
occurs \Pas for $n$ large enough.

We abbreviate $E_n = \left\{ n^{-\eps} \geq 1 - \frac{\pi_{k,B_n}}{\pi_{k+1,B_n}} \right\}$.
Since $\Pr \left[ \limsup_{n\to \infty} G_n^{\rm c} \right] = 0$, we already know that
\begin{align*}
  \Pr &\left[ \limsup_{n\to\infty} E_n\right]
  = \Pr \left[ \limsup_{n\to \infty}\,  \left( E_n \cap G_n \right) \right]\\
  &\qquad= \lim_{n\to\infty}\Pr \left[ \left(E_n \cap G_n\right) \cup \left(E_{n+1} \cap G_{n+1}\right) \cup \bigcup_{m= n+2}^\infty \left( E_m \cap G_m \cap \left( E_{m-2} \cap G_{m-2} \right)^{\rm c} \right) \right]\, .
\end{align*}
Since again $\Pr \left[ \limsup_{m\to \infty} G_m^{\rm c} \right] = 0$, we further conclude that
\begin{align}
  \Pr \left[ \limsup_{n\to\infty} E_n\right]&\leq (2k+3)\limsup_{n\to\infty}\Pr \left[ E_n \cap G_n \right] + \lim_{n\to\infty}\Pr \left[ \bigcup_{m= n+2k+3}^\infty \left( E_m \cap G_m \cap E_{m-2k-3}^{\rm c}\right) \right]\, .
  \label{equ:ProofLoc:Split}
\end{align}
We now treat the first and second term on the above RHS separately.

Let us first show that
\begin{align}
  \limsup_{n\to\infty}\Pr \left[ E_n \cap G_n \right] = 0\, .
  \label{equ:ProofLoc:Limsup}
\end{align}
Our aim is to use the notation of Section \ref{subsec:decomp} in order to decompose
\begin{align*}
  E_n \cap G_n &= \bigcup_{x\in B_{k+1}} \left( E_n \cap G_n\cap \left\{ \{z_{(1,n)},\ldots, z_{(k+1,n)}\} \subset \mathscr{V}\circ \tau_x \right\} \right)
\end{align*}
and to compare $\pi_{z_{(1,n)}}, \ldots, \pi_{z_{(k+1,n)}}$ with the $k+1$ smallest values of $\chi_y$, $y\in \{ y'\in ((2k+3)\Z^d + x)\colon \mathscr{N}\circ\tau_{y'} \cap B_n \neq \emptyset\}$ for an $x\in B_{k+1}$.
Let $x\in B_{k+1}$.
There is a technical difficulty since
\begin{align}
  \text{it is {\bfseries not} necessarily}\quad \bigcup_{y'\in ((2k+3)\Z^d + x), \atop\mathscr{N}\circ\tau_{y'} \cap B_n \neq \emptyset} \mathscr{N}\circ\tau_{y'} \subseteq B_n\, .
\end{align}
but instead
\begin{align}
  \bigcup_{y'\in ((2k+3)\Z^d + x), \atop\mathscr{N}\circ\tau_{y'} \cap B_{n} \neq \emptyset} \mathscr{N}\circ\tau_{y'} \subseteq B_{n+2k+1}\, .
\end{align}

Now we recall that the order statistics of $\left\{ \chi_y \colon y\in ((2k+3)\Z^d + x), \mathscr{N}\circ \tau_y \subset B_{n+2k+1} \right\}$ is denoted by $\chi_{(1,n)}^{(x)} \leq \chi_{(2,n)}^{(x)} \leq \ldots$.
Then on any event $G_n \cap \left\{ \left\{ z_{(1,n)}, \ldots, z_{(k,n)}\right\} \subset \mathscr{V}\circ \tau_x \right\}$, it follows that
\begin{align*}
  \pi_{z_{(k,n)}} &\geq \chi_{(k,n)}^{(x)}\, .
\end{align*}
Let us now show that on the event $\left\{ \{z_{(1,n)},\ldots ,z_{(k+1,n)}\} \subset \mathscr{V}\circ \tau_x \right\}$ \Pas for $n$ large enough $\pi_{z_{(k+1,n)}} \leq \chi_{(k+2,n)}^{(x)}$.
We assume the counter event, i.e., that $\pi_{z_{(k+1,n)}} > \chi_{(k+2,n)}^{(x)}$.
It follows that on the event $G_n\cap \left\{ \{z_{(1,n)},\ldots ,z_{(k+1,n)}\} \subset \mathscr{V}\circ \tau_x \right\}$ there exists two sites $z^\ast_1, z^\ast_2 \in \left( B_{n+2k+1}\backslash B_n \right)\cap \mathscr{V}\circ\tau_x$ with $\pi_{z_1^\ast}<\pi_{z_{(k+1,n)}}$ and $\pi_{z_2^\ast}<\pi_{z_{(k+1,n)}}$.
Since on $G_n$ each edge set $\mathfrak{E}(\mathscr{N}\circ \tau_y)$ with $y\in B_{n+2k+3}$ contains at most $3d-1$ edges with conductance less than or equal to $g(n^{1-\eps_2})$ and further $\pi_{z_{(k+1,n)}} < g(n^{1-\eps_2})$, it follows that for $n$ large enough, these two sites have to be located in different cubes $\mathscr{N}\circ y^\ast_1$ and $\mathscr{N}\circ y^\ast_1$ (with $y^\ast_1, y^\ast_2 \in (2k+3)\Z^d +x$).
Thus $\chi_{y^\ast_1}$ and $\chi_{y^\ast_2}$ are new records in the sense that both $\chi_{y^\ast_1}<\chi_{(k+1,n-2k-3)}^{(x)}$ and $\chi_{y^\ast_2}<\chi_{(k+1,n-2k-3)}^{(x)}$.
Now we observe that the cardinality of the set $\{ y'\in ((2k+3)\Z^d + x)\colon \mathscr{N}\circ\tau_{y'} \cap B_n \neq \emptyset, \mathscr{N}\circ\tau_{y'} \cap B_{n-2k-3} = \emptyset\}$ is of order $n^{d-1}$.
Further, by virtue of \cite[Proposition 4.3]{Resnick1987}, the probability that one specific value $\chi_{y^\ast}$ with $y^\ast\in \{ y'\in ((2k+3)\Z^d + x)\colon \mathscr{N}\circ\tau_{y'} \cap B_n \neq \emptyset, \mathscr{N}\circ\tau_{y'} \cap B_{n-2k-3} = \emptyset\}$ fulfills $\chi_{y^\ast}<\chi_{(k+1,n-2k-3)}^{(x)}$, is of order $n^{-d}$.
It follows that the probability of the event $\pi_{z_{(k+1,n)}} > \chi_{(n,k+2)}^{(x)}$ is of order $\left(n^{d-1}/n^d\right)^2 = n^{-2}$ and thus the claim follows by the Borel-Cantelli lemma.

Since $\{z_{(1,n)},\ldots, z_{(k+1,n)}\} \subset \mathscr{V}\circ \tau_x$ implies that
\begin{align*}
 \pi_{z_{(1,n)}},\ldots, \pi_{z_{(k+1,n)}} \in \left\{ \chi_y \colon y\in ((2k+3)\Z^d + x), \mathscr{N}\circ \tau_y \subset B_{n+2k+1} \right\}\, ,
\end{align*}
it follows that both values $\pi_{z_{(k,n)}}$ and $\pi_{z_{(k+1,n)}}$ are in $\left\{ \chi_{(k,n)}^{(x)}, \chi_{(k+1,n)}^{(x)}, \chi_{(k+2,n)}^{(x)}\right\}$.
Thus
\begin{align}
  \frac{\pi_{z_{(k,n)}}}{\pi_{z_{(k+1,n)}}} \;\leq\; \max_{k\leq i<j\leq k+2}\; \frac{\chi_{(i,n)}^{(x)}}{\chi_{(j,n)}^{(x)}}
\end{align}

Now we define $F_\chi$ as in Lemma \ref{lem:DistrChi}.
By definition, $F_\chi$ is an increasing function and thus on the event $G_n \cap \left\{ z_{(1,n)},\ldots, z_{(k+1,n)} \in \mathscr{V}\circ \tau_x \right\}$ it follows that 
\begin{align*}
  \Biggl( n^{-\eps} > 1 - \frac{\pi_{z_{(k,n)}}}{\pi_{z_{k+1,n}}} \Biggr)
  \text{ implies } \Biggl( \exists\, k\leq i<j\leq k+2 \,  \colon  F_\chi \left( \chi_{(i,n)}^{(x)} \right) \geq F_\chi \left( (1-n^{-\eps}) \chi_{(j,n)}^{(x)}\right)\Biggr)\, .
\end{align*}

Our next aim is to extract the factor $(1-n^{-\eps})$ from inside the function argument of $F_\chi$.
Since on the event $G_n$ we have $\chi_{(j,n)}^{(x)}\leq a^\ast$ for all $j\leq k+2$, we estimate by virtue of Lemmas \ref{lem:DistrChi} and \ref{lem:InequDistrPi}
\begin{align*}
  F_\chi &\left( (1-n^{-\eps}) \chi_{(j,n)}^{(x)}\right) \mspace{-2mu}\stackrel{\ref{lem:DistrChi}}{\geq} \mspace{-4mu}(2k+2)^d F_\pi \left((1-n^{-\eps}) \chi_{(j,n)}^{(x)}\right) \mspace{-2mu}-\mspace{-2mu} \binom{(2k+2)^d}{2}F\left( (1-n^{-\eps}) \chi_{(j,n)}^{(x)}\right)^{4d-1}\\
  &\qquad\qquad\qquad\stackrel{\ref{lem:InequDistrPi}}{\geq} (1-n^{-\eps})^{2d} (2k+2)^d F_\pi \left( \chi_{(j,n)}^{(x)}\right) - \binom{(2k+2)^d}{2}F\left( (1-n^{-\eps}) \chi_{(j,n)}^{(x)}\right)^{4d-1}\\
  &\qquad\qquad\qquad\stackrel{\ref{lem:DistrChi}}{\geq}(1-n^{-\eps})^{2d} F_\chi \left( \chi_{(j,n)}^{(x)}\right) - \binom{(2k+2)^d}{2}F\left( (1-n^{-\eps}) \chi_{(j,n)}^{(x)}\right)^{4d-1}\, .
\end{align*}
Thus, on the event $G_n \cap \left\{ \{z_{(1,n)},\ldots, z_{(k+1,n)}\} \subset \mathscr{V}\circ \tau_x \right\}$ the event $E_n$ implies that
\begin{align*}
  \exists\, k\leq i<j\leq k+2\,  \colon \frac{F_\chi \left( \chi_{(i,n)}^{(x)} \right)}{F_\chi \left( \chi_{(j,n)}^{(x)}\right)} \;\geq\; (1-n^{-\eps})^{2d} - \binom{(2k+2)^d}{2}\frac{F\left( (1-n^{-\eps}) \chi_{(j,n)}^{(x)}\right)^{4d-1}}{F_\chi \left( \chi_{(j,n)}^{(x)}\right)}\, .
\end{align*}

Now we observe that by virtue of Lemma \ref{lem:DistrChiCont}, the random variable $F_\chi \left( \chi \right)$ is uniform on $[0,1]$.
It follows that the random variable $\sigma := -\log F_\chi \left( \chi\right)$ is exponentially distributed with parameter $1$.
In analogy to the definitions above, we define $\sigma_z := -\log F_\chi \left( \chi_z\right)$, $\sigma_{(i,n)}^{(x)} := -\log \chi_{(i,n)}^{(x)}$ for $i=1,\ldots,k+2$.

Thus, we can bound $\Pr \left[ E_n\cap G_n \right]$ by
\begin{align}
  &\Pr \left[ E_n\cap G_n \right]\nonumber\\
  &\leq\mspace{-10mu} \sum_{x\in B_{k+1}} \sum_{i,j=k,\atop i<j}^{k+2} \Pr \Biggl[ \left\{ \sigma_{(i,n)}^{(x)} - \sigma_{(j,n)}^{(x)} \leq -\log \left( (1-n^{-\eps})^{2d} - \tbinom{(2k+2)^d}{2}\tfrac{F\left( (1-n^{-\eps}) \chi_{(j,n)}^{(x)}\right)^{4d-1}}{F_\chi \left( \chi_{(j,n)}^{(x)}\right)} \right)\right\}\nonumber\\
  &\qquad\qquad\qquad\qquad\qquad\qquad\qquad\cap G_n\cap \left\{ \{z_{(1,n)},\ldots, z_{(k+1,n)}\} \subset \mathscr{V}\circ \tau_x \right\} \Biggr]\nonumber\\
  &\leq \sum_{x\in B_{k+1}} \sum_{k\leq i<j\leq k+2} \Pr \Biggl[ \sigma_{(i,n)}^{(x)} - \sigma_{(j,n)}^{(x)} \leq -\log \left( (1-n^{-\eps})^{2d} - n^{-\eps} \right)\Biggr]+\nonumber\\
  &\qquad\qquad\qquad\qquad+\sum_{x\in B_{k+1}}\sum_{k\leq i<j\leq k+2} \Pr \Biggl[ \binom{(2k+2)^d}{2}\tfrac{F\left( (1-n^{-\eps}) \chi_{(j,n)}^{(x)}\right)^{4d-1}}{F_\chi \left( \chi_{(j,n)}^{(x)}\right)} >n^{-\eps}\Biggr]\, .
  \label{equ:UpperBoundOne}
\end{align}
In the first summand on the above RHS we have the difference between any pair of the $k$th to $(k+2)$th largest values of a sequence of independent exponential variables with parameter 1.
By virtue of \cite[Chapter 5, Theorem 2.3]{Devroye1986}, we know that the normalized spacings $\left\{i\cdot\left(\sigma_{(i,n)}^{(x)} - \sigma_{(n,i+1)}^{(x)}\right)\right\}_{i=k,k+1}$ are i.i.d.\ exponential variables with parameter 1.
It follows that
\begin{align}
  \sum_{k\leq i<j\leq k+2} \Pr \Biggl[ &\sigma_{(i,n)}^{(x)} - \sigma_{(j,n)}^{(x)} \leq -\log \left( (1-n^{-\eps})^{2d} - n^{-\eps} \right)\Biggr]\nonumber\\
  &\leq 3\,\Pr \Biggl[ \sigma_{(n,k+1)}^{(x)} - \sigma_{(n,k+2)}^{(x)} \leq -\log \left( (1-n^{-\eps})^{2d} - n^{-\eps} \right)\Biggr]\nonumber\\
  &= 3\left(1-{\rm e}^{(k+1)\log ((1-n^{-\eps})^{2d} - n^{-\eps})}\right) \leq 3(k+1)(2d+1) n^{-\eps}\, ,
  \label{equ:FinalEstimateProbEn}
\end{align}
which converges to zero.

For the second summand on the RHS of \eqref{equ:UpperBoundOne}, we infer since $F$ is increasing, and by virtue of Lemma \ref{lem:DistrChi}, that
\begin{align*}
  \Pr \Biggl[ \binom{(2k+2)^d}{2}&\tfrac{F\left( (1-n^{-\eps}) \chi_{(j,n)}^{(x)}\right)^{4d-1}}{F_\chi \left( \chi_{(j,n)}^{(x)}\right)} >n^{-\eps}\Biggr]
  \leq \Pr \Biggl[ \binom{(2k+2)^d}{2}n^\eps> \tfrac{F_\chi \left( \chi_{(j,n)}^{(x)}\right)}{F\left( \chi_{(j,n)}^{(x)}\right)^{4d-1}} \Biggr]\\
  &\leq \Pr \Biggl[ \binom{(2k+2)^d}{2}\left( n^\eps+1\right)> \tfrac{(2k+2)^d F_\pi \left( \chi_{(j,n)}^{(x)}\right)}{F\left( \chi_{(j,n)}^{(x)}\right)^{4d-1}}\Biggr]\, .
\end{align*}
Since $\pi$ is the sum of $2d$ independent copies of the conductance $w$, we can bound $F_{\pi}(a)\geq F\left( a/(2d) \right)^{2d}$ for all $a\geq0$.
Together with the assumption of Theorem \ref{thm:Loc} this implies that
\begin{align}
  F_{\pi}(a)\geq (2d)^{-2d} F\left( a \right)^{2d}\qquad\text{for all }a\leq a^\ast\, .
  \label{equ:LowerBoundFPi}
\end{align}
and therefore
\begin{align*}
  \Pr \Biggl[ \binom{(2k+2)^d}{2}\left( n^\eps+1\right)\mspace{-3mu}>\mspace{-3mu} \tfrac{(2k+2)^d F_\pi \left( \chi_{(j,n)}^{(x)}\right)}{F\left( \chi_{(j,n)}^{(x)}\right)^{4d-1}}\Biggr]
  \mspace{-5mu}\leq\mspace{-2mu} \Pr \Biggl[ \left( (2k+2)^d -1 \right)n^\eps \mspace{-3mu}>\mspace{-3mu} (2d)^{-2d} F\left( \chi_{(j,n)}^{(x)}\right)^{1-2d} \Biggr]
\end{align*}
where we have furthermore used that $n^\eps\geq 1$ for all $n\in\mathbb{N}$.
Since $F$ is continuous and increasing, it follows that there exists a constant $A<\infty$ such that
\begin{align*}
  \Pr \Biggl[ \left( (2k+2)^d -1 \right)n^\eps > (2d)^{-2d} F\left( \chi_{(j,n)}^{(x)}\right)^{1-2d} \Biggr]
  \leq \Pr \Biggl[ \chi_{(j,n)}^{(x)} > \inf\left\{ b\colon F(b) = A n^{-\frac{\eps}{2d-1}} \right\} \Biggr]\, .
\end{align*}
Let $\beta_n$ be the cardinality of the set $\left\{ y\in ((2k+3)\Z^d + x)\colon \mathscr{N}\circ \tau_y \subset B_{n+2k+1} \right\}$.
Then for any $a\geq 0$ and $j\in\{ k,k+1,k+2 \}$ we know that
\begin{align*}
  \Pr \left[ \chi_{(j,n)}^{(x)} > a \right] \;\;&\leq\;\; \Pr \left[ \chi_{(k+2,n)}^{(x)} > a \right]\\
  &\leq\;\;\Pr \left[ \chi > a \right]^{\beta_n} + \beta_n \Pr \left[ \chi > a \right]^{\beta_n -1} \left( 1 - \Pr \left[ \chi > a \right] \right)\\
  &\qquad + \ldots + \frac{\beta_n !}{\left( \beta_n-k\right) !}\Pr \left[ \chi > a \right]^{\beta_n -k-1} \left( 1 - \Pr \left[ \chi > a \right] \right)^{k+1}\\
  &\leq\;\; (k+2)\beta_n^{k+1} \Pr \left[ \chi > a \right]^{\beta_n -k-1}\, .
\end{align*}
By virtue of Lemma \ref{lem:DistrChi} and \eqref{equ:LowerBoundFPi} we thus obtain for all $0\leq a \leq a^\ast$ that
\begin{align*}
  \Pr \Biggl[ \chi_{(j,n)}^{(x)} > a \Biggr] &\leq (k+2)\beta_n^{k+1} \left( 1 - \left(\frac{k+1}{2d^2}\right)^d F(a)^{2d} + \binom{(2k+2)^d}{2} F(a)^{4d-1} \right)^{\beta_n -k-1}\, .\\
\end{align*}
We now insert $a = \inf\left\{ s\colon F(s) = A n^{-\frac{\eps}{2d-1}} \right\}$ and observe that there exist $C_1, C_2\in (0,\infty)$ such that $C_1 n^{d}+k+1\leq \beta_n \leq C_2 n^d$.
It follows that for $n$ large enough
\begin{align}
  \Pr \Biggl[ \chi_{(j,n)}^{(x)} > F^{-1}\left( A n^{-\frac{\eps}{2d-1}}\right) \Biggr]
  &\leq (k+2)C_2^{k+1} n^{2d} \left( 1 - \left(\frac{A^2(k+1)}{2d^2}\right)^d n^{-\frac{2d\eps}{2d-1}} \right)^{C_1 n^d}\, .
  \label{equ:UpperEstErrorTerm}
\end{align}
Since we have assumed that $\eps<1$ at the beginning of this proof, this converges to zero and is even summable.
This concludes the proof of \eqref{equ:ProofLoc:Limsup}.

Let us now treat the second term on the RHS of \eqref{equ:ProofLoc:Split}.
We split the event
\begin{align*}
  E_m = \left( E_m \cap \left\{ \pi_{2,B_m} < \pi_{2,B_{m-2k-3}} \right\}\right) \cup \left( E_m \cap \left\{ \pi_{2,B_m} = \pi_{2,B_{m-2k-3}} \right\}\right)
\end{align*}
and we observe that $E_m \cap \left\{ \pi_{2,B_m} = \pi_{2,B_{m-2k-3}} \right\} \cap E_{m-2k-3}^{\rm c} = \emptyset$.
Thus
\begin{align*}
  \lim_{n\to\infty}\Pr &\left[ \bigcup_{m= n+2k+3}^\infty \left( E_m \cap G_m \cap E_{m-2k-3}^{\rm c}\right) \right]\\
  &\qquad\qquad\qquad\leq \lim_{n\to\infty}\sum_{m= n+2k+3}^\infty \Pr \left[ E_m \cap G_m \cap \left\{ \pi_{2,B_m} < \pi_{2,B_{m-2k-3}} \right\}\right]\, .
\end{align*}
By Lemma \ref{lem:decompV} we decompose and estimate
\begin{align*}
  \Pr &\left[ E_m \cap G_m \cap \left\{ \pi_{2,B_m} < \pi_{2,B_{m-2k-3}} \right\}\right]\\
  &\qquad\leq \sum_{x\in B_{k+1}} \Pr \left[ E_m \cap G_m \cap \left\{ \pi_{2,B_m} < \pi_{2,B_{m-2k-3}} \right\}\right.\\
  &\qquad\qquad\qquad\qquad\left.\cap\left\{ \{z_{(1,m)},\ldots,z_{(k+1,m)}, z_{(1,m-2k-3)}, \ldots, z_{(k+1,m-2k-3)}\} \subset \mathscr{V} \circ \tau_x \right\}\right]
\end{align*}
Similar to the considerations above, we know that
\begin{align*}
  \pi_{z_{(k+1,m-2k-3)}}\in\left\{ \chi^{(x)}_{(k+1,m-2k-3)}, \chi^{(x)}_{(k+2,m-2k-3)}\right\}\, .
\end{align*}
Therefore we further decompose the above events by
\begin{align*}
  \Omega = \bigcup_{i=k}^{k+1} \bigcup_{j=i+1}^{k+2} \bigcup_{l=k+1}^{k+2} \left\{ \pi_{z_{(k,m)}} = \chi^{(x)}_{(i,m)}\right\} &\cap \left\{ \pi_{z_{(k+1,m)}} = \chi^{(x)}_{(j,m)}\right\}\\
  &\qquad\qquad\cap \left\{ \pi_{z_{(k+1,m-2k-3)}} = \chi^{(x)}_{(l,m-2k-3)} \right\}\, .
\end{align*}
And this we split into the three cases $l<j$, $l=j$ and $l>j$, i.e.,
\begin{align}
  \sum_{x\in B_{k+1}} \sum_{i=k}^{k+1} &\sum_{j=i+1}^{k+2} \sum_{l=k+1}^{k+2} \Pr \left[ \ldots \right]\nonumber\\
  &= \sum_{x\in B_{k+1}} \sum_{i=k}^{k+1} \left( \sum_{j=i+1}^{k+2} \sum_{l=k+1}^{j-1} \Pr \left[ \ldots \right] + \sum_{j=i+1, \atop l=j}^{k+2} \Pr \left[ \ldots \right] + \sum_{j=i+1}^{k+2} \sum_{l=j+1}^{k+2} \Pr \left[ \ldots \right]\right)
  \label{equ:FinalDecomp}
\end{align}
Let us consider the first term on the above RHS.
It has to be $j=k+2$ and $l=k+1$.
Thus it contains
\begin{align*}
  \Pr &\left[ E_m \cap G_m \cap \left\{ \{z_{(1,m)},\dots, z_{(k+1,m)}, z_{1,m-2k-3}, z_{(k+1,m-2k-3)}\} \subset \mathscr{V} \circ \tau_x \right\}\cap\right.\\
  &\qquad\qquad\cap \left\{ \pi_{k+1,B_m} < \pi_{k+1,B_{m-2k-3}} \right\} \cap \left\{ \pi_{z_{(k,m)}} = \chi^{(x)}_{(i,m)}\right\}\\
  &\qquad\qquad\left.\cap\, \left\{ \pi_{z_{(k+1,m)}} = \chi^{(x)}_{(k+2,m)}\right\} \cap \left\{ \pi_{z_{(k+1,m-2k-3)}} = \chi^{(x)}_{(k+1,m-2k-3)} \right\}\right]\\
  &\leq \Pr \left[ \left\{ \pi_{k+1,B_m} < \pi_{k+1,B_{m-2k-3}} \right\} \cap \left\{ \pi_{z_{(k+1,m)}} = \chi^{(x)}_{(k+2,m)}\right\}\right.\\
  &\qquad\qquad\left.\cap\, \left\{ \pi_{z_{(k+1,m-2k-3)}} = \chi^{(x)}_{(k+1,m-2k-3)} \right\}\right]\\
  &\leq \Pr \left[ \chi^{(x)}_{(k+2,m)} < \chi^{(x)}_{(k+1,m-2k-3)} \right]\, .
\end{align*}
If $\chi^{(x)}_{(k+2,m)} < \chi^{(x)}_{(k+1,m-2k-3)}$, then there are at least two sites
\begin{align*}
  y^\ast_1, y^\ast_2\in\left\{ y\in ((2k+3)\Z^d + x) \colon \mathscr{N}\circ\tau_{y} \cap B_{k} \neq \emptyset, \mathscr{N}\circ\tau_{y} \cap B_{m-2k-3} = \emptyset \right\}
\end{align*}
such that $\chi_{y^\ast_1}, \chi_{y^\ast_2} < \chi^{(x)}_{(k+1,m-2k-3)}$.
Since the cubes $\mathscr{N}\circ\tau_{y}$, $y\in ((2k+3)\Z^d + x)$, are disjoint, the probability that this happens is of order $m^{-2}$, see e.g. the considerations above and \cite[Proposition 4.3]{Resnick1987}.
This implies that the sum over $m$ is finite.

Now we consider the second term on the RHS of \eqref{equ:FinalDecomp}.
It contains
\begin{align*}
  \sum_{j=i+1}^{k+2} \Pr &\left[ E_m \cap G_m \cap \left\{ z_{(1,m)},\ldots, z_{(k+1,m)}, z_{(1,m-2k-3)}, z_{(k+1,m-2k-3)}\} \in \mathscr{V} \circ \tau_x \right\}\right.\\
  &\qquad\qquad\cap\, \left\{ \pi_{k+1,B_m} < \pi_{k+1,B_{m-2k-3}} \right\} \cap \left\{ \pi_{z_{k,m}} = \chi^{(x)}_{(i,m)}\right\}\\
  &\qquad\qquad\left.\cap \left\{ \pi_{z_{(k+1,m)}} = \chi^{(x)}_{(j,m)}\right\} \cap \left\{ \pi_{z_{(k+1,m-2k-3)}} = \chi^{(x)}_{(j,m-2k-3)} \right\}\right]\\
  &\leq \sum_{j=i+1}^{k+2} \Pr \left[ \left\{m^{-\eps} > 1 - \frac{\chi^{(x)}_{(i,m)}}{\chi^{(x)}_{(j,m)}}\right\} \cap \left\{ \chi^{(x)}_{(j,m)} < \chi^{(x)}_{(j,m-2k-3)} \right\}\right]\\
  &\leq \sum_{j=i+1}^{k+2} \Pr \left[ \left\{\sigma_{(i,m)}^{(x)} - \sigma_{(j,m)}^{(x)} \leq -\log \left( (1-m^{-\eps})^{2d} - m^{-\eps} \right)\right\} \cap \left\{ \sigma^{(x)}_{(j,m)} > \sigma^{(x)}_{(j,m-2k-3)} \right\}\right]\\
  &\qquad\qquad + \sum_{j=i+1}^{k+2} \Pr \left[ \binom{(2k+2)^d}{2}\tfrac{F\left( (1-m^{-\eps}) \chi_{(j,m)}^{(x)}\right)^{4d-1}}{F_\chi \left( \chi_{(j,m)}^{(x)}\right)} >m^{-\eps} \right]
  \, .
\end{align*}
where we have applied the same considerations as for \eqref{equ:UpperBoundOne}.
The second summand on the above RHS is already summable over $m$ as we have shown in \eqref{equ:UpperEstErrorTerm}.
For the first term we recall that the $\sigma^{(x)}_z$ are independent exponential random variables and by virtue of Lemma \ref{lem:IndepSpacingRelRank} the two events $\left\{\sigma_{(i,m)}^{(x)} - \sigma_{(j,m)}^{(x)} \leq -\log \left( (1-m^{-\eps})^{2d} - m^{-\eps} \right)\right\}$ and
$\left\{\sigma^{(x)}_{(j,m)} > \sigma^{(x)}_{(j,m-2k-3)}\right\}$ are independent.
The probability of the event $\left\{\sigma_{(i,m)}^{(x)} \mspace{-3mu}- \sigma_{(j,m)}^{(x)} \leq -\log \left( (1-m^{-\eps})^{2d} \mspace{-3mu}- m^{-\eps} \right)\right\}$ is of order $m^{-\eps}$, see \eqref{equ:FinalEstimateProbEn}, whereas the probability of the event $\left\{\sigma^{(x)}_{(j,m)} > \sigma^{(x)}_{(j,m-2k-3)}\right\}$ is of order $m^{-1}$, see e.g.\ \cite[Proposition 4.3]{Resnick1987}.
It follows that the second term on the RHS of \eqref{equ:FinalDecomp} is summable over $m$ as well.

Now we consider the third term on the RHS of \eqref{equ:FinalDecomp}.
Here, $i=k$, $j=k+1$ and $l=k+2$.
It follows that there exists a site $z^\ast\in\left(B_{m-2}\backslash B_{m-2k-3}\right)\cap \mathscr{V}\circ\tau_x$ such that $\pi_{z^\ast}<\pi_{z_{(k+1,m-2k-3)}}$.
Thus, the cube $\mathscr{N}\circ \tau_y$ with $y\in(2k+3)\Z^d+x$ that contains this site $z^\ast$, is associated with a $\chi_y$ that is a new record in the sense that $\chi_y=\chi^{(x)}_{(\kappa,m-2k-3)}$ with $\kappa\in\{ k,k+1 \}$ and $\chi^{(x)}_{(\kappa,m-2k-3)}<\chi^{(x)}_{(\kappa,m-4k-6)}$.
Since further $\left\{ \chi^{(x)}_{(k,m-2k-3)}<\chi^{(x)}_{(k,m-4k-6)}\right\} \subset \left\{ \chi^{(x)}_{(k+1,m-2k-3)}<\chi^{(x)}_{(k+1,m-4k-6)}\right\}$, we arrive at
\begin{align*}
  \Pr &\left[ E_m \cap G_m \cap \left\{ z_{(1,m)},\ldots, z_{(k+1,m)}, z_{(1,m-2k-3)}, z_{(k+1,m-2k-3)}\} \in \mathscr{V} \circ \tau_x \right\}\right.\\
  &\qquad\qquad\qquad\qquad\cap\, \left\{ \pi_{k+1,B_m} < \pi_{k+1,B_{m-2k-3}} \right\} \cap \left\{ \pi_{z_{(k,m)}} = \chi^{(x)}_{(k,m)}\right\}\\
  &\qquad\qquad\qquad\qquad\left.\cap\, \left\{ \pi_{z_{(k+1,m)}} = \chi^{(x)}_{(k+1,m)}\right\} \cap \left\{ \pi_{z_{(k+1,m-2k-3)}} = \chi^{(x)}_{(k
  +2,m-2k-3)} \right\}\right]\\
  &\qquad\qquad\leq 2\Pr \left[ \left\{m^{-\eps} > 1 - \frac{\chi^{(x)}_{(k,m)}}{\chi^{(x)}_{(k+1,m)}}\right\}\cap \left\{ \chi^{(x)}_{(k+1,m-2k-3)} < \chi^{(x)}_{(k+1,m-4k-6)} \right\}\right]\\
  &\qquad\qquad\leq 2\Pr \left[ \left\{\sigma_{(k,m)}^{(x)} - \sigma_{(k+1,m)}^{(x)} \leq -\log \left( (1-m^{-\eps})^{2d} - m^{-\eps} \right)\right\}\right.\\
  &\qquad\qquad\qquad\qquad\qquad\qquad\qquad\qquad\qquad\qquad\left.\cap\, \left\{ \sigma^{(x)}_{(k+1,m-2k-3)} > \sigma^{(x)}_{(k+1,m-4k-6)} \right\}\right]\\
  &\qquad\qquad\qquad\qquad + 2\Pr \left[ \binom{(2k+2)^d}{2}\tfrac{F\left( (1-m^{-\eps}) \chi_{(k+1,m)}^{(x)}\right)^{4d-1}}{F_\chi \left( \chi_{(k+1,m)}^{(x)}\right)} >m^{-\eps} \right]\, .
\end{align*}
Both summands are summable over $m$ by the same considerations as above.
We thus conclude the proof.
\end{proof}

\subsection{Proof of Theorem \ref{thm:Loc}}
\label{subsec:ProofLoc}
Setting $k=1$, we make exactly the same observations and definitions as at the beginning of the proof of Lemma \ref{lem:QuotOrder} until \eqref{equ:GnPas} where we have defined the event $G_n$.
We abbreviate $z_n = z_{(1,n)}$.

Now we let $\alpha_n = n^{-\eps_1/8}$ and note that
\begin{align}
  \left\{ \bigl\| \psi_1^{(n)}\bigr\|_{\ell^2(B_n\backslash \{ z_n\})}^2 > \alpha^2_n\right\}
  \subseteq \left\{ \bigl\| \psi_1^{(n)}\bigr\|_{\ell^2(\D_n)}^2 > \frac{\alpha^2_n}{2}\right\} \cup \left\{ \bigl\| \psi_1^{(n)}\bigr\|_{\ell^2(\mathscr{I}_n\backslash \{ z_n\})}^2 > \frac{\alpha^2_n}{2}\right\}\, .
  \label{equ:SplitEvent}
\end{align}
However, by virtue of Lemma \ref{lem:Loc:MassDn} we know that \Pas for $n$ large enough
\begin{align}
  \bigl\| \psi_1^{(n)} \bigr\|^2_{\ell^2 (\D_n)} \leq \alpha^4_n\, ,
  \label{equ:UpperBoundPsiD}
\end{align}
and thus \Pas the limit superior of the first event on the RHS of \eqref{equ:SplitEvent} vanishes.

In order to estimate the probability of the second event on the RHS of \eqref{equ:SplitEvent},
we now estimate $\bigl\| \psi_1^{(n)}\bigr\|_{\ell^2(\mathscr{I}_n\backslash \{ z_n\})}$ in terms of $\bigl\| \psi_1^{(n)}\bigr\|_{\ell^2(\D_n)}$.
By virtue of Remark \ref{rem:PF}, we can assume without loss of generality that $\psi_1^{(n)}$ nonnegative.
Let $y\in\mathscr{I}_n\backslash \{z_n\}$ and define $m_y = 2 \max_{x:x\sim y} \psi_1^{(n)}(x)$.
On the event where $\mathscr{I}_n$ is sparse, $y\nsim z_n$.
Therefore we know by virtue of Lemma \ref{lem:uniquemax} that $\psi_1^{(n)} (y) \leq m_y \left( 1 - \frac{\pi_{z_n}}{\pi_y}\right)^{-1}$.
By definition $\pi_y\geq \pi_{z_{(2,n)}}$ and thus it follows that on the event $G_n$ we have
\begin{align*}
  \bigl\| \psi_1^{(n)}\bigr\|_{\ell^2(\mathscr{I}_n\backslash \{ z_n\})}^2 \leq \left(1 - \frac{\pi_{z_n}}{\pi_{z_{(2,n)}}}\right)^{-2} \sum_{y\in \mathscr{I}_n\backslash \{ z_n\}} m_y^2\, .
\end{align*}
Moreover, on the event where $\mathscr{I}_n$ is sparse, any neighbor of $y\in\mathscr{I}_n$ is in $\D_n$ and therefore
\begin{align}
\bigl\| \psi_1^{(n)}\bigr\|^2_{\ell^2(\mathscr{I}_n\backslash \{ z_n\})} \leq 8d\, \left(1 - \frac{\pi_{z_n}}{\pi_{z_{(2,n)}}}\right)^{-2} \bigl\| \psi_1^{(n)} \bigr\|^2_{\ell^2 (\D_n)}\, .
	   \label{equ:psiInzn}
\end{align}
On the event where \eqref{equ:UpperBoundPsiD} is true and $\mathscr{I}_n$ is sparse, we hence infer that
\begin{align*}
  \left\{ \bigl\| \psi_1^{(n)}\bigr\|_{\ell^2(\mathscr{I}_n\backslash \{ z_n\})}^2 > \frac{\alpha^2_n}{2}\right\} \subseteq
  \left\{  4\sqrt{d}\,\alpha_n >1 - \frac{\pi_{z_n}}{\pi_{z_{(2,n)}}} \right\}\, .
\end{align*}

However, by virtue of Lemma \ref{lem:QuotOrder} we know that \Pas for $n$ large enough $4\sqrt{d}\,\alpha_n < 1 - \frac{\pi_{z_{n}}}{\pi_{z_{(2,n)}}}$.
The claim follows.

\subsection{Asymptotics of principal Dirichlet eigenvalue}
\label{subsec:ProofCorLoc}

\begin{proof}[Proof of Corollary \ref{cor:Loc}]
By virtue of \eqref{equ:TrivialUpperBound}, we already know that $\lambda_1^{(n)} \leq \min_{z\in B_n} \pi_z$.
By \eqref{equ:Psi1zn} we further know that there exists $\eps_1>0$ such that \Pas for $n$ large enough
\begin{align*}
  \psi_1^{(n)} (z_n)^2 \geq 1 - n^{-\eps_1/4}\, .
\end{align*}
It follows that \Pas for $n$ large enough
\begin{align}
  \lambda_1^{(n)} &= \langle \psi_1^{(n)}, \mathcal{L}_{\boldsymbol{w}} \psi_1^{(n)} \rangle \geq \sum_{x\colon x\sim z_n} w_{xz_n} \left( \psi_1^{(n)} (z_n) - \psi_1^{(n)} (x) \right)^2\nonumber\\
  &\geq \left( n^{-\eps_1/8} - \sqrt{1 - n^{-\eps_1/4}} \right)^2 \min_{z\in B_n} \pi_z
\end{align}
The claim \eqref{equ:PrincipalEigValMinPi} follows.
\end{proof}

\begin{proof}[Proof of Corollary \ref{cor:ConvInLaw}]
Because of Corollary \ref{cor:Loc} it remains to prove that
\begin{align*}
   \lim_{n\to\infty} \Pr \left[ \min_{x\in B_n}\pi_x > \frac{\zeta}{n^{\frac{1}{2\gamma}}L^\ast(n)} \right] = \exp\left( -\zeta^{2d\gamma} \right)\qquad \text{for all }\zeta\geq 0\, ,
\end{align*}
where the difficulty lies in the dependence of the random variables $\left( \pi_x \right)_{x\in B_n}$.
However, the dependence is very short-ranged since $\pi_{x_1}$ and $\pi_{x_2}$ are dependent if and only if the sites $x_1, x_2$ are neighbors.
We strongly rely on the ideas of \cite{Watson1954}, which we easily adapt to our needs.
   
In what follows, we always mean that a statement holds for all $\zeta\geq 0$ even if we do not explicitly write so.
We define
\begin{align*}
  a_n \;:=\; \left( n^{\frac{1}{2\gamma}} L^\ast (n)\right)^{-1}\;=\;\frac{1}{h (|B_n|)} = \sup\, \left\{ t\colon F_\pi (t) = |B_n|^{-1} \right\}\, .
\end{align*}
   Then $|B_n| = \left( \Pr \left[ \pi_0 \leq a_n\right] \right)^{-1}$ and thus
   \begin{align}
      \lim_{n\to\infty} |B_n| \Pr \left[ \pi_0 \leq a_n \zeta \right] = \lim_{n\to\infty} \frac{F_\pi \left( a_n \zeta \right)}{F_\pi \left( a_n \right)} \;=\;\zeta^{2d\gamma}
      \label{equ:WatsonCond1}
   \end{align}
   since $a_n\to 0$ as $n\to\infty$ and $F_\pi$ varies regularly at zero with index $2d\gamma$.
   We further note that if $\vec{e}_1\in \Z^d$ is a neighbor of the origin, then $\Pr \left[ \left\{ \pi_{0} \leq a_n \zeta\right\} \cap \left\{ \pi_{\vec{e}_1} \leq a_n \zeta\right\} \right] \leq F(a_n\zeta)^{4d-1}$ since for the event $\left\{ \pi_{0} \leq a_n \zeta\right\} \cap \left\{ \pi_{\vec{e}_1} \leq a_n \zeta\right\}$ at least $4d-1$ independent conductances $w$ have to be smaller than or equal to $a_n \zeta$.
   Since $F$ varies regularly at zero with index $\gamma$, it follows that   
   \begin{align}
      |B_n|\, \Pr \left[ \left\{ \pi_{0} \leq a_n \zeta\right\} \cap \left\{ \pi_{\vec{e}_1} \leq a_n \zeta\right\} \right]  \to 0\qquad\text{as }n\to \infty\, .
      \label{equ:WatsonCond2}
   \end{align}

   Now for any even integer $l\leq |B_n|$ we estimate
   \begin{align}
     1 - &\sum_{m=1}^{l-1}(-1)^{m-1}\sum_{A \subset B_n,\atop |A|=m} \Pr \left[ \bigcap_{x\in A}\left\{ \pi_{x} \leq a_n\zeta\right\} \right]\nonumber\\
     &\mspace{50mu}\leq\; \Pr \left[ \min_{x\in B_n}\pi_x > a_n\zeta\right]
     \leq\; 1 - \sum_{m=1}^{l}(-1)^{m-1}\sum_{A \subset B_n,\atop |A|=m} \Pr \left[ \bigcap_{x\in A}\left\{ \pi_{x} \leq a_n\zeta\right\} \right]
     \label{equ:Series}
   \end{align}
   The first term in the above sums over $m$, i.e.\, $\sum_{x\in B_n}\Pr \left[ \pi_x \leq a_n\zeta \right]$, is equal to $|B_n| \Pr \left[ \pi_0 \leq a_n\zeta \right]$ and converges to $\zeta^{2d\gamma}$.
   For the second term in the sums over $m$ in \eqref{equ:Series} we observe that
   \begin{align}
      &\sum_{\{ x_1,x_2\}\subset B_n,\atop x_1\neq x_2} \Pr \left[ \left\{ \pi_{x_1} \leq a_n\zeta\right\} \cap \left\{ \pi_{x_2} \leq a_n\zeta\right\} \right]\nonumber\\
      &\mspace{25mu}=\mspace{-10mu} \sum_{\{ x_1,x_2\}\subset B_n, \atop x_1\neq x_2,\, x_1\nsim x_2}\mspace{-10mu} \Pr \left[ \left\{ \pi_{x_1} \leq a_n\zeta\right\} \cap \left\{ \pi_{x_2} \leq a_n\zeta\right\} \right]
      + \mspace{-7mu}\sum_{\{ x_1,x_2\} \subset B_n, \atop x_1\sim x_2}\mspace{-3mu} \Pr \left[ \left\{ \pi_{x_1} \leq a_n\zeta\right\} \cap \left\{ \pi_{x_2} \leq a_n\zeta\right\} \right]\, .
      \label{equ:SeriesSecondOrder}
   \end{align}
   Let $C_2^{(n)}$ be the number of combinations of distinct $x_1,x_2\in B_n$ with $x_1\nsim x_2$.
   Since $\pi_{x_1}$ and $\pi_{x_2}$ are independent if $x_1\nsim x_2$, it follows that
   \begin{align*}
      \sum_{\{ x_1,x_2\}\subset B_n, \atop x_1\neq x_2,\, x_1\nsim x_2} \Pr \left[ \left\{ \pi_{x_1} \leq a_n\zeta\right\} \cap \left\{ \pi_{x_2} \leq a_n\zeta\right\} \right]
      \;=\; C_2^{(n)}\, \Pr \left[ \pi_0 \leq a_n\zeta \right]^2\, .
   \end{align*}
   As $n$ grows to infinity, the leading term in $C_2^{(n)}$ is $\frac{1}{2}|B_n|^2$.
   It follows that
   \begin{align*}
      C_2^{(n)} \Pr \left[ \pi_0 \leq a_n\zeta \right]^2 \to \tfrac{1}{2}\,\zeta^{4d\gamma}\, .
   \end{align*}
   The second term on the RHS of \eqref{equ:SeriesSecondOrder}, however, is of order $|B_n| \Pr \left[ \left\{ \pi_{0} \leq a_n\zeta\right\} \cap \left\{ \pi_{\vec{e}_1} \leq a_n\zeta\right\} \right]$, which converges to zero by \eqref{equ:WatsonCond2}.
   
   For each $q\leq |B_n|$ the general sum $\sum_{A \subset B_n,\atop |A|=q} \Pr \left[ \bigcap_{x\in A}\left\{ \pi_{x} \leq a_n\zeta\right\} \right]$ has $\binom{|B_n|}{q}$ terms.
   Following the reasoning of \cite{Watson1954}, we note that there are asymptotically $|B_n|^q/q!$ terms in which each $x_i$ is not adjacent to any of the $x_j$'s.
   The sum over these terms yields asymptotically $\zeta^{2qd\gamma}/q!$.
   Then we have an amount of order $|B_n|^{q-1}$ terms where there exists exactly one neighbored pair $x_i \sim x_j$.
   The sum over these terms yields a constant times
   \begin{align*}
      |B_n|^{q-1} \Pr \left[ \pi_0 \leq a_n\zeta \right]^{q-2} \Pr \left[ \left\{\pi_{0} \leq a_n\zeta\right\} \cap \left\{\pi_{\vec{e}_1} \leq a_n\zeta\right\} \right]\, ,
   \end{align*}
   which converges to zero as $n$ tends to infinity by virtue of \eqref{equ:WatsonCond1} and \eqref{equ:WatsonCond2}.
   Further, we have an amount of order $|B_n|^{q-m+1}$ terms where there exists a dependence between $m\in\N$ of the $\pi_{x_i}$.
   The sum over these terms is smaller than a constant times
   \begin{align*}
      |B_n|^{q-m+1} \Pr \left[ \pi_0 \leq a_n\zeta \right]^{q-m} \Pr \left[ \left\{\pi_{0} \leq a_n\zeta\right\} \cap \left\{\pi_{\vec{e}_1} \leq a_n\zeta\right\} \right]\, ,
   \end{align*}
   which converges to zero as well.
   It follows that for any even integer $l$ we have
   \begin{align*}
      \sum_{q=0}^{l-1} \frac{(-\zeta)^{2qd\gamma}}{q!}\leq \lim_{n\to\infty} \Pr \left[ \min_{x\in B_n}\pi_x > a_n\zeta \right] \leq \sum_{q=0}^{l} \frac{(-\zeta)^{2qd\gamma}}{q!}\, .
   \end{align*}
\end{proof}

\paragraph{Acknowledgments.}
I am grateful to Wolfgang K\"onig, Martin Slowik, Renato Soares dos Santos, Marek Biskup, Christian Hirsch, and Jan Nagel for fruitful discussions and valuable hints, as well as Michiel Renger for useful suggestions. I thank Wolfgang König for making me acquainted with the question of this paper.
Further, I acknowledge the financial support of DFG within the RTG 1845 ``Stochastic Analysis with Applications in Biology, Finance and Physics''.
Finally, I thank an anonymous referee for his or her invaluable support by reading earlier versions of this manuscript in great detail and thus detecting a number of inconsistencies and gaps.

\bibliographystyle{alpha}

\end{document}